\documentclass[11pt]{article}

\usepackage{amsfonts}
\usepackage{amssymb,amsmath,amsthm, enumerate}
\usepackage{latexsym}
\usepackage{fullpage}
\usepackage{hyperref}

\newtheorem{theorem}{Theorem}[section]

\newtheorem{prop}[theorem]{Proposition}
\newtheorem{question}{Question}[section]
\newtheorem{conjecture}[question]{Conjecture}
\newtheorem{lemma}[theorem]{Lemma}
\newtheorem{cor}[theorem]{Corollary}
\newtheorem{defn}[theorem]{Definition}

\theoremstyle{definition}
\newtheorem{remark}{Remark}

\newcounter{tenumerate}

\def\P{\mathbb{P}}

\newcommand{\one}{\1}
\newcommand{\deq}{\stackrel{\scriptscriptstyle\triangle}{=}}

\newcommand{\tr}{\mathrm{tr}}
\newcommand{\reff}{R_{\mathrm{eff}}}
\newcommand{\ceff}{C_{\mathrm{eff}}}
\newcommand{\spn}{\mathrm{span}}
\newcommand{\proj}{\mathrm{proj}}
\newcommand{\mc}{\mathcal}

\newcommand{\ddeg}{\deg_{\downarrow}}

\renewcommand{\epsilon}{\varepsilon}
\newcommand{\e}{\varepsilon}

\newcommand{\1}{\mathbf{1}}

\newcommand{\rc}{\circlearrowright}
\newcommand{\covreturn}{{t_{\mathrm{cov}}^{\circlearrowright}}}

\DeclareMathOperator{\var}{Var}

\newcommand{\dist}{\mathsf{dist}}

\newcommand{\mcT}{\mathcal T}

\newcommand{\E}{{\mathbb E}}
\newcommand{\remove}[1]{}
\renewcommand{\le}{\leqslant}

\renewcommand{\leq}{\leqslant}
\renewcommand{\geq}{\geqslant}
\newcommand{\diam}{\mathsf{diam}}
\newcommand{\f}{\varphi}

\newcommand{\CH}{\mathsf{aff}}

\newcommand{\val}{\mathsf{val}}
\newcommand{\C}{\mathcal{C}}

\def\XXint#1#2#3{{\setbox0=\hbox{$#1{#2#3}{\int}$}
\vcenter{\hbox{$#2#3$}}\kern-.5\wd0}}

\begin{document}

\title{{\bf Cover times, blanket times, and majorizing measures}}

\author{Jian Ding\footnote{A substantial portion of this work
was completed during visits of the author to Microsoft Research.} \\ U. C. Berkeley \and James R. Lee\thanks{Partially
supported by NSF grant CCF-0915251 and a Sloan Research Fellowship.} \footnotemark[1] \\ University of Washington \and Yuval Peres \\ Microsoft Research}
\date{}

\maketitle

\begin{abstract}
We exhibit a strong connection between cover times of graphs,
Gaussian processes, and Talagrand's theory of majorizing measures.
In particular, we show that the cover time of any graph $G$
is equivalent, up to universal constants,
to the square of the expected maximum of the Gaussian free field on $G$,
scaled by the number of edges in $G$.

This allows us
to resolve a number of open questions.
We give a deterministic polynomial-time algorithm
that computes the cover time to within an $O(1)$ factor
for any graph, answering a question of Aldous and Fill (1994).
We also positively resolve the blanket time conjectures
of Winkler and Zuckerman (1996), showing that for any graph,
the blanket and cover times are within an $O(1)$ factor.
The best previous approximation factor for
both these problems
was $O((\log \log n)^2)$ for $n$-vertex graphs,
due to Kahn, Kim, Lov\'asz, and Vu (2000).
\end{abstract}

\setcounter{tocdepth}{2} \tableofcontents
\section{Introduction}
\label{sec:intro}

Let $G=(V,E)$ be a finite, connected graph,
and consider the simple random walk on $G$.
Writing $\tau_{\mathrm{cov}}$ for the first time
at which every vertex of $G$ has been visited,
let $\E_v \tau_{\mathrm{cov}}$ denote
the expectation of this quantity when
the random walk is started at some vertex $v \in V$.
The following fundamental parameter is known as the
{\em cover time of $G$,}
\begin{equation}\label{eq:covtimedef}
t_{\mathrm{cov}}(G) = \max_{v \in V} \E_v \tau_{\mathrm{cov}}\,.
\end{equation}
We refer to the books \cite{AF, LPW09} and the survey
\cite{Lov96} for relevant background material.

We also recall the discrete {\em Gaussian free field} (GFF) on the graph $G$.
This is a centered Gaussian process $\{\eta_v\}_{v \in V}$ with $\eta_{v_0}=0$ for
some fixed $v_0 \in V$.  The process is characterized by the relation
$\mathbb E\, (\eta_u-\eta_v)^2 = \reff(u,v)$ for all $u,v \in V$,
where $\reff$ denotes the effective resistance on $G$.
Equivalently, the covariances $\E(\eta_u \eta_v)$ are given by the Green kernel
of the random walk killed at $v_0$.
(We refer to Sections \ref{sec:prelims} and \ref{sec:outline} for background
on electrical networks and Gaussian processes.)

The next theorem
represents one of the primary connections put forward in this work.
We use the notation $\asymp$ to denote equivalence up to a universal constant factor.

\begin{theorem}\label{thm:GFF}
For any finite, connected graph $G=(V,E)$, we have
$$
t_{\mathrm{cov}}(G) \asymp |E| \left(\E \max_{v \in V} \eta_v\right)^2,
$$
where $\{\eta_v\}_{v \in V}$ is the Gaussian free field on $G$.
\end{theorem}

The utility of such a characterization
will become clear soon.
Despite being an intensively studied parameter of graphs,
a number of basic questions involving the cover time have remained open.
 We now highlight two of these,
whose resolution we discuss subsequently.

\medskip
\noindent {\bf The blanket time.} For a node $v \in V$, let $\pi(v)
= \frac{\deg(v)}{2|E|}$ denote the stationary measure of the random
walk, and let $N_v(t)$ be a random variable denoting the number of
times the random walk has visited $v$ up to time $t$. Now define
$\tau_{\mathrm{bl}}^{\circ}(\delta)$ to be the first time $t \geq 1$
at which
\begin{equation}\label{eq:visits}
N_v(t) \geq \delta t\, \pi(v)
\end{equation}
holds for all $v \in V$.
In other words, $\tau_{\mathrm{bl}}^{\circ}(\delta)$ is the first time
at which all nodes have been visited at least
a $\delta$ fraction
as much as we expect at stationarity.
Using the same notation
as in \eqref{eq:covtimedef}, define the {\em $\delta$-blanket time} as
\begin{equation}\label{eq:blanketdef}
t_{\mathrm{bl}}^{\circ}(G,\delta) = \max_{v \in V} \E_v \tau_{\mathrm{bl}}^{\circ}(\delta)\,.
\end{equation}
Clearly for $\delta \in (0,1)$, we have $t_{\mathrm{bl}}^{\circ}(G,\delta) \geq t_{\mathrm{cov}}(G)$.
Winkler and Zuckerman \cite{WZ96} made the following conjecture.

\begin{conjecture}\label{con:WZ}
For every $0 < \delta < 1$, there exists a $C$ such that
for every graph $G$, one has
$$t_{\mathrm{bl}}^{\circ}(G,\delta) \leq C \cdot t_{\mathrm{cov}}(G).$$
In other words, for every fixed $\delta \in (0,1)$, one has $t_{\mathrm{cov}}(G) \asymp t_{\mathrm{bl}}^{\circ}(G,\delta)$.
\end{conjecture}

Kahn, Kim, Lov\'asz, and Vu \cite{KKLV00} showed that for
every fixed $\delta \in (0,1)$,
one can take $C \asymp (\log \log n)^2$
for $n$-node graphs, but whether there is
a universal constant, independent of $n$, remained open
for every value of $\delta > 0$.

In order to bound $t_{\mathrm{bl}}^{\circ}(G,\delta)$, we introduce the following stronger notion.
Let $\tau_{\mathrm{bl}}(\delta)$ be the first time $t \geq 1$
such that for every $u,v \in V$, we have
$$
\frac{N_u(t)/\pi(u)}{N_v(t)/\pi(v)} \geq \delta,
$$
i.e. the first time at which all the values $\{N_u(t)/\pi(u)\}_{u
\in V}$ are within a factor of
 $\delta$.
As in \cite{KKLV00}, we define the {\em strong $\delta$-blanket time} as
$$
t_{\mathrm{bl}}(G,\delta) = \max_{v \in V} \E_v \tau_{\mathrm{bl}}(\delta).
$$
Clearly one has $t_{\mathrm{bl}}^{\circ}(G,\delta) \leq t_{\mathrm{bl}}(G,\delta)$
for every $\delta \in (0,1)$.

\medskip

The second question we highlight is computational in nature.

\begin{question}[\cite{AF,KKLV00}]
\label{ques:AF}
Is there a deterministic, polynomial-time algorithm that approximates
$t_{\mathrm{cov}}(G)$ within a constant factor?
\end{question}

In other words, is there a quantity $A(G)$ which can be computed
deterministically, in polynomial-time in $|V|$,
such that $A(G) \asymp t_{\mathrm{cov}}(G)$.  It is crucial that one asks for
a deterministic procedure, since a randomized algorithm
can simply simulate the chain, and output the empirical
mean of the observed times at which the graph is first covered.
This is guaranteed
to produce an accurate estimate with high-probability
in polynomial time, since
the mean and standard deviation of $\tau_{\mathrm{cov}}$
are $O(|V|^3)$
\cite{AKLLR79}.

A result of Matthews \cite{Matthews88} can be used to
produce a determinisically computable bound which is within
a $\log |V|$ factor of $t_{\mathrm{cov}}(G)$.
Subsequently, \cite{KKLV00}
showed how one could compute a bound which lies
within an $O((\log \log |V|)^2)$ factor
of the cover time.

\medskip

Before we state our main theorem and resolve the preceding questions,
we briefly review the $\gamma_2$ functional
from Talagrand's theory of majorizing measures \cite{Talagrand87,Talagrand96}.

\medskip
\noindent
{\bf Majorizing measures and Gaussian processes.}
Consider a compact
metric space $(X,d)$.  Let $M_0 = 1$ and $M_k = 2^{2^k}$ for $k \geq
1$. For a partition $P$ of $X$ and an element $x \in X$, we will
write $P(x)$ for the unique $S \in P$ containing $x$. An {\em
admissible sequence} $\{A_k\}_{k \geq 0}$ of partitions of $X$ is
such that $A_{k+1}$ is a refinement of $A_k$ for $k \geq 0$, and
$|A_k| \leq M_k$ for all $n \geq 0$.  Talagrand defines the
functional
\begin{equation}\label{eq:gamma2}
\gamma_2(X,d) = \inf \sup_{x \in X} \sum_{k \geq 0} 2^{k/2} \diam(A_k(x)),
\end{equation}
where the infimum is over all admissible sequences $\{A_k\}$.

Consider now a Gaussian process $\{\eta_i\}_{i \in I}$
over some index set $I$.
This is a stochastic process such that every finite
linear combination of random variables is normally distributed.
For the purposes of the present paper,
one may assume that $I$ is finite.
We will assume that all Gaussian processes
are centered,
i.e. $\E (\eta_{i}) = 0$ for all $i \in I$.
The
index set $I$ carries a natural metric which
assigns, for $i,j \in I$,
\begin{equation}\label{eq:gaussmetric}
d(i,j) = \sqrt{\E\,|\eta_i-\eta_j|^2}\,.
\end{equation}
The following result
constitutes a primary consequence of the
majorizing measures theory.

\medskip
\noindent
{\bf Theorem (MM)} (Majorizing measures theorem \cite{Talagrand87}).
For any centered Gaussian process $\{\eta_i\}_{i \in I}$,
$$
\gamma_2(I,d) \asymp \E \sup \left\{ \eta_i : i \in I \right\}.
$$

We remark that the upper bound of the preceding theorem,
i.e. $\E \sup \left\{ \eta_i : i \in I \right\} \leq C \gamma_2(I,d)$
for some constant $C$, goes back to work of Fernique \cite{F1,F2}.
Fernique formulated this result in the language of measures (from
whence the name ``majorizing measures'' arises), while
the formulation of $\gamma_2$ given in \eqref{eq:gamma2} is
due to Talagrand.  The fact that the two notions
are related is non-trivial;
we refer to \cite[\S 2]{Talagrand96}
for a thorough discussion of the connection
between them.

\medskip
\noindent {\bf Commute times, hitting times, and cover times.} In
order to relate the majorizing measure theory to cover times of
graphs, we recall the following natural metric. For any two nodes
$u,v \in V$, use $H(u,v)$ to denote the {\em expected hitting time}
from $u$ to $v$, i.e. the expected time for a random walk started at
$u$ to hit $v$. The {\em expected commute time} between two nodes
$u,v \in V$ is then defined by
\begin{equation}\label{eq:commutetime}
\kappa(u,v) = H(u,v) + H(v,u).
\end{equation}
It is immediate that $\kappa(u,v)$ is a metric on any finite,
connected graph. A well-known fact \cite{CRRST96} is that
$\kappa(u,v) = 2 |E| \,\reff(u,v)$, where $\reff(u,v)$ is the {\em
effective resistance} between $u$ and $v$, when $G$ is considered as
an electrical network with unit conductances on the edges. We now
restate our main result in terms of majorizing measures. For a
metric $d$, we write $\sqrt{d}$ for the distance $\sqrt{d}(u,v) =
\sqrt{d(u,v)}$.

\begin{theorem}[Cover times, blanket times, and majorizing measures]\label{thm:mainsimple}
For any graph $G=(V,E)$ and any $0 < \delta < 1$, we have
$$
t_{\mathrm{cov}}(G) \asymp \left[\gamma_2(V, \sqrt{\kappa})\right]^2 =
|E| \cdot \left[\gamma_2(V,\sqrt{\reff})\right]^2 \asymp_{\delta} t_{\mathrm{bl}}(G,\delta),
$$
where $\asymp_{\delta}$ denotes equivalence up to a constant
depending on $\delta$.
\end{theorem}

Clearly this
yields a positive resolution to Conjecture \ref{con:WZ}.
Moreover, we prove the preceding theorem in the setting
of general finite-state reversible Markov chains.
See Theorem \ref{thm:mainthm} for a statement
of our most general theorem.

\medskip

We now address some additional consequences of the main theorem.
First, observe that by combining Theorem \ref{thm:mainsimple} with Theorem (MM),
we obtain Theorem \ref{thm:GFF}.

\begin{theorem}[Cover times and the Gaussian free field]
For any graph $G=(V,E)$ and any $0 < \delta < 1$, we have
$$
t_{\mathrm{cov}}(G) \asymp |E| \left(\E \max_{v \in V} \eta_v\right)^2 \asymp_{\delta} t_{\mathrm{bl}}(G,\delta),
$$
where $\{\eta_v\}$ is the Gaussian free field on $G$.
\end{theorem}

In fact, in Section \ref{sec:gffstrong}, we exhibit the following strong asymptotic upper bound.

\begin{theorem}
For every graph $G=(V,E)$, if $t_{\mathrm{hit}}(G)$ denotes the maximal hitting time in $G$, and $\{\eta_v\}_{v \in V}$
is the Gaussian free field on $G$, then
\[ t_{\mathrm{cov}}(G) \leq \left(1+ C\sqrt{\frac{t_{\mathrm{hit}}(G)}{t_{\mathrm{cov}}(G)}}\,\right) \cdot |E| \cdot \left(\E \sup_{v\in V} \eta_v\right)^2\,,\]
where $C > 0$ is a universal constant.
\end{theorem}

In Section \ref{sec:mm}, we prove the following theorem
which, in conjunction with Theorem \ref{thm:mainsimple},
resolves Question \ref{ques:AF}.

\begin{theorem}
\label{thm:algorithm}
Let $(X,d)$ be a finite metric space, with $n=|X|$.
If, for any two points $x,y \in X$, one can deterministically compute
$d(x,y)$ in time polynomial in $n$, then one can deterministically compute
a number
$A(X,d)$ in polynomial time, for which
$$
A(X,d) \asymp \gamma_2(X,d).
$$
\end{theorem}

\remove{
Turning to Conjecture \ref{con:WZ}, we note
that our proof
of Theorem \ref{thm:mainsimple}
relies on the following theorem.

\begin{theorem}\label{thm:blanketsimple}
There exists a constant $C$
such that for any graph $G=(V,E)$,
$$
t_{\mathrm{bl}}(G) \leq C \cdot \left[\gamma_2(V,\sqrt{\kappa})\right]^2.
$$
\end{theorem}

Together, Theorem \ref{thm:mainsimple} and Theorem \ref{thm:blanketsimple}

\medskip}

A ``comparison theorem''
follows immediately from Theorem \ref{thm:mainsimple},
and the fact that $\gamma_2(X,d) \leq L \gamma_2(X,d')$ whenever $d \leq L d'$
(see \eqref{eq:gamma2}).

\begin{theorem}[Comparison theorem for cover times]
Suppose $G$ and $G'$ are two graphs on the same set of nodes $V$,
and $\kappa_G$ and $\kappa_{G'}$ are the distances induced by
respective commute times. If there exists a number $L \geq 1$ such
that $\kappa_G(u,v) \leq L \cdot \kappa_{G'}(u,v)$ for all $u,v \in
V$, then
$$
t_{\mathrm{cov}}(G) \leq O(L) \cdot t_{\mathrm{cov}}(G')\,.
$$
\end{theorem}

Finally, our work implies that there is an extremely
simple randomized algorithm for computing the cover time
of a graph, up to constant factors.
To this end, consider a graph $G=(V,E)$ whose vertex
set we take to be $V = \{1,2,\ldots,n\}$.  Let $D$
be the diagonal degree matrix, i.e. such that $D_{ii}=\deg(i)$
and $D_{ij}=0$ for $i \neq j$, and let $A$ be the adjacency
matrix of $G$.  We define the following normalized Laplacian,
$$
L_G = \frac{D-A}{\tr(D)}\,.
$$
Let $L_G^+$ denote the Moore-Penrose peudoinverse of $L_G$.
Note that both $L_G$ and $L_G^+$ are positive semi-definite.
We have the following characterization.

\begin{theorem}\label{thm:pseudo}
For any connected graph $G$, it holds that
$$
t_{\mathrm{cov}}(G) \asymp \E \,\left\|\sqrt{L_G^+}\, g\right\|^2_{\infty},
$$
where $g=(g_1, \ldots, g_n)$ is an $n$-dimensional Gaussian,
i.e. such that $\{g_i\}$ are i.i.d. N(0,1) random variables.
\end{theorem}

The preceding theorem yields an $O(n^{\omega})$-time randomized algorithm
for approximating $t_{\mathrm{cov}}(G)$, where $\omega \in [2,2.376)$
is the best-possible exponent for matrix multiplication \cite{CW90}.
Using the linear-system solvers of Spielman and Teng \cite{ST06}
(see also \cite{SpielmanICM}), along with ideas from
Spielman and Srivistava \cite{SS08},
we present an algorithm that runs in near-linear time
in the number of edges of $G$.

\begin{theorem}[Near-linear time randomized algorithm]
\label{thm:randalg}
There is a randomized algorithm which, given an $m$-edge connected graph $G=(V,E)$,
runs in time $O(m (\log m)^{O(1)})$ and outputs
a number $A(G)$ such that $t_{\mathrm{cov}}(G) \asymp \mathbb E\left[A(G)\right] \asymp (\E\left[A(G)^2\right])^{1/2}$.
\end{theorem}

\subsection{Related work}

Cover times of finite graphs have been studied for over 30 years. We
refer to \cite{AF,Lov96,LPW09} for the basic theory. Works of Feige
showed that the cover time for any $n$-node graph is at least
$(1-o(1)) n \log n$ \cite{Feige95a}, and at most $4n^3/27$
\cite{Feige95b}. Both of these bounds are asymptotically tight, with
the tight example for the lower bound given by the complete graph on
$n$ nodes.  

The connection between cover times, commute times,
and the theory of electrical networks was
laid out in \cite{CRRST96}.
In general, the electrical viewpoint provides a powerful methodology
for analyzing random walks (see, for example, \cite{DS84,Tetali91,LP}).
Indeed, this point of view will be central
to the present work.

A fundamental bound of Matthews \cite{Matthews88} shows that
$$t_{\mathrm{cov}}(G) \leq \left(\max_{u,v \in V} H(u,v)\right) (1+\log n)\,,$$
where we recall that $H(u,v)$ is the expected hitting time from $u$ to $v$.
Using the straightforward lower bound $t_{\mathrm{cov}}(G) \geq \max_{u,v \in V} H(u,v)$,
this fact provides a deterministic $O(\log n)$-approximation to $t_{\mathrm{cov}}(G)$ in
$n$-node graphs.

Matthews also proved the lower bound,
\begin{equation}\label{eq:kklv2}
t_{\mathrm{cov}}(G) \geq \max_{S \subseteq V} \left(\min_{u \neq v \in S} H(u,v)\right) \log (|S| -1).
\end{equation}
In \cite{KKLV00}, it is shown that
taking the maximum of the lower bound in \eqref{eq:kklv2}
and the maximal hitting time $\max_{u,v \in V} H(u,v)$
is an $O((\log \log n)^2)$-approximation for $t_{\mathrm{cov}}$. Recently, Feige and Zeitouni \cite{FZ09} have
shown that on trees, one can obtain a very strong bound: For every
$\e > 0$, there is a $(1+\e)$-approximation obtainable by a
deterministic, polynomial-time algorithm.

The cover time has also been studied for many
 specific families of graphs.
 Kahn, Linial, Nisan, and Saks \cite{KLNS89} established an $O(n^2)$ upper bound for regular graphs.
 Broder and Karlin \cite{BK89} proved that the cover time of constant-degree expander graphs is $O(n\log n)$. For planar graphs
 of maximum degree $d$,  Jonasson and Schramm \cite{JS00} showed that the cover time is at least $c_d\, n (\log n)^2$ and at most $6n^2$.
 The order of the cover time on lattices was determined by Aldous \cite{Aldous89} and Zuckerman \cite{Zuckerman92}.
 The latter paper also calculated the order of the cover time
 on regular trees.

Furthermore, for a few families of specific examples, the asymptotics
of the cover time have been calculated more precisely.
These include the work of Aldous \cite{Aldous91} for regular trees,
Dembo, Peres, Rosen, and Zeitouni \cite{DPRZ04} for the 2-dimensional discrete torus, and
Cooper and Frieze \cite{CF08} for the giant component
of various random graphs.

Finally, we remark on an upper bound of Barlow, Ding, Nachmias, and Peres \cite{BDNP09}
which was part of the motivation for the present work.
Consider a connected graph $G=(V,E)$ and the metric space $(V,\kappa)$,
where we recall the commute distance from \eqref{eq:commutetime}.
For each $h \in \mathbb Z$, let $A_h \subseteq V$ be
a set of minimal size whose $2^{h}$-neighborhood (in the metric $\kappa$) covers $V$.
Then,
\begin{equation}\label{eq:bdnpbound}
t_{\mathrm{cov}}(G) \leq O(1) \cdot \left(\sum_{h \in \mathbb Z} 2^{h/2} \sqrt{\log |A_h|}\right)^2\,.
\end{equation}
It turns out that this upper bound is tight (up to a universal constant)
for a number of concrete examples with approximately
``homogeneous'' geometry (we refer to \cite{BDNP09}
for examples, mostly related to various
random graphs arising from percolation).
For instance,
the results of the present paper imply that the
right-hand side of \eqref{eq:bdnpbound}
is equivalent to $t_{\mathrm{cov}}(G)$
for any vertex-transitive graph $G$.
Furthermore, the formula \eqref{eq:bdnpbound} resembles the appearance of the Dudley integral \cite{Dudley67}, which gives a tight bound for Gaussian processes with stationary increments.
This suggests, in particular, a connection between the cover time of graphs
and majorizing measures.

\subsection{Preliminaries}
\label{sec:prelims}

To begin, we introduce some fundamental
notions from random walks and electrical networks.

\medskip
\noindent
{\bf Electrical networks and random walks.}
A {\em network} is a finite, undirected graph $G=(V,E)$, together
with a set of non-negative conductances $\{c_{xy} : x,y \in V\}$
supported exactly on the edges of $G$, i.e. $c_{xy} > 0 \iff xy \in E$.  The conductances are symmetric
so that $c_{xy} = c_{yx}$ for all $x,y \in V$.
We will write $c_x = \sum_{y \in V} c_{xy}$ and
$\mathcal C = \sum_{x \in V} c_x$ for the {\em total conductance.}
We will often use the notation $G(V)$ for a network
on the vertex set $V$.  In this case, the associated conductances
are implicit.  In the few cases when there are
multiple networks under consideration simultaneously,
we will use the notation $c_{xy}^G$ to refer
to the conductances in $G$.

Associated to such a network is the canonical {\em discrete time random walk on $G$},
whose transition probabilities are given by $p_{xy} = c_{xy}/c_x$ for all $x,y \in V$.
It is easy to see that this defines the transition matrix of a reversible
Markov chain on $V$, and that every finite-state reversible Markov chain arises in this way
(see \cite[\S 3.2]{AF}).  The stationary measure of a vertex is precisely $\pi(x) = c_x/\mathcal C$.

Associated to such an electrical network are the classical quantities $\ceff, \reff : V \times V \to \mathbb R_{\geq 0}$
which are referred to, respectively, as the {\em effective conductance} and {\em effective resistance}
between pairs of nodes.  We refer to \cite[Ch. 9]{LPW09} for a discussion of the connection
between electrical networks and the corresponding random walk.
For now, it is useful to keep in mind the following fact \cite{CRRST96}:  For any $x,y \in V$,
\begin{equation}\label{eq:reffkappa}
\reff(x,y) = \frac{\kappa(x,y)}{\mathcal C},
\end{equation}
where the commute time $\kappa$ is defined as before \eqref{eq:commutetime}.

For convenience, we will work exclusively with {\em continuous-time}
Markov chains, where the transition rates between nodes are given by
the probabilities $p_{xy}$ from the discrete chain. One way to
realize the continuous-time chain is by making jumps according to
the discrete-time chain, where the times spent between jumps are
i.i.d.\ exponential random variables with mean 1. We refer to these
random variables as the {\em holding times.} See \cite[Ch. 2]{AF}
for background and relevant definitions.

\medskip
\noindent
{\bf Cover times, local times, and blanket times.}
We will now define various stopping times for the continuous-time random walk.
First, we observe that if $\tau^{\star}_{\mathrm{cov}}$ is the first time at which
the continuous-time random walk has visited every node of $G$,
then for every vertex $v$,
$$
\E_v \tau^{\star}_{\mathrm{cov}} = \E_v \tau_{\mathrm{cov}}\,,
$$
where we recall that the latter quantity refers to the discrete-time chain.
Thus we may also define
the cover time with respect to the continuous-time chain,
i.e. $t_{\mathrm{cov}}(G) = \max_{v \in V} \E_v \tau^{\star}_{\mathrm{cov}}$.

In fact, it will be far more convenient to work with the
{\em cover and return time} defined as follows.
Let $\{X_t\}_{t \in [0,\infty)}$ be the continuous-time chain, and define
\begin{equation}\label{eq:covandreturn}
\tau_{\mathrm{cov}}^{\circlearrowright} = \inf \left\{ t > \tau^{\star}_{\mathrm{cov}} : X_t = X_0 \right\}.
\end{equation}
For concreteness, we define the {\em cover and return time of $G$}
as
$$t_{\mathrm{cov}}^{\circlearrowright}(G) = \max_{v \in V} \E_v \tau_{\mathrm{cov}}^{\circlearrowright}\,,$$
but the following fact shows that the choice of initial vertex is not of great importance for us
(see \cite[Ch. 5, Lem. 25]{AF}),
\begin{equation}\label{eq:covrelations}
\frac12 \covreturn(G) \leq t_{\mathrm{cov}}(G) \leq \covreturn (G) \leq 3\min_{v \in V}\E_v \tau_{\mathrm{cov}}^{\circlearrowright}.
\end{equation}

\medskip

For a vertex $v \in V$ and time $t$,
we define the {\em local time $L_t^v$} by
\begin{equation}\label{eq:localtimedef}
L_t^v = \frac{1}{c_v} \int_0^t \1_{\{X_s = v\}} ds\,,
\end{equation}
where we recall that
$c_v = \sum_{u \in V} c_{uv}$.
For $\delta \in (0,1)$, we define
$\tau_{\mathrm{bl}}^\star(\delta)$ as the first time $t > 0$
at which $$\min_{u,v \in V} \frac{L_t^u}{L_t^v} \geq \delta.$$
Furthermore,
the {\em continuous-time strong $\delta$-blanket time} is defined to be
\begin{equation}\label{eq:expblanketreturn}
t_{\mathrm{bl}}^{\star}(G,\delta) = \max_{v \in V} \E_v \tau_{\mathrm{bl}}^{\star}(\delta).
\end{equation}

\medskip
\noindent
{\bf Asymptotic notation.}
For expressions $A$ and $B$, we will use the notation $A \lesssim B$
to denote that $A \leq C \cdot B$ for some constant $C > 0$.
If we wish to stress that the constant $C$ depends on some parameter,
e.g. $C=C(p)$, we will use the notation $A \lesssim_p B$.
We use $A \asymp B$ to denote the conjunction
$A \lesssim B$ and $B \lesssim A$, and we use
the notation $A \asymp_p B$ similarly.

\remove{
The reader should take the next theorem as guidance;
its proof
\begin{theorem}
For any network $G$ and any $\delta > 0$,
$$t^{\mathrm{disc}}_{\mathrm{bl}}(G,\delta) \asymp t_{\mathrm{bl}}(G,\delta).$$
\end{theorem}
}

\subsection{Outline}
\label{sec:outline}

We first state our main theorem in full generality.
We use only the language of effective resistances,
since this is most natural in the context to follow.

\begin{theorem}
\label{thm:mainthm}
For any network $G=(V,E)$ and any $0 < \delta < 1$,
$$
t_{\mathrm{cov}}(G) \asymp \mathcal C\left[\gamma_2(V,\sqrt{\reff})\right]^2 \asymp_{\delta} t_{\mathrm{bl}}(G,\delta) \asymp_\delta t_{\mathrm{bl}}^\star(G, \delta),
$$
where $\mathcal C$ is the total conductance of $G$.
\end{theorem}

We now present an overview of our main arguments,
and layout the organization of the paper.

\medskip
\noindent {\bf Hints of a connection.} First, it may help the reader
to have some intuition about why cover times should be connected to
the Gaussian processes and particularly the theory of majorizing
measures.

A first hint goes back to work of Aldous
\cite{Aldous82}, where it is shown that the hitting times of
Markov chains are approximately distributed as exponential
random variables.  It is well-known that an exponential variable can be
represented as the sum of the squares of two Gaussians.
Observing that the cover time is just the maximum of all the hitting
times, one might hope that the cover time can be related to the
maximum of a family of Gaussians.

This point of view is strengthened by some quantitative
similarities. Let $\{\eta_i\}_{i \in I}$ be a centered Gaussian
process, and let $d(i,j)$ be the natural metric on $I$ from
\eqref{eq:gaussmetric}. The following two lemmas are central to the
proof of the majorizing measures theorem (Theorem (MM)). We refer to
\cite{LT91}
\cite{Talagrand96} for their utility in the majorizing measures
theory.
The next lemma follows directly from the definition of the
Gaussian density; see, for instance, \cite[Lem. 5.1.3, Eq. (5.18)]{MR06}.

\begin{lemma}[Gaussian concentration]
For every $i,j \in I$, and $\alpha > 0$,
$$
\P\left(\eta_i - \eta_j > \alpha\right) \leq \exp\left(\frac{-\alpha^2}{2\, d(i,j)^2}\right).
$$
\end{lemma}

The next result can be found in \cite[Thm. 3.18]{LT91}.

\begin{lemma}[Sudakov minoration]
\label{lem:sudakov}
For every $\alpha > 0$,
If $I' \subseteq I$ is such that $i,j \in I'$ and $i \neq j$
implies $d(i,j) \geq \alpha$, then
$$
\E \sup_{i \in I'} \eta_i \gtrsim \alpha \sqrt{\log |I'|}.
$$
\end{lemma}

Now, let $G=(V,E)$ be a network, and consider the associated
continuous-time random walk $\{X_t\}$ with local times $L_t^v$. We
define also the {\em inverse local times} $\tau^v(t) = \inf\{s:
L_s^{v} > t\}$. An analog of the following lemma was proved in
\cite{KKLV00} for the discrete-time chain; the continuous-time
version can be similarly proved, though we will not do so here, as
it will not be used in the arguments to come. In interpreting
the next lemma, it helps to recall that $L^u_{\tau^u(t)} = t$.

\begin{lemma}[Concentration for local times]
For all $u,v \in V$ and any $\alpha > 0$ and $t \geq 0$, we have
$$
\P_u\left(L_{\tau^u(t)}^u - L_{\tau^u(t)}^v \geq \alpha\right) \leq \exp\left(\frac{-\alpha^2}{4 t \reff(u,v)}\right),
$$
where $\P_u$ denotes the measure for the random
walk started at $u$.
\end{lemma}

Thus local times satisfy sub-gaussian concentration, where now the distance $d$
is replaced by $\sqrt{t \cdot \reff}$.
On the other side, the classical
bound of Matthews \cite{Matthews88} provides
an analog to Lemma~\ref{lem:sudakov}.

\begin{lemma}[Matthews bound]
For every $\alpha > 0$,
if $V' \subseteq V$ is such that $u,v \in V'$ and $u \neq v$ implies $H(u,v) \geq \alpha$, then
$$
t_{\mathrm{cov}}(G) \geq \alpha \log (|V'| - 1).
$$
\end{lemma}

Of course the similar structure of these lemmas
offers no formal connection, but merely a hint
that something deeper may be happening.
We now discuss a far more concrete connection between local times
and Gaussian processes.

\medskip
\noindent
{\bf The isomorphism theorems.}
The distribution of the local times for a Borel right process can be
fully characterized by certain associated Gaussian processes;
results of this flavor go by the name of \emph{Isomorphism
Theorems}. Several versions have been developed by Ray \cite{Ray63} and Knight \cite{Knight63},
Dynkin \cite{Dynkin83, Dynkin84}, Marcus and Rosen \cite{MR92,
MR01}, Eisenbaum \cite{Eisenbaum95} and Eisenbaum, Kaspi, Marcus,
Rosen and Shi \cite{EKMRS00}. In what follows, we present the second
Ray-Knight theorem in the special case of a continuous-time random
walk. It first appeared in \cite{EKMRS00}; see also
Theorem 8.2.2 of the book by Marcus and Rosen \cite{MR06} (which
contains a wealth
of information on the connection between local times and Gaussian processes). It is easy
to verify that the continuous-time random walk on a connected graph
is indeed a recurrent strongly symmetric Borel right process.

\begin{theorem}[Generalized Second Ray-Knight Isomorphism Theorem]
\label{thm:rayknight} Fix $v_0 \in V$ and define
the inverse local time,
\begin{equation}\label{eq:inverselt}
\tau(t) = \inf\{s:
L_s^{v_0} > t\}.
\end{equation}
Let $T_{0}$ be the hitting time to $v_0$ and let
$\Gamma_{v_0}(x, y) = \E_x(L_{T_0}^y)$. Denote by $\eta = \{\eta_x:
x\in V\}$ a mean zero Gaussian process with covariance
$\Gamma_{v_0}(x, y)$. Let $P_{v_0}$ and $P_{\eta}$ be the measures
on the processes $\{L_{T_0}^x\}$ and $\{\eta_x\}$, respectively.
Then under the measure $P_{v_0} \times P_\eta$, for any $t > 0$
\begin{equation}\label{eq:law}
\left\{L_{\tau(t)}^x + \frac{1}{2} \eta_x^2: x\in V\right\}
\stackrel{law}{=} \left\{\frac{1}{2} (\eta_x + \sqrt{2t})^2 : x\in V
\right\}\,.
\end{equation}
\end{theorem}

Thus to every continuous-time random walk, we can associate
a Gaussian process $\{\eta_v\}_{v \in V}$.
As discussed in Section \ref{sec:GFF},
we have the relationship $d(u,v) = \sqrt{\reff(u,v)}$,
where $d(u,v) = \sqrt{\E\,|\eta_u-\eta_v|^2}$.
In particular, the process $\{\eta_v\}_{v \in V}$ is
the Gaussian free field on the network $G$.

Using the Isomorphism Theorem in conjunction
with concentration bounds
for Gaussian processes, we already have enough
machinery to prove the following upper bound
in Section \ref{sec:blanket},
\begin{equation}\label{eq:firsthalf}
t_{\mathrm{cov}}(G) \leq t_{\mathrm{bl}}(G,\delta)
\lesssim_{\delta} \mathcal C \left[\gamma_2(V, d)\right]^2 = \mathcal C \left[\gamma_2(V,\sqrt{\reff})\right]^2.
\end{equation}

We also show how to prove a matching lower bound in terms
of $\gamma_2$, but for a slightly different notion of ``blanket time.''

\medskip

Thus \eqref{eq:firsthalf} proves the first half of Theorem \ref{thm:mainthm}.
The lower bound for cover times quite a bit more difficult to prove.
Of course, the cover and return
time relates to the event $\left\{ \exists v : L^v_{\tau(t)} = 0 \right\}$,
and unfortunately the correspondence \eqref{eq:law} seems
too coarse to provide lower bounds on the probability
of this event directly.

To this end, we need to show that for the right value of $t$
in Theorem \ref{thm:rayknight},
we often have $\eta_x \approx - \sqrt{2t}$ for some $x \in V$.
The main difficulty is that we will have to show that
there is often a vertex $x \in V$ with $|\eta_x + \sqrt{2t}|$
being {\em much smaller} than the standard
deviation of $\eta_x$.
In doing so, we will use the full power
of the majorizing measures theory,
as well as the special structure of the
Gaussian processes arising from the Isomorphism Theorem.

\medskip
\noindent {\bf The discrete Gaussian free field and a tree-like
subprocess.} In Section \ref{sec:GFF} (see \eqref{eq:density}), we
recall that the Gaussian processes arising from the Isomorphism
Theorem are not arbitrary, but correspond to the Gaussian free field
(GFF) associated with $G$. Special properties of such processes will
be essential to our proof of Theorem \ref{thm:mainthm}. In
particular, if we use $\reff(v,S)$ to denote the effective
resistance between a point $v$ and a set of vertices $S \subseteq
V$, then we have the relationship
\begin{equation}\label{eq:above}
\sqrt{\reff(v,S)} = \dist_{L^2}(\eta_v, \CH(\{\eta_w\}_{w \in S})),
\end{equation}
where $\CH(\cdot)$ denotes the affine hull, and $\dist_{L^2}$
is the $L^2$ distance in the Hilbert space
underlying the process $\{\eta_v\}_{v \in V}$.
In Section \ref{sec:resist}, we prove a number
of properties of the effective resistance metric
(e.g. Foster's network theorem); combined
with \eqref{eq:above}, this yields
some properties unique to processes
arising from a GFF.

\medskip

Next, in Section \ref{sec:mm}, we recall that one of the
primary components of the majorizing measures theory is that
every Gaussian process $\{\eta_i\}_{i \in I}$ contains a ``tree like'' subprocess
which controls $\E \sup_{i \in I} \eta_i$.
After a preprocessing step that ensures our trees
have a number of additional features,
we use the structure of the GFF to
select a representative
subtree with very strong independence properties
that will be essential to our analysis of cover times.

\medskip
\noindent
{\bf Restructuring the randomness and a percolation argument.}
The majorizing measures theory is designed to control the first
moment $\E \sup_{i \in I} \eta_i$ of the supremum
of Gaussian process.  In analyzing \eqref{eq:law}
to prove a lower bound on the cover times,
we actually need to employ a variant
of the second moment method.  The need for this,
and a detailed discussion of how it proceeds,
are presented at the beginning of
Section \ref{sec:cover}.

Towards this end, we want to associate events to the leaves of our
``tree like'' subprocess which can be thought of as ``open events''
in a percolation process on the tree. For general trees, it is known
that the second moment method gives accurate estimates for the
probability of having an open path to a leaf \cite{Lyons92}.  While
our trees are not regular, they are ``regularized'' by the
majorizing measure, and we do a somewhat standard analysis of such a
process in Section \ref{sec:percolation}.

The real difficulty involves setting up the right
filtration on the probability space corresponding
to our tree so that the percolation argument yields
the desired control on the cover times.
This requires a delicate definition of the
events associated to each edge,
and the ensuing analysis forms the
technical core of our argument in Section \ref{sec:cover}.

\medskip
\noindent
{\bf Algorithmic issues.}
In order to complete the proof of Theorem \ref{thm:algorithm}
and thus resolve Question \ref{ques:AF},
we present a deterministic algorithm
which computes an approximation to $\gamma_2(X,d)$
for any metric space $(X,d)$.  This is achieved
in Section \ref{sec:algorithm}.  While the algorithm is fairly
elementary to describe, its analysis requires
a number of tools from the majorizing measures theory.

We remark that, in combination with Theorem \ref{thm:mainthm},
this yields the following result.

\begin{theorem}
For any finite-state, reversible Markov chain
presented as a network $G=(V,E)$ with given conductances
$\{c_{xy}\}$, there is a deterministic, polynomial-time algorithm
which computes a value $A(G)$ such that
$$
A(G) \asymp t_{\mathrm{cov}}(G).
$$
\end{theorem}

Observe that for general reversible chains, the cover time
is not necessarily bounded a polynomial in $|V|$, and thus
even randomized simulation of the chain does not yield a
polynomial-time algorithm for approximating $t_{\mathrm{cov}}(G)$.
Finally, in Section \ref{sec:coverapps}, we prove Theorems \ref{thm:pseudo}
and \ref{thm:randalg}
in the setting of arbitrary reversible Markov chains,
leading to a near-linear time randomized algorithm
for computing cover times.

\section{Gaussian processes and local times}
\label{sec:iso}

We now discuss properties of the Gaussian
processes arising from the isomorphism theorem (Theorem \ref{thm:rayknight}).
In Section \ref{sec:blanket}, we show that the isomorphism theorem,
combined with concentration properties of Gaussian processes,
is already enough to get strong control
on blanket times and related quantities.

In Section \ref{sec:resist},
we prove some geometric properties of the resistance
metric on networks that will be crucial
to our work on the cover time in Sections \ref{sec:mm}
and \ref{sec:cover}.  Finally,
in Section \ref{sec:GFF}, we recall the
definition of the Gaussian free field
and show how the geometry
of such a process relates to the geometry
of the underlying resistance metric.

\subsection{The blanket time}
\label{sec:blanket}

We first remark that the covariance matrix
of the Gaussian process arising from
the isomorphism theorem can be calculated
explicitly in terms of the resistance metric on the
network $G(V)$.
Throughout this section, the process $\{\eta_x\}_{x \in V}$
refers to the one resulting from Theorem \ref{thm:rayknight}
with $v_0 \in V$ some fixed (but arbitrary) vertex,
$\tau(t)$ refers to the inverse
local time defined in \eqref{eq:inverselt},
and $T_0$ is the hitting time to $v_0$.

\begin{lemma}\label{lem-covarance-resistance}
For every $x, y \in V$,
\[\Gamma_{v_0}(x, y) = \E_x(L_{T_0}^y) = \tfrac{1}{2}(R_{\mathrm{eff}}(x, v_0) +
R_{\mathrm{eff}}(v_0, y) - \mathrm{R}_{\mathrm{eff}}(x, y)) \,.\] In
particular,
$$
\E\,(\eta_x-\eta_y)^2 = \reff(x,y).
$$
\end{lemma}

\begin{proof}
To prove
the lemma, we use the cycle identity for hitting times (see, e.g., \cite[Lem. 10.10]{LPW09}) which asserts that,
\begin{equation}\label{eq:cycle}
H(x,v_0) + H(v_0, y) + H(y,x) = H(x,y) + H(y,v_0) + H(v_0, x).
\end{equation}
Averaging both sides of \eqref{eq:cycle} and recalling \eqref{eq:reffkappa} yields
\[
H(x,v_0) + H(v_0, y) + H(y,x) = \frac{\mathcal C}{2} \left[\reff(x,v_0)+\reff(v_0,y)+\reff(x,y)\right].
\]
Now, we subtract $\mathcal C \reff(x,y) = H(x,y) + H(y,x)$ from both sides, giving
\[
H(x,v_0) + H(v_0, y) - H(x,y) = \frac{\mathcal C}{2} \left[\reff(x,v_0,)+\reff(v_0,y)-\reff(x,y)\right]
\]
Finally, we conclude using the identity (see, e.g. \cite[Ch 2., Lem. 9]{AF}),
\[
\E_x(L_{T_{0}}^y) = \frac{1}{\mathcal C} \left(H(x,v_0) + H(v_0, y) - H(x,y)\right).
\]
\end{proof}

We now relate the blanket time of the random walk to the expected
supremum of its associated Gaussian process. \remove{ The first
property established by Talagrand \cite{Talagrand}, has to do with
the boundedness of Gaussian process. For a mean zero Gaussian
process $\{\eta_x: x\in V\}$, define the generated metric $(V, d)$
by
\[d(x, y) = \big(\E (\eta_x - \eta_y)^2\big)^{1/2}\,, \mbox{ for } x, y\in V\,.\]
\begin{theorem}[\cite{Talagrad}]\label{thm-gaussian-sup}
Let $\mathcal{M}$ be the majorizing measure of the metric $(V, d)$
generated by the mean zero Gaussian process $\{\eta_x: x\in V\}$.
Then $\E \sup_{x \in V} \eta_x \asymp \mathcal{M}$.
\end{theorem}}
The following is a central facet of the theory of concentration of
measure; see, for example, \cite[Thm. 7.1, Eq. (7.4)]{Ledoux89}.

\begin{lemma}\label{lem-gaussian-concentration}
Consider a Gaussian process $\{\eta_x: x\in V\}$ and define $\sigma
= \sup_{x \in V} (\E(\eta_x^2))^{1/2}$. Then for $\alpha>0$,
\[\P\left(\left|\sup_{x\in V} \eta_x - \E \,\sup_{x \in V} \eta_x\right| > \alpha\right) \leq 2 \exp(-\alpha^2 / 2\sigma^2)\,.\]
\end{lemma}

We are now ready to establish the upper bound on the
strong blanket time
$t^\star_{\mathrm{bl}}(G,\delta)$, for any fixed $0< \delta < 1$. Note that this will naturally yield an upper bound on $t_{\mathrm{bl}}(\delta)$.

\begin{theorem}\label{thm:blanket}
Consider a network $G(V)$ and its total conductance $\C = \sum_{x
\in V} c_x$. For any fixed $0<\delta <1$, the blanket time
$t^\star_{\mathrm{bl}}(G,\delta)$ of the random walk on $G(V)$
satisfies
$$t^\star_{\mathrm{bl}} (G,\delta)\lesssim_\delta  \C  \cdot \left( \E\,\sup_{x \in V} \eta_x\right)^2,$$
where $\{\eta_x\}$ is the associated Gaussian process from Theorem
\ref{thm:rayknight}.
\end{theorem}
\begin{proof}
We first prove that for some $A_\delta>0$
\begin{equation}\label{eq-blanket-time}
t^\star_{\mathrm{bl}} (\delta)\leq A_\delta \C \left(\left( \E\,\sup_{x
\in V} \eta_x\right)^2 + \sup_{x \in V}
\E\,(\eta_x^2)\right).\end{equation} Fix a vertex $v_0 \in V$ and
consider the local times $\{L_{\tau(t)}^x: x\in V\}$, where for $t >
0$, we write $\tau(t) = \inf\{s: L_{s}^{v_0} > t\}$. Let $\sigma =
\sup_{x \in V} \sqrt{\E(\eta_x^2)}$ and $\Lambda = \E \sup_x
\eta_x$.

\medskip

Use $\{\eta_x^L\}$ to denote the copy of the Gaussian process
corresponding to the left-hand side of \eqref{eq:law}, and
$\{\eta_x^R\}$ to denote the i.i.d. process corresponding to the
right-hand side. Fix $\beta>0$, and set $t = t(\beta) = \beta
(\Lambda^2 + \sigma^2)$. By Theorem~\ref{thm:rayknight}, we get that
\begin{align*}\P\left(\min_{x} L^x_{\tau(t)} \leq \sqrt{\delta} t\right) \leq
\P\left(\inf_x \frac{1}{2}(\eta_x^R + \sqrt{2t})^2 \leq \frac{1 +
\sqrt{\delta}}{2} t\right) + \P\left(\sup_x \frac{1}{2}(\eta_x^L)^2
\geq \frac{1 - \sqrt{\delta}}{2} t\right)\,.
\end{align*}
Therefore,
\begin{align*}
\P\left(\min_{x} L^x_{\tau(t)} \leq \sqrt{\delta} t\right) \leq
\P\left(\inf_x \eta_x^R \leq - a_\delta \sqrt{t}\right) +
\P\left(\sup_x |\eta_x^L| \geq b_\delta \sqrt{t}\right)\,,
\end{align*}
where $a_\delta = \sqrt{2} - \sqrt{1+\sqrt{\delta}}$ and $b_\delta =
\sqrt{1 - \sqrt{\delta}}$. Applying
Lemma~\ref{lem-gaussian-concentration}, we obtain that if $\beta >
\beta_0(\delta)$ for some $\beta_0(\delta)>0$, then
\begin{equation}\label{eq-bound-L}\P\left(\min_{x} L^x_{\tau(t)} \leq \sqrt{\delta} t \right)
\leq 6 \exp(-\gamma_\delta \beta)\,,\end{equation} where
$\gamma_\delta = \frac{1}{2} (a_\delta^2 \wedge b_\delta^2)$.
On the other hand, we have
\begin{align*}
\P\left(\max_x
L^x_{\tau(t)} \geq t/\sqrt{\delta}\right) \leq \P\left(\max_x
\frac{1}{2}(\eta_x^R + \sqrt{2t})^2 \geq t/\sqrt{\delta}\right) =
\P\left(\max_x \eta_x \geq a'_\delta \sqrt{t}\right)\,,
\end{align*}
where $a'_\delta = \sqrt{1/\delta} - 1$. Applying
Lemma~\ref{lem-gaussian-concentration} again for $\beta >
\beta_0(\delta)$, we get that
\begin{equation}\label{eq-bound-tau}
\P\left(\max_x
L^x_{\tau(t)} \geq t/\sqrt{\delta}\right) \leq 2 \exp(-\gamma'_\delta
\beta)\,,
\end{equation}
where $\gamma'_\delta = (a'_\delta)^2 /2$.  Note that assuming
$\min_{x} L^x_{\tau(t)} \geq \sqrt{\delta } t$ and $\max_x
L^x_{\tau(t)} \leq t/\sqrt{\delta}$, we have $\tau(t) = \sum_x c_x
L^x_{\tau(t)} \leq \C t/\sqrt{\delta}$ as well as $\min_{x, y}
L^{x}_{\tau(t)}/L^y_{\tau(t)} \geq \delta$. It then follows that
$\tau^\star_{\mathrm{bl}} \leq \tau(t) \leq \C t/\sqrt{\delta}$.
Therefore, we can deduce that
\begin{equation*}\left\{\tau^\star_{\mathrm{bl}} \geq \C t
/\sqrt{\delta}\right\} \subset \left\{\min_{x} L^x_{\tau(t)} \leq
\sqrt{\delta } t\right\} \bigcup \left\{\max_x L^x_{\tau(t)} \geq
t/\sqrt{\delta}\right\}\,.\end{equation*} Combined with
\eqref{eq-bound-L} and \eqref{eq-bound-tau}, it yields that
\[\P(\tau^\star_{\mathrm{bl}} \geq \C t /\sqrt{\delta}) \leq 6 \exp(-\gamma_\delta \beta) + 2 \exp(-\gamma'_\delta
\beta)\,.\] It then follows that $t^\star_{\mathrm{bl}} \leq A_\delta \C
(\Lambda^2 + \sigma^2)$ for some $A_\delta>0$ which depends only on
$\delta$, establishing \eqref{eq-blanket-time}.

It remains to prove that $\sigma = O(\Lambda)$. To this end, let
$x^*$ be such that $\E \eta_{x^*}^2 = \sigma^2$. We have
\begin{equation}\label{eq-Lambda-sigma-relation}
\Lambda \geq \E \max(\eta_{v_0},\eta_{x^*}) = \E \max(0,\eta_{x^*})
= \frac{\sigma}{\sqrt{2\pi}}\,.\end{equation} This completes the
proof for the continuous-time case.
\end{proof}

\begin{remark}
An interesting question is the asymptotic behavior of $\delta$-blanket time as $\delta \to 1$, namely the dependence on $\delta$ of $A_\delta$ in \eqref{eq-blanket-time}. As implied in the proof, we can see that $$A_\delta \lesssim \frac{1}{\gamma_\delta} + \frac{1}{\gamma'_\delta} \lesssim \frac{1}{(1- \delta)^2}\,.$$
These asymptotics are tight for the complete graph; see e.g. \cite[Cor. 2]{WZ96}.
\end{remark}

We next extend the proof
of the preceding theorem to the case of the discrete-time random walk.
The next lemma contains the main estimate required for this extension.

\begin{lemma}\label{lem:dev}
Let $G(V)$ be a network and write $\gamma_2 = \gamma_2 (V, \sqrt{R_{\mathrm{eff}}})$. Then for all $u \geq 16$, we have
\[\sum_{v\in V}e^{- u \cdot c_v \gamma_2^2} \lesssim e^{-u/8}\,.\]
\end{lemma}
\begin{proof}
By definition of the $\gamma_2$ functional, we can choose a sequence of partitions $\mathcal{A}_k$ with $|\mathcal{A}_k| \leq 2^{2^k}$ such that
\[\gamma_2 \geq \frac{1}{2} \sup_{v\in V}\sum_{k\geq 0} 2^{k/2} \diam(A_k(v))\,.\]
For $v\in V$, let $k_v = \min\{k: \{v\} \in \mathcal{A}_k\}$. It is clear that $R_{\mathrm{eff}} (u,v) \geq 1/c_v$ for all $u\neq v$ and hence $(\diam (A_{k_v - 1}(v)))^2 \geq 1/c_v$. Therefore, we see that
\[\sum_{v\in V}e^{- u \cdot c_v \gamma_2^2} = \sum_{k=0}^\infty \sum_{v: k_v = k+1} e^{-u \cdot c_v \gamma_2^2} \leq \sum_{k=1}^\infty 2^{2^{k+1}} e^{-u 2^{k} /4} \lesssim e^{- u /8}\,,\]
completing the proof.
\end{proof}

\begin{theorem}\label{thm:blanketdiscr}
Consider a network $G(V)$ and its total conductance $\C = \sum_{x
\in V} c_x$. For any fixed $0<\delta <1$, the discrete blanket time
$t_{\mathrm{bl}}(G,\delta)$ of the random walk on on $G(V)$
satisfies
$$t_{\mathrm{bl}} (G,\delta)\lesssim_\delta  \C  \cdot \left( \E\,\sup_{x \in V} \eta_x\right)^2,$$
where $\{\eta_x\}$ is the associated Gaussian process from Theorem
\ref{thm:rayknight}.
\end{theorem}

\begin{proof}
We now consider the embedded discrete-time random walk of the
continuous-time counterpart (i.e. the corresponding jump chain; see
\cite[Ch. 2]{AF}). Let $N^v_t$ be such that $c_v \cdot N^v_t$ is the
number of visits to vertex $v$ up to continuous time $t$, i.e.
$N^v_t$ is a discrete-time analog of the local time $L^v_t$.

Fix a vertex $v_0 \in V$ and
consider the local times $\{L_{\tau(t)}^x: x\in V\}$.
Let $\sigma =
\sup_{x \in V} \sqrt{\E(\eta_x^2)}$ and $\Lambda = \E \sup_x
\eta_x$.
Again, set $t = \beta(\Lambda^2 + \sigma^2)$.

Let $\tau_{\mathrm{bl}}(\delta)$ denote the first time at which
$N^x_t \geq \frac{\delta t}{\mathcal C}$ for every $x \in V$.
Assuming that $\min_{x} N^x_{\tau(t)} \geq \delta^{1/4} t$ and
$\max_x N^x_{\tau(t)} \leq t/\delta^{3/4}$, we have $\tau(t) =
\sum_x c_x N^x_{\tau(t)} \leq \C t/\delta^{3/4}$ and thus $\min_x
N^x_{\tau(t)} \geq \delta \tau(t)/\mathcal{C}$. It then follows that
$\tau_{\mathrm{bl}}(\delta) \leq \tau(t) \leq \C t/\delta^{3/4}$.
Therefore, we deduce that
\[\left\{\tau_{\mathrm{bl}}(\delta) \geq \frac{\C t}{\delta^{3/4}}\right\} \subset
\left\{\min_{x} N^x_{\tau(t)} \leq \delta^{1/4} t\right\} \bigcup \left\{\max_x
N^x_{\tau(t)} \geq t/\delta^{3/4}\right\}\,.\]
Therefore we have,
\begin{align*}\P\left(\tau_{\mathrm{bl}}(\delta)
\geq \frac{\C t}{\delta^{3/4}}\right) &\leq \P\left(\min_{x} L^x_{\tau(t)} \leq \sqrt{\delta } t \textrm{ or } \max_x
L^x_{\tau(t)} \geq t/\sqrt{\delta}\right)\\
&+ \P\left(\forall x: \sqrt{\delta} t\leq L^x_{\tau(t)} \leq  t / \sqrt{\delta} \mid \min_{x} N^x_{\tau(t)} \leq \delta^{1/4} t \mbox{ or } \max_x
N^x_{\tau(t)} \geq t/\delta^{3/4}  \right)\,.
\end{align*}
Note that we have already bounded the first term
 in \eqref{eq-bound-L}
 and \eqref{eq-bound-tau}.

The second term can be bounded by a simple application of
a large deviation inequality
on the sum of i.i.d. exponential variables. Precisely,
\[\sum_{x\in V}\P\left(  \sqrt{\delta} t\leq L^x_{\tau(t)} \leq  t / \sqrt{\delta} \mid N^x_{\tau(t)} \leq \delta^{1/4} t \mbox{ or } N^x_{\tau(t)} \geq  t/\delta^{3/4} \right) \lesssim \sum_{x\in V}e^{-\tilde{a}_\delta \cdot c_x t}\]
for some constant $\tilde{a}_\delta>0$ depending only on $\delta$.
 Recall that Theorem (MM)
implies $\E\sup_x \eta_x \asymp \gamma_2(V,\sqrt{\reff})$. By
\eqref{eq-Lambda-sigma-relation}, we see that $\sigma \leq
\sqrt{2\pi} \Lambda$. Altogether, we get that $t \asymp \Lambda^2
\asymp_{\beta} \left[\gamma_2(V,\sqrt{\reff})\right]^2$. Applying
Lemma \ref{lem:dev},
  we conclude that there exists $\tilde{\beta}_0(\delta)>0$ depending only on $\delta$ such that for all $\beta \geq \tilde{\beta}_0(\delta)$, we have
\[\P(\tau_{\mathrm{bl}}(G,\delta) \geq \C t / \delta^{3/4}) \lesssim e^{-\tilde{b}_\delta \beta}\]
where $\tilde{b}_\delta$ is a constant depending only on $\delta$. This immediately yields the desired upper bound on the blanket time for the discrete-time random walk.
\end{proof}

We next exhibit a lower bound on a variation of blanket time
(considered in \cite{KKLV00}). It is apparent that the lower bound
on the cover time, which will be proved in Section \ref{sec:cover},
is an automatic lower bound on the blanket time.  In what follows,
though, we try to give a simple argument that can be regarded as a
warm up. For the convenience of analysis, we consider the following
notion. For $0<\varepsilon<1$, define
\begin{equation}\label{eq-def-new-blanket}
t_{\mathrm{bl}}^*(G,\varepsilon) = \max_{w\in V}\inf\{s : \P_w(\forall
u, v\in V: L^u_t \leq 2 L^v_t) > \varepsilon \mbox{ for all } t\geq
s\}\,.
\end{equation}
\begin{theorem}\label{thm-lower-blanket}
Consider a network $G(V)$ and its total conductance $\C = \sum_{x
\in V} c_x$. For any fixed $0<\varepsilon <1$, we have
\[t_{\mathrm{bl}}^*(G,\varepsilon) \gtrsim_\varepsilon  \C  \cdot \left( \E\,\sup_{x \in V} \eta_x\right)^2\,.\]
\end{theorem}
In order to prove Theorem~\ref{thm-lower-blanket}, we will use the
next simple lemma.  We will also require this
estimate in Section \ref{sec:cover}.

\begin{lemma}\label{lem-tau-t}
Let $\tau(t)$ be the inverse local time at vertex $v_0$, as defined
in \eqref{eq:inverselt}. Let $\C$ be the total conductance
and let $\mathfrak{D} = \max_{x, y\in
V}\sqrt{R_{\mathrm{eff}}(x,y)}$. Then, for all $\beta>0$ and $t \geq
\mathfrak{D^2}/\beta^2$,
\[\P_{v_0}\left(\tau(t) \leq  \beta \C t\right) \leq 3 \beta\,.\]
\end{lemma}

\begin{proof}
We use $\P_v$ to denote the measure on random walks started at a
vertex $v \in V$, and we use $\E_v$ similarly. Let $p_\delta =
\min_v \{\P_{v}\left(\tau(t) \leq \delta \C t\right)\}$ for some
$\delta>0$. Using the strong Markov property, we get that for all $v
\in V$,
\[\P_v(\tau(t) \geq k \delta \C t) \leq (1 - p_\delta)^k\,.\]
In particular, $\E_v \tau(t) \leq \delta \C t/p_\delta$.

By Theorem~\ref{thm:rayknight}, it follows easily that $\E_{v_0}
\tau(t) = \C t$. Since $\E_v \tau(t) \geq \E_{v_0}(\tau(t))$, we
deduce that $p_\delta \leq \delta$. Let $u = u(\delta)$ be such that
$\P_{u}\left(\tau(t) \leq \delta \C t\right) = p_\delta$. Let $Y, Z$
be random variables with the law $\tau(t)$, when the random walk is
started at $u$ and $v_0$, respectively. Clearly,
\begin{equation}\label{eq-Y-Z}
Y \stackrel{law}{=} Z + T_{v_0}\,,\end{equation} where $T_{v_0}$ is
distributed as the hitting time to $v_0$, when then random walk is
started at $u$ and $T_{v_0}$ is independent of $Z$.

Since $\reff(u,v_0) \leq \mathfrak{D}^2$, we have $\E_u T_{v_0} \leq
\C \mathfrak{D}^2$ (by \eqref{eq:reffkappa}), and this yields
$\P_u(T_{v_0} \geq  \C \mathfrak{D}^2/\beta)\leq \beta$. Using the
assumption $t \geq \mathfrak{D}^2/\beta^2$ and \eqref{eq-Y-Z}, we
conclude that \[\P(Z \leq \beta \C t) \leq \P(Z \leq 2 \beta \C t -
\C \mathfrak{D}^2/\beta) \leq \P(Y \leq 2\beta \C t) + \P(T_{v_0}
\geq \C \mathfrak{D}^2/\beta)\leq p_{2\beta} + \beta \leq 3
\beta\,,\] as required.
\end{proof}

We are now ready to establish the lower bound on
$t_{\mathrm{bl}}^*(G,\varepsilon)$.
\begin{proof}[Proof of Theorem \ref{thm-lower-blanket}]
We consider the associated Gaussian process as in the proof of
Theorem~\ref{thm:blanket}. Let $\sigma = \sup_{x\in V}\sqrt{\E
\eta_x^2}$ and $\Lambda  = \E \sup_x \eta_x$. Observe that the
maximal hitting time is a simple lower bound on
$t_{\mathrm{bl}}^*(G,\varepsilon)$ up to a constant depending only on
$\varepsilon$. In light of Lemma~\ref{lem-covarance-resistance}, we
see $t_{\mathrm{bl}}^*(G,\varepsilon) \gtrsim_\varepsilon \C \cdot
\sigma^2$. Therefore, we can assume in what follows
\begin{equation}\label{eq-Lambda-sigma}\Lambda^2 \geq 100 \log(4/\varepsilon) \varepsilon^{-2} \,\sigma^2\,.\end{equation} Let $t_* =
\frac{1}{2} \Lambda^2$. By Lemma~\ref{lem-gaussian-concentration},
we get
\[\P\left(\inf_{x\in V} \frac{1}{2}(\eta^R_x + \sqrt{2t_*})^2 \leq  \log(4/\varepsilon) \sigma^2\right) \geq  \P\left(|\sup_{x\in V} \eta_x^R - \Lambda|\leq \sqrt{2 \log(4/\varepsilon)}\, \sigma\right)\geq 1 - \frac{\varepsilon}{2}\,.\]
Applying Theorem~\ref{thm:rayknight}, we obtain
\[\P\left(\inf_{x\in V} L^x_{\tau(t_*)} \leq \log(4/\varepsilon) \sigma^2\right)\geq 1- \frac{\varepsilon}{2}\,. \]
By triangle inequality, we have $\mathfrak{D} \leq 2 \sigma$.
Recalling the assumption \eqref{eq-Lambda-sigma}, we can apply
Lemma~\ref{lem-tau-t} and deduce that
\[\P(\tau(t_*) \leq \varepsilon \C t_* /6 ) \leq \varepsilon/2\,.\]
Writing $t_0 = \varepsilon \C t_* /6$, we can then obtain that
\[\P\left(\inf_{x\in V} L^x_{t_0} \leq \log(4/\varepsilon) \sigma^2, \tau(t_*) \geq t_0\right)\geq 1- \varepsilon\,.\]
Also, we see that $\sup_{x\in V} L^x_{t_0} \geq \varepsilon
\Lambda^2/12$ whenever $\tau(t_*) \geq t_0$. Using assumption
\eqref{eq-Lambda-sigma} again, we conclude
\[\P_{v_0}(\exists x, y\in V: L^x_{t_0} > 2 L^y_{t_0}) \geq 1-\varepsilon\,.\]
This implies that $t_{\mathrm{bl}}^*(G,\varepsilon) \geq t_0$,
completing the proof.
\end{proof}

\subsection{An asymptotically strong upper bound}
\label{sec:gffstrong}

Finally, we show a strong upper bound for the asymptotics of $t_{\mathrm{cov}}$ on a sequence of graphs $\{G_n\}$,
assuming $t_{\mathrm{hit}}(G_n) = o(t_{\mathrm{cov}}(G_n))$.

\begin{theorem}\label{thm:tight}
For any graph $G= (V, E)$ with $v_0\in V$, let $t_{\mathrm{hit}}(G)$
be the maximal hitting time in $G$ and let $\{\eta_v\}_{v\in V}$ be
the GFF on $G$ with $\eta_{v_0} = 0$. Then, for a universal constant
$C>0$,
\[ t_{\mathrm{cov}}(G) \leq \left(1+ C\sqrt{\frac{t_{\mathrm{hit}}(G)}{t_{\mathrm{cov}}(G)}}\,\right) \cdot |E| \cdot \left(\E \sup_{v\in V} \eta_v\right)^2\,.\]
\end{theorem}

\begin{proof}
Theorem \ref{thm:blanketdiscr} asserts that
\begin{equation}\label{eq-variance-sup}
t_{\mathrm{cov}}(G)  \preceq (\E \max_v\eta_v)^2\,,
\end{equation}
where $\preceq$ denotes stochastic domination. Write $\sigma^2 =
\max_v \E \eta_v^2$. Note that $\sigma^2$ corresponds to the
diameter of $V$ in the effective resistance metric, thus
$t_{\mathrm{hit}}(G) \asymp |E| \sigma^2$. Denote by $S = \sum_v d_v
\eta_v^2$, where $d_v$ is the degree of vertex $v$. By a generalized
H\"older inequality and moment estimates for Gaussian variables
(here we use that $\E X^6 = 15$ for a standard Gaussian variable
$X$), we obtain that
$$\E S^3 \leq \sum_{u, v, w} d_u d_v d_w \E(\eta_u^2 \eta_v^2 \eta_w^2) \leq \sum_{u, v, w} d_u d_v d_w \E(\eta_u^6)^{1/3}\E(\eta_v^6)^{1/3} \E(\eta_w^6)^{1/3} \leq 15 |E|^3 \sigma^6\,.$$
An application of Markov's inequality then yields
\begin{equation}\label{eq-S}\P(S \geq \alpha |E| \sigma^2) \leq
\frac{15}{\alpha^3}\,.\end{equation} Write $Q = \sum_v d_v \eta_v$.
Clearly, $Q$ is a centered Gaussian with variance bounded by $4|E|^2
\sigma^2$ and therefore,
\begin{equation}\label{eq-Q}\P(|Q| \geq \alpha |E| \sigma) \leq
2\mathrm{e}^{-\alpha^2/8}\,.\end{equation}

For $\beta>0$, let $t = \frac{1}{2}(\E \max_v \eta_v + \beta
\sigma)^2$. Noting $\tau(t) = \sum_v d_v L^v_{\tau(t)}$ and
recalling the Isomorphism theorem (Theorem \ref{thm:rayknight}), we get that
$$\tau(t) \preceq 2|E|t + \frac{\sqrt{2t}}{2} |Q| + \frac{1}{2} S\,.$$
Combined with \eqref{eq-S} and \eqref{eq-Q}, we deduce that
\begin{equation}\label{eq-tau-t}
\P(\tau(t) \geq 2|E|t + \sqrt{2t} \beta |E| \sigma + \beta |E|
\sigma^2) \leq \frac{12}{(\beta-2)^2}+ 2\mathrm{e}^{-\beta^2/8}\,.
\end{equation}

We now turn to bound the probability for $\tau_{\mathrm{cov}} >
\tau(t)$. Observe that on the event $\{\tau_{\mathrm{cov}} >
\tau(t)\}$, there exists $v\in V$ such that $L^v_{\tau(t)} = 0$. It
is clear that for all $v\in V$, we have $\P(\eta_v^2 \geq \beta
\sigma^2/2) \leq 2\mathrm{e}^{-\beta/4}$. Since $\{\eta_v\}_{v\in
V}$ and $\{L^v_{\tau(t)}\}_{v\in V}$ are two independent processes,
we obtain
\begin{equation}\label{eq-decompose}\P\left(\left\{\tau_{\mathrm{cov}} > \tau(t)\right\} \setminus
\left\{\exists v\in V:L^v_{\tau(t)} + \tfrac{1}{2}\eta_v^2 < \beta
\sigma^2/2 \right\}\right) \leq
2\mathrm{e}^{-\beta/4}\,.\end{equation} On the other hand, we deduce
from the concentration of Gaussian processes (Lemma \ref{lem-gaussian-concentration}) that
\[\P\left(\inf_v (\sqrt{2t} + \eta_v)^2 \leq \beta \sigma/2\right) \leq 2\mathrm{e}^{-\beta/8}\,.\]
Applying Isomorphism theorem again and combined with
\eqref{eq-decompose}, we get that $$\P(\tau_{\mathrm{cov}}>\tau(t))
 \leq 4\mathrm{e}^{-\beta/8} \,.$$
Combined with \eqref{eq-tau-t}, it follows that
$$\P(\tau_{\mathrm{cov}} \geq 2|E|t + \sqrt{2t} \beta |E| \sigma +
\beta |E| \sigma^2)\leq \frac{15}{\beta^3}+ 2\mathrm{e}^{-\beta^2/8}
+ 4\mathrm{e}^{-\beta/8}\,.$$ Since $t = \frac{1}{2}(\E \max_v
\eta_v + \beta \sigma)^2$, we can deduce that for some universal
constant $C_1>0$,
$$t_{\mathrm{cov}}(G) \leq |E|(\E\sup_v \eta_v)^2 + C_1 |E| (\sigma^2 + \sigma \E\sup_v\eta_v)\,.$$
Recalling \eqref{eq-variance-sup}, we complete the proof.
\end{proof}

\subsection{Geometry of the resistance metric}
\label{sec:resist}

We now discuss some relevant properties of
the resistance metric on a network $G(V)$.

\medskip

\noindent {\bf Effective resistances and network reduction.}
For a subset $S \subseteq V$, define the
quotient network $G/S$ to have vertex set $(V \setminus S) \cup \{ v_S
\}$, where $v_S$ is a new vertex disjoint from $V$. The conductances
in $G/S$ are defined by $c^{G/S}_{xy} = c_{xy}$ if $x,y \notin S$
and $c_{v_S x} = \sum_{y \in S} c_{xy}$ for $x \notin S$.

Now, given $v \in V$ and $S \subseteq V$, we put
\begin{equation}\label{eq:reffsub}
\reff(v,S) \deq \reff^{G/S}(v, v_S),
\end{equation}
where the latter effective resistance is computed in $G/S$. For two
disjoint sets $S,T \subseteq V$, we define
$$
\reff(S,T) \deq \reff^{G/S}(v_S, T),
$$
and the resistance is defined to be $0$ if $S \cap T \neq
\emptyset$. It is straightforward to check that
$\reff(S,T)=\reff(T,S)$. The following network reduction lemma was
discovered by Campbell \cite{Campbell} under the name ``star-mesh
transformation'' (see also, e.g., \cite[Ex. 2.47(d)]{LP}). We give a
proof for completeness.

\begin{lemma}\label{lem-network-reduction}
For a network $G(V)$ and a subset $\widetilde{V} \subset V$, there
exists a network $\tilde{G}(\tilde{V})$ such that for all $u,v\in
\widetilde{V}$, we have
\[\tilde{c}_v = c_v  \mbox{ and } R^{\widetilde{G}}_{\mathrm{eff}}(u, v) = R_{\mathrm{eff}}(u, v)\,.\]
We call $\widetilde{G}(\widetilde{V})$ the \emph{reduced network}.
Furthermore, if $\widetilde{V} = V \setminus \{x\}$, we then have
the formula
\begin{equation}\label{eq-network-reduction} \tilde{c}_{yz} = c_{yz}
+ c^{*,x}_{yz}\,, \mbox{ where } c^{*,x}_{yz} =
\frac{c_{xy}c_{xz}}{\sum_{w\in V_x} c_{xw}}\,.\end{equation}
\end{lemma}
\begin{proof}
Let $P$ be the transition kernel of the discrete-time random walk $\{S_t\}$
on the network $G$ and let $P^{\widetilde{V}}$ be the transition
kernel of the induced random walk on $\widetilde{V}$, namely for $u, v\in \widetilde{V}$
$$P^{\widetilde{V}} (u, v) = \P_u (T^+_{\widetilde{V}} = v)\,,$$
where $T_{A}^+ \deq \min\{t \geq 1: S_t \in A\}$ for all $A
\subseteq V$. In other words, $P^{\widetilde{V}}$ is the chain
watched in the subset $\widetilde{V}$. We observe that
$P^{\widetilde{V}}$ is a reversible Markov chain on $\widetilde{V}$
(see, e.g., \cite{AF, LPW09}). It is clear that the chain
$P^{\widetilde{V}}$ has the same invariant measure as that of $P$
restricted to $\widetilde{V}$, up to scaling by a constant. Therefore, there
exists a (unique) network $\widetilde{G}(\widetilde{V})$
corresponding to the Markov chain $P^{\widetilde{V}}$ such that
$\tilde{c}_u = c_u$ for all $u \in \widetilde{V}$.

We next show that the effective resistances are preserved in
$\widetilde{G}(\widetilde{V})$. To this end, we use the following
identity relating effective resistance and the random walk (see,
e.g., \cite[Eq. (2.5)]{LP}),
\begin{equation}\label{eq-walk-resistance}
\P_v(T_v^+ > T_u) = \frac{1}{c_v R_{\mathrm{eff}}(u,v)}\,,
\end{equation}
where $T_u = \min\{t\geq 0 : S_t = u\}$. Since
$P^{\widetilde{V}}$ is a watched chain on the subset
$\widetilde{V}$, we see that $\P_v^{\widetilde{V}}(T_v^+ >
T_u) = \P_v(T_v^+ > T_u)$ for all $u, v\in \widetilde{V}$.
This yields $R^{\widetilde{G}}_{\mathrm{eff}}(u, v) =
R_{\mathrm{eff}}(u, v)$.

To prove the second half of the lemma, we let
$\widetilde{G}(\widetilde{V})$ be the network defined by
\eqref{eq-network-reduction}. A straightforward calculation yields
that
\[\tilde{c}_v = c_v - c_{xv} + \sum_{y\in V_x} c^{*,x}_{vy} = c_v - c_{xv} + \sum_{y\in V_x}\frac{c_{xv}c_{xy}}{\sum_{z\in V_x} c_{xz}} = c_v\,.\]
Let $P^{\widetilde{G}}$ be the transition kernel for the random walk
on the network $\widetilde{G}(\widetilde{V})$.  Then,
\begin{equation*}P^{\widetilde{G}}(u, v) =
\frac{\tilde{c}_{uv}}{\tilde{c}_u} = \frac{c_{uv} +
\frac{c_{ux}c_{xv}}{\sum_{y\in V_x} c_{xy}}}{c_u}\,.\end{equation*}
On the other hand, the watched chain $P^{\widetilde{V}}$ satisfies
\[P^{\widetilde{V}} (u, v) = \frac{c_{uv}}{c_u} + \frac{c_{ux}}{c_u} \frac{c_{xv}}{\sum_{y\in V_x}c_{xy}}\,.\]
Altogether, we see that $P^{\widetilde{G}}(u, v) = P^{\widetilde{V}}
(u, v)$, completing the proof.
\end{proof}

\medskip
\noindent {\bf Well-separated sets.} The following result is an
important property of the resistance metric, crucial for our analysis.

\begin{prop}\label{prop:resistance}Consider a network $G(V)$ and its associated resistance metric $(V, R_{\mathrm{eff}})$. Suppose
that for some subset $S \subseteq V$, there is a partition $S = B_1
\cup B_2 \cup \cdots \cup B_m$ which satisfies the following
properties.
\begin{enumerate}
\item For all $i=1,2,\ldots, m$ and for all $x,y \in B_i$, we have $R_{\mathrm{eff}}(x,y) \leq \varepsilon/48$.
\item For all $i \neq j \in \{1,2,\ldots,m\}$, for all $x \in B_i$ and $y \in B_j$, we have $R_{\mathrm{eff}}(x,y) \geq \varepsilon.$
\end{enumerate}
Then there is a subset $I \subseteq \{1,2,\ldots,m\}$ with $|I| \geq
m/2$ such that for all $i \in I$,
$$
R_{\mathrm{eff}}(B_i, S \setminus B_i)  \geq \varepsilon/24.
$$
\end{prop}

In order to prove Proposition~\ref{prop:resistance}, we need
the following two
ingredients.

\begin{lemma}\label{lem-resistance-set}
Suppose the network $H(W)$ can be partitioned into two disjoint
parts $A$ and $B$ such that for some $\epsilon >0$, and some vertices
$u \in A$ and $v \in B$, we have
\begin{enumerate}
\item $R_{\mathrm{eff}}^H (u, v)\geq \epsilon$, and \label{item-network-1}
\item $R_{\mathrm{eff}}^H(u, x) \leq \epsilon / 12$ for all $x \in A$, and
$R_{\mathrm{eff}}^H(v, x) \leq \epsilon / 12$ for all $x \in B$.
\label{item-network-2}
\end{enumerate}
Then, $R_{\mathrm{eff}}^H (A, B) \geq \epsilon/6$.
\end{lemma}

\begin{proof}
Recall that by Thomson's Principle (see, e.g., \cite[Ch. 2.4]{LP}),
the effective resistance satisfies
\[R_{\mathrm{eff}}(x, y) = \min_{f} \mathcal{E}(f)\,, \mbox{ where }  \mathcal{E}(f) = \frac{1}{2} \sum_{x, y} f^2(x,y) r_{xy}\,,\]
and the minimum is over all unit flows from $x$ to $y$.
Here,
$r_{xy} = 1/c_{xy}$ is the edge resistance for $\{x,y\}$.

Suppose now
that $R_{\mathrm{eff}}^H (A, B) < \epsilon/6$. Then there exists a
unit flow $f_{AB}$ from set $A$ to set $B$ such that
$\mathcal{E}(f_{AB}) <\epsilon/6$. For $x \in A$, let $q_x$ be the
amount of flow sent out from vertex $x$ in $f_{AB}$ and for $x\in
B$, let $q_x$ be the amount of flow sent in to vertex $x$. Note that
$\sum_{x\in A} q_x = \sum_{x\in B} q_x = 1$.

Analogously, by
assumption \eqref{item-network-2}, there exist flows $\{f_{ux} : x \in A\}$
and $\{f_{xv} : x \in B \}$ such that $f_{xy}$ is a unit flow from $x$ to $y$ and
$\mathcal{E}(f_{xy}) \leq \epsilon/12$.
 We next build a flow $f$ such that
\[f = f_{AB} + \sum_{w\in A}q_w f_{uw} + \sum_{z \in B} q_z f_{zv}\,.\]
We see that $f$ is indeed a unit flow from $u$ to $v$. Furthermore,
by Cauchy-Schwartz,
\begin{align*}
\mathcal{E}(f) & = \frac{1}{2}\sum_{x,y} f^2(x, y)r_{xy} \\
&=
\frac{1}{2} \sum_{x, y} r_{xy}\left(f_{AB}(x, y) + \sum_{w\in
A}q_{w} f_{uw}(x, y)
+ \sum_{z \in B}q_{z} f_{zv}(x, y)\right)^2 \\
&\leq \frac{3}{2} \sum_{x, y}r_{xy} \left(f_{AB}^2(x, y) +
\sum_{w\in A} q_{w} f_{uw}^2(x,
y) + \sum_{z\in B }q_{z} f_{zv}^2(x, y)\right)\\
& = 3 \left(\mathcal{E}(f_{AB}) + \sum_{w\in
A}q_{w}\mathcal{E}(f_{uw}) + \sum_{z\in B}
q_{z}\mathcal{E}(f_{zv})\right) \\ &<
\epsilon \,.
\end{align*}
This contradicts assumption \eqref{item-network-1}, completing the
proof.
\end{proof}

\begin{lemma}\label{lem:resist-sep}
For any network $G(V)$, the following holds. If there is a subset $S
\subseteq V$ and a value $\varepsilon > 0$ such that $\reff(u,v)
\geq \varepsilon$ for all $u,v \in S$, then there is a subset $S'
\subseteq S$ with $|S'| \geq |S|/2$ such that for every $v \in S'$,
$$
\reff(v,S \setminus \{v\}) \geq \varepsilon/4.
$$
\end{lemma}
\begin{proof}
Consider the reduced network $\widetilde{G}$ on the vertex set $S$,
as defined in Lemma~\ref{lem-network-reduction}. Let the new
conductances be denoted $\tilde{c}_{xy}$ for $x,y \in S$. By
Lemma~\ref{lem-network-reduction}, our initial assumption that
$\reff(u,v) \geq \varepsilon$ for all $u,v \in S$ implies that
$R_{\mathrm{eff}}^{\widetilde{G}}(u,v) \geq \varepsilon$ for all
$u,v \in S$.

Let $n = |S|$. Foster's Theorem \cite{Foster48} (see also
\cite{Tetali91}) states that
\[\frac{1}{2}\sum_{u\neq v\in S}R_{\mathrm{eff}}^{\widetilde{G}}(u, v)\tilde{c}_{u, v} = n-1\,.\]
Combined with the fact that $R^{\widetilde{G}}_{\mathrm{eff}}(u, v)
\geq \varepsilon$,  this yields
\[\frac{1}{2}\sum_{u \neq v \in S}
\tilde{c}_{uv} \leq \frac{n}{\varepsilon}\,.\] In particular, there
exists a subset $S' \subseteq S$ with $|S'| \geq n/2$ such that for
all $v \in S'$,
$$
\sum_{u \in S \setminus \{v\}} \tilde{c}_{uv} \leq
\frac{4}{\varepsilon}.
$$
It follows that for every $v \in S'$, we have
$\ceff^{\widetilde{G}}(v, S \setminus \{v\}) \leq 4/\varepsilon$,
hence $$\reff(v,S\setminus \{v\}) = \reff^{\widetilde{G}}(v,S
\setminus \{v\}) \geq \varepsilon/4.$$
\end{proof}

\begin{proof}[Proof of Proposition \ref{prop:resistance}]
For each $i \in \{1,2,\ldots, m\}$, choose some $v_i \in B_i$.
By assumption (2), $\reff(v_i, v_j) \geq \e$ for $i \neq j$.
Thus applying
Lemma \ref{lem:resist-sep}, we find a subset $I \subseteq \{1,2,\ldots,m\}$
with $|I| \geq m/2$ and such that for every $i \in I$, we have
\begin{equation}\label{eq:fix1}
\reff(v_i, \{v_1, \ldots, v_m\} \setminus \{v_i\}) \geq \e/4\,.
\end{equation}
We claim that this subset $I$ satisfies the conclusion of the proposition.

To this end, fix $i \in I$, and let $\tilde{G}$ be the quotient network formed
by gluing $\{v_1, \ldots, v_m\} \setminus \{v_i\}$ into a single vertex $\tilde{v}$.
By \eqref{eq:fix1}, we have $\reff^{\tilde{G}}(v_i, \tilde{v}) \geq \e/4$.
Now let, $$\tilde{B} = \left(\{ \tilde v \} \cup \bigcup_{j \neq i} B_j\right) \setminus \{v_i\}_{i \in I}\,.$$

Consider any $x \in \tilde{B}$ with $x \neq \tilde{v}$.  Then $x \in B_j$ for some $j \neq i$, hence
by assumption (1), we conclude that,
$$
\reff^{\tilde{G}}(x, \tilde{v}) \leq \reff(x, v_j) \leq \e/48\,.
$$
We may now apply Lemma~\ref{lem-resistance-set} to the sets $B_i$ and $\tilde B$ in $\tilde G$
(with respective vertices $v_i$ and $\tilde v$) to conclude that
$$
\reff^{\tilde G}(B_i, \tilde B) \geq \e/24\,.
$$
But the preceding line immediately yields,
$$
\reff(B_i, S \setminus B_i) \geq \e/24,
$$
finishing the proof.
\remove{Applying Lemma~\ref{lem-resistance-set} to the sets $B_i$ and $B_j$
for each pair $i\neq j$, we obtain that $R_{\mathrm{eff}} (B_i, B_j)
\geq \epsilon / 6$. Now we consider the quotient graph $\widetilde{G}$
of $G$, in which each ball $B_i$ of $G$ is glued together and viewed
as a single vertex $v_i$. Therefore,
$R_{\mathrm{eff}}^{\widetilde{G}}(v_i, v_j) \geq \epsilon /6$ for all $i
\neq j$. Now applying Lemma~\ref{lem:resist-sep} to the set
$\widetilde{S} = \{v_1, \ldots, v_m\}$ on the network $\widetilde{G}$, we
obtain a subset $I \subset \{1, \ldots, m\}$ with $|I| \geq
m/2$ such that for all $i \in I$ \[R_{\mathrm{eff}}^{\widetilde{G}}
(v_i, \tilde{S} \setminus \{v_i\}) \geq \epsilon /24\,.\] It implies
that $R_{\mathrm{eff}}(B_i, S \setminus B_i) \geq \epsilon /24$ for
all $i \in I$, as required.}
\end{proof}

We end this section with the following simple lemma.

\begin{lemma}\label{lem-three-sets}
For any network $G(V)$, if $A,B_1, B_2 \subseteq V$ are disjoint,
then
$$R_{\mathrm{eff}}(A, B_1 \cup B_2) \geq \frac{R_{\mathrm{eff}}(A, B_1)\cdot R_{\mathrm{eff}}(A, B_2)}{R_{\mathrm{eff}}(A, B_1) + R_{\mathrm{eff}}(A, B_2)}\,.$$
\end{lemma}
\begin{proof}
By considering the quotient graph, the lemma can be reduced to the
case when $A = \{u\}$. Let $\{S_t\}$ be the discrete-time random walk on
the network and define
\[T_B = \min\{t \geq 0: S_t \in B\} \mbox{ and } T_B^+ = \min\{t\geq 1: S_t \in B\} \mbox { for } B\subseteq V\,.\]
It is clear that for a random walk started at $u$, we have
\[\P_u (T_u^+ > T_{B_1 \cup B_2}) \leq \P_u (T_u^+ > T_{B_1}) +\P_u (T_u^+ > T_{B_2})\,.\]
Combined with \eqref{eq-walk-resistance}, this gives
$$\frac{1}{R_{\mathrm{eff}}(u, B_1 \cup B_2)} \leq \frac{1}{R_{\mathrm{eff}}(u, B_1)} + \frac{1}{R_{\mathrm{eff}}(u, B_2)}\,,$$
yielding the desired inequality.
\end{proof}



\subsection{The Gaussian free field}
\label{sec:GFF}

We
recall the graph Laplacian $\Delta : \ell^2(V) \to \ell^2(V)$
defined by
$$
\Delta f(x) = c_x f(x) - \sum_{y} c_{xy} f(y).
$$

Consider a connected network $G(V)$. Fix a vertex $v_0 \in V$, and
consider the random process $\mathcal X = \{\eta_v\}_{v \in V}$,
where $\eta_{v_0}=0$, and $\mathcal X$ has density proportional to
\begin{equation}\label{eq:density}
\exp\left(-\frac12 \langle \mathcal X, \Delta \mathcal X \rangle
\right) = \exp\left(-\frac14 \sum_{u,v} c_{uv} |\eta_u-\eta_v|^2 \right).
\end{equation}
The process $\mathcal X$ is called the Gaussian free field
(GFF) associated with $G$. The next lemma is known, see, e.g.,
Theorem 9.20 of \cite{Janson97}. We include the proof for
completeness.

\begin{lemma}\label{lem:resistGFF}
For any connected network $G(V)$, if $\mathcal X= \{\eta_v\}_{v \in V}$
is the associated GFF, then for all $u,v \in V$,
\begin{equation}\label{eq:resistGFF}
\E\, (\eta_u-\eta_v)^2 = \reff(u,v).
\end{equation}
\end{lemma}

\begin{proof}
From \eqref{eq:density}, and the fact that the Laplacian is positive
semi-definite, it is clear that $\mathcal X$ is a Gaussian process.
Let $\Gamma_{v_0}(u,v) = \E_u L_{T_0}^v$, where $T_0$ is the hitting
time for $v_0$ as in Theorem~\ref{thm:rayknight}. From Lemma
\ref{lem-covarance-resistance}, we have
\begin{equation}\label{eq:gammadef}
\Gamma_{v_0}(u,v) = \frac12\left(\reff(v_0,u) + \reff(v_0,v) -
\reff(u,v)\right).
\end{equation}

Let $\widetilde{\Delta}$ and $\widetilde{\Gamma}_{v_0}$,
respectively, be the matrices $\Delta$ and $\Gamma_{v_0}$ with the
row and column corresponding to $v_0$ removed. Appealing to
\eqref{eq:density}, if we can show that $\widetilde{\Delta}
\widetilde{\Gamma}_{v_0}=I$, it follows that $\Gamma_{v_0}$ is the
covariance matrix for $\mathcal X$.  In this case, comparing
\eqref{eq:gammadef} to
$$
\E(\eta_u \eta_v) = \frac12 \left(\E \eta_u^2 + \E \eta_v^2 - \E
(\eta_u-\eta_v)^2\right)
$$
and using $\eta_{v_0}=0$, we see that \eqref{eq:resistGFF} follows.

In order to demonstrate $\widetilde{\Delta}
\widetilde{\Gamma}_{v_0}=I$, we consider $u, v$ such that $v_0
\notin \{u,v\}$. Conditioning on the first step of the walk from $u$ gives,
\begin{eqnarray}
c_u \Gamma_{v_0}(u,v) = c_u \E_u L^v_{T_0} &=& \1_{\{u=v\}} +  \sum_{w} c_{uw} \E_{w} L^v_{T_0} \nonumber \\
&=& \1_{\{u=v\}} + \sum_{w} c_{uw} \Gamma_{v_0}(v,w)
\label{eq:laplace1}
\end{eqnarray}
On the other hand, by definition of the Laplacian,
$$
(\Delta \Gamma_{v_0})(u,v) = c_u \Gamma_{v_0}(u,v) - \sum_w c_{uw}
\Gamma_{v_0}(v,w) = \1_{\{u=v\}},
$$
where the latter equality is precisely \eqref{eq:laplace1}. Thus
$\widetilde{\Delta} \widetilde{\Gamma}_{v_0}=I$, completing the
proof.
\end{proof}

\medskip
\noindent{\bf A geometric identity.} In what follows, for a set of
points $Y$ lying in some Hilbert space, we use $\CH(Y)$ to denote
their affine hull, i.e. the closure of $\{\sum_{i=1}^n \alpha_i y_i:
n\geq 1, y_i \in Y, \sum_{i=1}^n \alpha_i = 1\}$. Of course, when
$Y$ contains the origin, $\CH(Y)$ is simply the linear span of $Y$.

\begin{lemma}\label{lem-identity}
For any network $G(V)$, if $\mathcal X = \{\eta_v\}_{v \in V}$ is the
GFF associated with $G$, then for any $w\in V$ and subset
$S\subseteq V$,
$$
\sqrt{\reff(w , S)} = \dist_{L^2}\left(\eta_w, \CH(\{\eta_u\}_{u \in
S})\right).
$$
\end{lemma}

\begin{proof}
Since the statement of the lemma is invariant under translation, we
may assume that the GFF is defined with respect to some $v_0 \in
S$.

In this case, by the definition in \eqref{eq:density}, the GFF for
$G/S$ has density proportional to
$$
\exp\left(-\frac14 \left( \sum_{u,v \notin S} c_{uv} |\eta_u-\eta_v|^2 +
\sum_{u \notin S} c_{v_S u} |\eta_u|^2\right) \right),
$$
i.e. the GFF on $G/S$ is precisely the initial Gaussian process
$\mathcal X$ conditioned on the linear subspace $A_S = \{\eta_v =
\eta_{v_0} = 0 : v \in S \}$.

Using \eqref{eq:reffsub} and Lemma \ref{lem:resistGFF}, we have
$$
\reff(w,S) = \reff^{G/S}(w, v_S) = \E\left[|\eta_w-\eta_{v_0}|^2 \,\big|\,
A_S\right] = \E\left[|\eta_w|^2 \,\big|\, A_S\right].
$$
To compute the latter expectation, write $\eta_w = Y + Y'$, where $Y'
\in \mathrm{span}(\{\eta_v\}_{v \in S})$ and $\E (Y Y') = 0$. It
follows immediately that
$$
\dist_{L^2}\left(\eta_w, \CH(\{\eta_u\}_{u \in S})\right) = \sqrt{\E[Y^2]}
= \sqrt{\E\left[|\eta_w|^2 \,\big|\, A_S\right]},
$$
completing the proof.
\end{proof}

\section{Majorizing measures}
\label{sec:mm}

We now review the relevant parts of the majorizing measure theory.
One is encouraged to consult the book \cite{Talagrand05} for further
information.
In Section \ref{sec:intro}, we saw Talagrand's $\gamma_2$ functional.
For our purposes, it will be more convenient to work with
a different value that is equivalent to the functional $\gamma_2$,
up to universal constants.
In Section \ref{sec:septrees}, we discuss
separated trees, and prove
a number of standard properties about such objects.
In Section \ref{sec:algorithm},
we present a deterministic algorithm for
computing $\gamma_2(X,d)$ for any finite
metric space $(X,d)$.
Finally, in Section \ref{sec:treelike},
we specialize the theory of
Gaussian processes and trees to the case
of GFFs.  There, we will use the geometric
properties proved in
Sections \ref{sec:resist} and \ref{sec:GFF}.

\medskip

Before we begin, we attempt to give some rough intuition
about the role of trees in the majorizing measures theory.
A good reference for this material is \cite{GZ03}.
 A {\em tree of subsets of $X$} is a finite collection $\mathcal F$ of
subsets with the property that for all $A,B \in \mc F$, either $A
\cap B = \emptyset$, or $A \subseteq B$, or $B \subseteq A$. A set
$B$ is a {\em child} of $A$ if $B \subseteq A$, $B \neq A$, and
$$
C \in \mathcal F, B \subseteq C \subseteq A \implies C=B \textrm{ or
} C=A.
$$
We assume that $X \in \mathcal F$, and $X$ is referred to as the
root of the tree $\mathcal F$. To each $A \in \mathcal F$, we use
$N(A)$ to denote the number of children of $A$. A {\em branch of
$\mathcal F$} is a sequence $A_1 \supset A_2 \supset \cdots$ such
that each $A_{k+1}$ is a child of $A_k$. A branch is {\em maximal}
if it is not contained in a longer branch.
We will assume additionally that every maximal branch
terminates in a singleton set $\{x\}$ for $x \in X$.

Let $\{\eta_x\}_{x \in X}$ be a centered Gaussian process with $X$ finite, and let $d(x,y) = \sqrt{\E \,(\eta_x - \eta_y)^2}$.
The basic premise of the tree interpretation of the majorizing measures
theory is that one can assign a measure of ``size'' to any tree
of subsets in $X$, and this size provides a lower bound on $\E \sup_{x \in X} \eta_x$.
The majorizing measures theorem then claims that the value of the optimal such tree
is within absolute constants of the expected supremum.
The size of the tree (see \eqref{eq:size}) can
be defined using only the metric structure of $(X,d)$, without reference
to the underlying Gaussian process.  Thus much of the theorems in this section
are stated for general metric spaces.

The tree of subsets is meant to capture the structure of $(X,d)$ at all
scales simultaneously.  In general, to obtain a multi-scale lower bound
on the expected supremum of the process, one arranges
so that the diameter of the subsets decreases exponentially as one goes down
the tree, and all subsets at one level of the tree are separated
by a constant fraction of their diameter (see Definitions \ref{def:qrtree} and \ref{def:septree} below).
This allows a certain level of independence between different
branches of the tree which is exploited in the lower bounds.
Much of this section
is devoted to proving that one can construct a near-optimal tree
with a number of regularity properties that will be crucial
to our approach in Section \ref{sec:cover}.

\subsection{Trees, measures, and functionals}

Let $(X,d)$ be an arbitrary metric space.

\begin{defn}\label{def:qrtree}
For values $q \in \mathbb N$ and $\alpha,\beta > 0$, and $r \geq 2$,
a tree of subsets $\mc F$ in $X$ is called
a {\em $(q,r,\alpha,\beta)$-tree} if to each $A \in \mc F$, one can
associate a number $n(A) \in \mathbb Z$ such that the following
three conditions are satisfied.
\begin{enumerate}
\item For all children $B$ of $A$, we have $n(B) \leq n(A) - q$.
\item If $B$ and $B'$ are two distinct children of $A$, then $d(B,B') \geq \beta \, r^{n(A) - 1}.$
\item $\diam(A) \leq  \alpha \, r^{n(A)}.$
\end{enumerate}
We will refer to a $(q,r,4,\frac12)$-tree as simply a {\em $(q,r)$-tree.}
\end{defn}

The {\em $r$-size of a tree of subsets $\mc F$,} written
$\mathsf{size}_r(\mc F)$, is defined as the infimum of
\begin{equation}\label{eq:size}
\sum_{k \geq 1} r^{n(A_k)} \sqrt{\log^{+} N(A_k)}
\end{equation}
over all possible maximal branches of $\mathcal F$, where we use the
notation $\log^{+} x = \log x$ for $x \neq 0$, and $\log^{+}(0)=0$.

\remove{
\begin{theorem}\label{thm:move}
For every $\alpha, \beta, \delta > 0$ and $q \in \mathbb N$,
there exists a value $r_0 = r_0(\alpha,\beta,q,\delta) \geq 2$
and a value $C = C(\alpha,\beta,q,\delta) > 0$
such that the following holds.
For every $r \geq r_0$, there is an $r' \geq r$ such that
for every metric space $(X,d)$, we have
$$
\sup \{ \mathsf{size}_r(\mathcal F) : \mathcal F \textrm{ is a $(q,r,\alpha,\beta)$-tree in $X$} \}
\geq C^{-1} \sup \{ \mathsf{size}_{r'}(\mathcal F) : \mathcal F
\textrm{ is a $(1,r',1,\delta)$-tree in $X$} \}\,.
$$
\end{theorem}
}

To connect trees of subsets with the $\gamma_2$ functional,
we recall the relationship with majorizing measures.
The next result is from \cite[Thm. 1.1]{Tala01}

\begin{theorem}
\label{thm:ameasure}
For every metric space $(X,d)$, we have
$$
\gamma_2(X,d) \asymp \inf \sup_{x \in X} \int_{0}^{\infty} \left(\log \frac{1}{\mu(B(x,\varepsilon))}\right)^{1/2}\,d\varepsilon,
$$
where $B(x,\varepsilon)$ is the closed
ball of radius $\varepsilon$ about $x$,
and the infimum is over all finitely supported probability measures on $X$.
\end{theorem}

We will also need the following theorem due to Talagrand (see Proposition
4.3 of \cite{Talagrand96} and also Theorem T5 of \cite{GZ03}.)
We will employ it now and also in
Section \ref{sec:algorithm}.

\begin{theorem}\label{thm:T5}
There is a value $r_0 \geq 2$ such that the following holds.
Let $(X,d)$ be a finite metric space, and $r \geq r_0$.
Assume there
is a family of functions $\{\f_i : X \to \mathbb R_+ : i \in \mathbb Z\}$
such that the following conditions hold for some $\beta > 0$.
\begin{enumerate}
\item $\varphi_i(x) \geq \varphi_{i-1}(x)$ for all $i \in \mathbb Z$ and $x \in X$.
\item If $t_1, t_2, \ldots, t_N \in B(s,r^j)$ are such that $d(t_i, t_{i'}) \geq r^{j-1}$
for $i \neq i'$, then
$$
\f_j(s) \geq \beta r^{j} \sqrt{\log N} + \min \left\{ \f_{j-2}(t_i) : i = 1,2,\ldots,N\right\}.
$$
Under these conditions,
$$
\gamma_2(X,d) \lesssim_{r, \beta} \sup_{x \in X, i \in \mathbb Z} \f_i(x).
$$
\end{enumerate}
\end{theorem}

The preceding two theorems allow us to present the following connection
between trees and $\gamma_2$.  Such a connection is well-known
(see, e.g. \cite{Talagrand95a}), but we record the proofs
here for completeness, and for the precise
quantitative bounds we will use in future sections.

\begin{lemma}\label{lem:mmupper}
There is a value $r_0 \geq 2$ such that
for every finite metric space $(X,d)$, and every $r \geq r_0$, we have
\begin{equation}\label{eq:ourbound}
\gamma_2(X,d) \lesssim_r \sup \{ \mathsf{size}_r(\mathcal F) : \mathcal F \textrm{ is a
$(1,r,4,\tfrac1{2})$-tree in $X$}\}\,.
\end{equation}
\end{lemma}

\begin{proof}
First, for a subset $S \subseteq X$, let
$$
\theta(S) = \sup \{ \mathsf{size}_r(\mathcal F) : \mathcal F \textrm{ is a
$(1,r,4,\tfrac1{2})$-tree in $X$}\}\,.
$$
Then define, for every $i \in \mathbb Z$ and $x \in X$, define
$$
\varphi_i(x) = \theta(B(x, 2 r^i))\,.
$$
where $B(x,R)$ is the closed ball of radius $R$ about $x \in X$.
We now wish to verify that the conditions of Theorem \ref{thm:T5} hold for $\{\varphi_i\}$.
Condition (1) is immediate.

Assume that $r \geq 8$.
Given $t_1, t_2, \ldots, t_N$ as in condition (2) of Theorem \ref{thm:T5},
consider the set $A = B(s, 2 r^j)$ which has diameter bounded by $4 r^j$,
and the disjoint subset sets of $A$ given by
 $A_i = B(t_i, 2 r^{j-2})$ which each have diameter bounded by $4 r^{j-2}$,
 and which satisfy $d(A_i, A_j) \geq r^{j-1}/2$ for $i \neq j$.
We also have $A_i \subseteq A$ for each $i \in \{1,\ldots,N\}$.

Taking the tree of subsets with root $A$, $n(A)=j$, and children $\{A_i\}_{i=1}^N$, and in each $A_i$ a
tree which achieves value at least $\theta(A_i) = \theta(B(t_i, 2 r^{j-2})) = \varphi_{j-2}(i)$,
we see immediately that
$$
\varphi_j(s) = \theta(B(s,2 r^j)) \geq r^j \sqrt{\log N} + \min \{\varphi_{j-2}(t_i) : i=1,2,\ldots,N \},
$$
confirming condition (2) of Theorem \ref{thm:T5}.
Applying the theorem, it follows that $\gamma_2(X,d) \lesssim_r \theta(X)$,
proving \eqref{eq:ourbound}.
\end{proof}

We will need the upper bound \eqref{eq:ourbound} to hold for
$(2,r,4,\tfrac12)$-trees.  Toward this end, we state a version
of \cite[Thm 3.1]{Talagrand95a}.
The theorem there is only proved for $\alpha=1$ and $\beta=\frac12$,
but it is straightforward to see that it works for all values $\alpha,\beta > 0$
since the proof merely proceeds by choosing an appropriate subtree of the given tree;
the values $\alpha$ and $\beta$ are not used.

\begin{theorem}\label{thm-talagrand-95a}
For every metric space $(X,d)$, the following holds.
For every $\alpha, \beta, r>0$ and $q \in \mathbb N$, and for every
$(1, r, \alpha, \beta)$-tree $\mathcal F$ in $X$, there exists a
$(q, r, \alpha, \beta)$-tree $\mathcal{F}'$ in $X$ such that
$$\mathsf{size}_r(\mathcal{F}) \lesssim q \cdot \mathsf{size}_r (\mathcal{F}')\,.$$
\end{theorem}

Combining Theorem \ref{thm-talagrand-95a} with Lemma \ref{lem:mmupper} yields the following
upper bound using $(2,r)$-trees.

\begin{cor}\label{cor:mmupper}
There is a value $r_0 \geq 2$ such that
for every finite metric space $(X,d)$, and every $r \geq r_0$, we have
\begin{equation}
\gamma_2(X,d) \lesssim_r \sup \{ \mathsf{size}_r(\mathcal F) : \mathcal F \textrm{ is a
$(2,r,4,\tfrac1{2})$-tree in $X$}\}\,.
\end{equation}
\end{cor}

Now we move onto a lower bound on $\gamma_2$.

\begin{lemma}\label{lem:mmlower}
There is a value $r_0 \geq 2$ such that for every finite metric space $(X,d)$, and
every $r \geq r_0$, we have
$$
\gamma_2(X,d) \gtrsim \sup \{ \mathsf{size}_r(\mathcal F) : \mathcal F \textrm{ is a $(1,r,8,\tfrac{1}{6})$-tree} \}\,.
$$
\end{lemma}

\begin{proof}
We will show for any probability measure $\mu$ on $X$
and any $(1,r,8,\frac{1}{6})$-tree $\mathcal F$ in $X$, we have
$$
\mathsf{size}_r(\mathcal F) \lesssim_r \sup_{x \in X} \int_{0}^{\infty} \left(\log \frac{1}{\mu(B(x,\varepsilon))}\right)^{1/2}\,d\varepsilon\,.
$$
The basic idea is that if $A_1, A_2, \ldots A_k$ are children of $A$,
in $\mathcal F$, then the sets $B(A_i, \frac1{20} r^{n(A)-1})$ are disjoint
by property (2) of Definition \ref{def:qrtree}, where we write
$B(S,R) = \{ x \in X : d(x,S) \leq R \}$.
Thus one of these sets $A_i$ has $\mu(B(A_i, \frac{1}{20} r^{n(A)-1})) \leq 1/N(A)$.

Thus we may find a finite sequence of sets, starting with $A^{(0)}=X$ such that
$A^{(i+1)}$ is a child $A^{(i)}$ and $$\mu(B(A^{(i+1)},\tfrac{1}{20} r^{n(A^{(i)})-1})) \leq 1/N(A^{(i)}).$$
Since every maximal branch in a tree of subsets terminates in a singleton, the
sequence ends with some set $A' = A^{(h)} = \{x\}$.
By construction, we have
$$
\mu(B(x, \tfrac{1}{20} r^{n(A')-1})) \leq \frac{1}{N(A')}\,.
$$
Thus, assuming $r \geq 40$,
\begin{equation}\label{eq:aprime}
r^{n(A')-2} \sqrt{\log^{+} N(A')} \leq \int_{r^{n(A')-2}}^{\frac1{20} r^{n(A')-1}} \sqrt{\frac{1}{\log \mu(B(x,\varepsilon))}}\,d\varepsilon\,.
\end{equation}

By property of Definition \ref{def:qrtree}, the intervals $(r^{n(A)-2}, \frac{1}{20} r^{n(A)-1})$ are disjoint
for different sets $A \in \mathcal F$ with $x \in A$, thus summing \eqref{eq:aprime} yields
$$
\mathsf{size}_r(\mathcal F) \lesssim_r \sum_{A \in \mathcal F : x \in A} r^{n(A)-2} \sqrt{\log^{+} N(A)} \leq \int_{0}^{\infty} \sqrt{\frac{1}{\log \mu(B(x,\varepsilon))}}\,d\varepsilon\,,
$$
completing the proof.
\end{proof}

\subsection{Separated trees}
\label{sec:septrees}

Let $(X,d)$ be an arbitrary metric space.
Consider a finite, connected, graph-theoretic tree $\mcT=(V,E)$ (i.e., a connected, acyclic graph)
such that $V \subseteq X$,
with a fixed root $z \in V$,
and a mapping $s : V \to \mathbb Z$.
Abusing notation, we will sometimes use $\mcT$ for the vertex set of $\mcT$.
For a vertex $x \in \mcT$, we use $\mcT_x$ to denote the subtree rooted at $x$, and
we use $\Gamma(x)$ to denote the set of children\footnote{Formally, these are precisely
the neighbors of $x$ in $\mathcal T$ whose unique path to the root $z$ passes through $x$.} of $x$ with respect to the root $z$.
Finally, we write $\Delta(x)=|\Gamma(x)|+1$ for all $x \in \mcT$.

Let $\mathcal L$ be the set of leaves of $\mcT$. For any $v \in
\mcT$, let $\mc P(v) = \{z, \ldots, v\}$ denote the set of nodes on
the unique path from the root to $v$. For a pair of nodes $u,v \in
\mcT$, we use $\mc P(u,v)$ to denote the sequence of nodes on the
unique path from $u$ to $v$. If $u$ is the parent of $v$, we write
$u = p(v)$ and in particular we write $z = p(z)$.
For any such pair $(\mcT,s)$ and $r \geq 2$, we define the {\em value of $(\mc T,s)$}
by
\begin{equation}\label{eq:treevalue}
\val_r(\mcT,s) = \inf_{\ell \in \mathcal L} \sum_{v \in \mc P(\ell)} r^{s(v)} \sqrt{\log \Delta(v)}.
\end{equation}
The following definition will be central.

\begin{defn}\label{def:septree}
For a value $r \geq 2$, we say that the pair $(\mcT,s)$ is an {\em
$r$-separated tree in $(X, d)$} if it satisfies the following
conditions for all $x \in \mcT$.
\begin{enumerate}
\item For all $y \in \Gamma(x)$, $s(y) \leq s(x) -2$.
\item For all $u,v \in \Gamma(x)$, we have $d(x, \mcT_u) \geq \frac{1}{2}\, r^{s(x) - 1}$ and $d(\mcT_u, \mcT_v) \geq \frac{1}{2}\, r^{s(x)-1}$.
\item $\diam(\mc T_x) \leq 4 r^{s(x)}$.
\end{enumerate}
\end{defn}

We remark that our separated tree is a slightly different version of
the $(2,r)$-tree introduced in the preceding section.
The main difference is
that the nodes of our separated tree are point in the metric space
$X$, whereas a node in a $(2,r)$-tree is a subset of $X$. Our definition
is tailored for the application in Section \ref{sec:cover}.

Not surprisingly, we have a similar version of the
above theorem for separated trees.

\begin{theorem}\label{thm:packingtree}
For some $r_0 \geq 2$ and every $r \geq r_0$, and any metric space $(X, d)$, we have
$$
\sup_{\mc T} \val_r(\mcT,s) \asymp_r \gamma_2(X,d),
$$
where the supremum is over all $r$-separated trees in $X$.
\end{theorem}

Theorem \ref{thm:packingtree} follows
from Corollary \ref{cor:mmupper} and the following lemma.

\begin{lemma}
Consider $r\geq 8$ and any metric space $(X, d)$. For any $(2,r)$-tree
$\mc F$, there is an $r$-separated tree $\mc T$ such that
$\mathsf{size}_r(\mc F) = \val_r(\mc T)$. Also, for any
$r$-separated tree $\mc T$, there is a $(2,r)$-tree $\mc F$ such that
$\mathsf{size}_r(\mc F) \geq \val_r(\mc T) - r \,\diam(X)$.
\end{lemma}
\begin{proof}
We only prove the first half of the statement, since the second half
can be obtained by reversing the construction. The additive factor
$-r \,\diam(X)$ is due to the slight difference in the definitions of
the value for a separated tree and the size for a $(2,r)$-tree (see
\eqref{eq:treevalue} and \eqref{eq:size}).

Let $\mathcal F$ be a $(2,r)$-tree on $(X, d)$. For each $A \in \mathcal
F$ with $N(A) \geq 1$, we select one child $c(A)$ and an arbitrary
point $v_A \in c(A)$. We now construct the separated tree $\mathcal
T$.  Its vertex set is a subset of $\{v_A : A \in \mathcal F\}$. The
root of $\mcT$ is $v_X$, and its children are $\{v_B : B \textrm{ is
a child of $X$ with $B \neq c(X)$}\}$. In general, if $v_A$ is a
node of $\mathcal T$, then its children are the points $\{v_B : B
\textrm{ is a child of $A$ with $B \neq c(A)$}\}$. Finally, for $v_A
\in \mcT$, we put $s(v_A) = n(A)$.

Let us first verify that $\mcT$ is an $r$-separated tree. Condition
(1) of Definition \ref{def:septree} holds because if $y$ is a child
of $v_A \in \mcT$, then $y=v_B$ for some child $B$ of $A$ (in $\mc
F$), which implies $s(y)=n(B) \leq  n(A) - 2 = s(v_A) -2$. Secondly,
If $v_A$ is a node with children $v_{B_1}, v_{B_2}, \ldots,
v_{B_k}$, then clearly by Definition \ref{def:qrtree},
\begin{align*}
d(v_A, \mcT_{v_{B_i}}) &\geq d(c(A), B_i) \geq \frac{1}{2}\,r^{s(v_A) - 1}, \\
d(\mcT_{v_{B_i}}, \mcT_{v_{B_j}}) &\geq d(B_i, B_j) \geq
\frac{1}{2}\, r^{s(v_A)-1},
\end{align*}
verifying condition (2) of Definition \ref{def:septree}.

Thirdly, if $x_A \in \mcT$, then for any child $x_B$ of $x_A$, we
know $B$ is a child of $A$, hence $$\diam(\mcT_{x_B}) \leq \diam(B)
\leq 4 r^{n(A)}=  4 r^{s(x_A)},$$ using property (3) of a $q$-tree.
This verifies condition (3) of Definition \ref{def:septree}.

Finally, observe that for every non-leaf node $v_A \in \mcT$, we
have $\Delta(v_A)=|\Gamma(v_A)|+1=N(A)$, and for leaves, we have
$\log \Delta(v_A) = \log^{+} N(A) = 0$. It follows that
$\val_r(\mcT,s)=\mathsf{size}_r(\mc F)$, completing the proof.
\end{proof}

\subsubsection{Additional structure}

We now observe that we can take our separated trees
to have some additional properties.
Say that an $r$-separated tree $(\mcT,s)$ is
{\em $C$-regular} for some $C \geq 1$, if
it satisfies, for every $v \in \mcT \setminus \mathcal L$,
\begin{equation}\label{eq:regular}
\Delta(v) \geq \exp\left(\vphantom{\bigoplus} C^2 r^2 4^{s(z)-s(v)}\right).
\end{equation}

\begin{lemma}\label{lem:regular}
For every $C \geq 1$ and $r \geq 4$, for every $r$-separated tree
$(\mcT, s)$ in $X$, if $$\val_r(\mcT,s) \geq 4 C r^{s(z)+1},$$ then
there is a $C$-regular $r$-separated tree $(\mcT',s')$ in $X$ with
$$\tfrac12 \val_r(\mcT,s) \leq \val_r(\mcT',s') \leq \val_r(\mcT,s).$$
\end{lemma}

\begin{proof}
Consider the following operation on an $r$-separated tree $(\mcT,s)$.
For $x \in \mcT \setminus \mc L$, consider a new $r$-separated tree
$(\mcT', s') = \Phi_x (\mcT, s)$, which is defined as follows. Let
$u$ be the child of $x$ and let $S$ contain the remaining children
such that
\begin{equation}\label{eq-pick-u}
\val_r(\mcT_u, s|_{\mcT_u}) \leq \val_r(\mcT_v, s|_{\mcT_v}) \mbox{
for all } v\in S\,,
\end{equation}
where $\mcT_u$ is the subtree of $\mcT$ rooted at $u$ and containing
all its descendants, and $s|_{\mcT_u}$ is the restriction of $s$ on
the subtree $\mcT_u$. Consider the tree $\mcT'$ that results from
deleting all the nodes in $S$, as well as the subtrees under them,
and then contracting the edge $(x,u)$. We also put $s'(x)=s(u)$ and
$s'(y)=s(y)$ for all $y \in \mcT'$.

As long as there is a node $x \in \mcT \setminus \mc L$ which
violates \eqref{eq:regular} (for the current $(\mcT', s')$), we
iterate this procedure (namely, we replace $(\mcT',s')$ by $\Phi_x
(\mcT',s')$). It is clear that we end with a $C$-regular tree
$(\mcT',s')$. Note that different choices of $x$ at each stage will
lead to different outcomes, but the following proof shows that all
of them satisfy the required condition.

It is also straightforward to verify that for any $\ell \in \mathcal
L'$, we have
\begin{eqnarray*}
\sum_{v \in \mc P_{\mc T'}(\ell)} r^{s'(v)} \sqrt{\log \Delta_{\mc T'}(v)}
&\geq &
\sum_{v \in \mc P_{\mc T}(\ell)} r^{s(v)} \sqrt{\log \Delta_{\mc T}(v)}
- C r \sum_{v \in P_{\mc T}(\ell)} r^{s(v)} 2^{s(z)-s(v)} \\
&\geq &
\sum_{v \in \mc P_{\mc T}(\ell)} r^{s(v)} \sqrt{\log \Delta_{\mc T}(v)}
- C r^{s(z)+1} \sum_{k=0}^{\infty} 2^{2k} r^{-2k} \\
&\geq &
\sum_{v \in \mc P_{\mc T}(\ell)} r^{s(v)} \sqrt{\log \Delta_{\mc T}(v)}
- 2C r^{s(z)+1} \\
&\geq &
\val_r(\mcT,s) - 2C r^{s(z)+1} \\
&\geq &
\tfrac12 \val_r(\mcT,s).
\end{eqnarray*}
where in the second line we have used property (1) of Definition \ref{def:septree},
 in the third line, we have used $r \geq 4$, and
 in the final line we have used our assumption that $\val_r(\mcT,s) \geq 4C r^{s(z)+1}$.

It remains to prove that $\val_r(\mcT, s) \geq \val_r(\mcT', s')$.
The issue here is that it is possible $\mathcal L' \subsetneq
\mathcal{L}$. However, by our choice of $u$ at each stage (as in
equation \eqref{eq-pick-u}), it is guaranteed that $\ell\in \mathcal
L'$ for a certain $\ell\in \mathcal L$ such that $\val_r(\mcT,s) =
\sum_{v \in \mc P(\ell)} r^{s(v)} \sqrt{\log \Delta(v)}.$ This
completes the proof.
\end{proof}

We next study the subtrees of separated trees. In what follows, we
continue denoting by $s|_{\mcT'}$ the restriction of $s$ on $\mcT'$
for $\mcT' \subseteq \mcT$, and we use a subscript $\mcT'$ to refer
to the subtree $\mcT'$.

\begin{lemma}\label{lem:subtree}
For every $r$-separated tree $(\mcT,s)$, there
is a subtree $\mcT' \subseteq \mcT$ such that
$(\mcT',s|_{\mcT'})$ is an $r$-separated tree satisfying
the following conditions.
\begin{enumerate}
\item $\val_r(\mcT,s) \asymp \val_r(\mcT', s|_{\mcT'}).$
\item For every $v \in \mcT' \setminus \mc L_{\mc T'}$, $\Delta_{\mc T'}(v)=\Delta(v)$.
\item For every $v \in \mcT' \setminus \mc L_{\mcT'}$ and  $w \in \mc L_{\mc T'} \cap \mcT_v$,
\begin{equation}\label{eq:downtree}
\sum_{u\in \mathcal{P}(v, w)} r^{s(u)} \sqrt{\log \Delta_{\mc
T'}(u)} \geq \frac12 r^{s(p(v))} \sqrt{\log \Delta_{\mc T'}(p(v))}.
\end{equation}
\end{enumerate}
\end{lemma}
\begin{proof}
We construct the subtree $\mcT'$ in the following way. We examine
the vertices of $v\in \mcT$ in the breadth-first search order (that
is, we order the vertices such that their distances to the root are
non-decreasing). If $v$ is not deleted yet and for some $\ell \in
\mathcal{L} \cap \mcT_v$,
\begin{equation}\label{eq-delete}\sum_{u\in \mathcal{P}(v, \ell)}
r^{s(u)} \sqrt{\log \Delta_{\mc T}(u)} \leq r^{s(p(v))} \sqrt{\log
\Delta_{\mc T}(p(v))}\,,\end{equation} we delete all the descendants
of $v$. Let $\mcT'$ be the subtree obtained at the end of the
process. It is clear that $(\mcT', s|_{\mcT'})$ is a separated tree,
and it remains to verify the required properties.

By the construction of our subtree $\mc T'$, we see that whenever a
vertex is deleted, all its siblings are deleted. So for a
node $v\in \mc T' \setminus \mc L_{\mc T'} $, all the children in
$\mc T$ of $v$ are preserved in $\mc T'$, yielding property (2).

Note that if $v\in \mc L_{\mcT'} \setminus \mc L$, there exists
$\ell \in \mc L \cap \mc T_v$ such that \eqref{eq-delete} holds.
Therefore, we see
\[\sum_{u\in \mathcal{P}(z,
v) } r^{s(u)} \sqrt{\log \Delta_{\mc T'}(u)}  = \sum_{u\in
\mathcal{P}(z, v) \setminus \{v\} } r^{s(u)} \sqrt{\log \Delta_{\mc
T}(u)} \geq \frac{1}{2} \sum_{u \in \mathcal{P}(z, \ell) } r^{s(u)}
\sqrt{\log \Delta_{\mc T}(u)} \geq \frac{1}{2}\val_r(\mcT,s)\,.\] This
verifies property (1) (noting that the reverse inequality is
trivial).

Take $v\in \mc T' \setminus \mc L_{\mc T'}$ and $w \in \mc L_{\mc
T'} \cap \mcT_v$. If $w\in \mc L$, we see that \eqref{eq:downtree}
holds for $v$ and $w$ since \eqref{eq-delete} does not hold for $v$
and $\ell = w$ (otherwise all the descendants of $v$ have to be
deleted and $v$ will be a leaf node in $\mc T'$). If $w\not\in \mc
L$, there exists $\ell_0 \in \mc L \cap \mc T_{w}$ such that
$$\sum_{u\in \mathcal{P}(w,
\ell_0)} r^{s(u)} \sqrt{\log \Delta_{\mc T}(u)} \leq r^{s(p(w))}
\sqrt{\log \Delta_{\mc T}(p(w))}\,.$$ Recall that \eqref{eq-delete}
fails with $\ell = \ell_0$. Altogether, we conclude that
\begin{align*}\sum_{u\in \mathcal{P}(v, w)} r^{s(u)} \sqrt{\log \Delta_{\mc
T'}(u)}& = \sum_{u\in \mathcal{P}(v, \ell_0)} r^{s(u)} \sqrt{\log
\Delta_{\mc T}(u)}- \sum_{u\in \mathcal{P}(w, \ell_0)} r^{s(u)}
\sqrt{\log \Delta_{\mc T}(u)}\\
&\geq \frac{1}{2}  \sum_{u\in \mathcal{P}(v, \ell_0)}
r^{s(u)}\sqrt{\log \Delta_{\mc T}(u)} \\
&\geq \frac{1}{2} r^{s(p(v))}
\sqrt{\log \Delta_{\mc T}(p(v))}\,,
\end{align*}
establishing property (3) and completing the proof.
\end{proof}

Finally, we observe that separated trees are stable in the following
sense.

\begin{lemma}\label{lem:delete}
Fix $0 < \delta < 1$. Suppose that $(\mcT,s)$ is an $r$-separated
tree in $X$, and for every node $v \in V$, we delete all but $\lceil
\delta\cdot |\Gamma(v)| \rceil$ of its children. Denote by $\mcT'$
the induced tree on the connected component containing $z(\mcT)$.
Then $(\mcT', s|_{\mcT'})$ is an $r$-separated tree and
$$
\val_r(\mcT,s) \asymp_{\delta} \val_r(\mcT',s|_{\mcT'}).
$$
\end{lemma}

\begin{proof}
It is clear that Properties (1), (2) and (3) of separated trees are
preserved for the induced tree $\mcT'$ for $s|_{\mcT'}$. So $(\mcT',
s)$ is an $r$-separated tree. Furthermore, for every leaf $\ell$ of
$\mcT'$,
\begin{align*}\sum_{v\in \mathcal{P}(\ell)} r^{s(v)} \sqrt{\log
\Delta_{\mcT'}(v)} &\geq \sum_{v\in \mathcal{P}(\ell)} r^{s(v)}
\sqrt{\log (1 + \lceil \delta\cdot |\Gamma(v)| \rceil)}\\& \geq
c(\delta) \sum_{v\in \mathcal{P}(\ell)} r^{s(v)} \sqrt{\log (1 +
|\Gamma(v)|)} \geq c(\delta) \val_r(\mcT,s)\,,\end{align*} where
 $c(\delta)$
is a constant depending only on $\delta$. It follows that
$\val_r(\mcT',s|_{\mcT'}) \geq c(\delta) \val_r(\mcT,s)$, completing
the proof since the reverse direction is obvious.
\end{proof}

\subsection{Computing an approximation to $\gamma_2$ deterministically}
\label{sec:algorithm}

We now present a deterministic algorithm for computing an approximation to $\gamma_2$.

\begin{theorem}
Let $(X,d)$ be a finite metric space, with $n=|X|$.
If, for any two points $x,y \in X$, one can compute
$d(x,y)$ in time polynomial in $n$, then one can compute
a number
$A(X,d)$ in polynomial time, for which
$$
A(X,d) \asymp \gamma_2(X,d).
$$
\end{theorem}

\begin{proof}
Fix $r \geq 16$. First, let us
assume that $1 \leq d(x,y) \leq r^M$ for $x \neq y \in X$ and some
$M \in \mathbb N$. Fix $x_0 \in X$.

\medskip

Our algorithm constructs functions $\f_0, \f_1, \ldots, \f_M : X \to
\mathbb R_{+}$. We will return the value $A(X,d) = \f_M(x_0)$. First
put $\f_1(x)=\f_0(x)=0$ for all $x \in X$. Next, we show how to construct
$\f_j$ given $\f_0, \f_1, \ldots, \f_{j-1}$.

\medskip

For $x \in X$ and $r \geq 0$, we use $B(x,r) \deq \{ y \in X : d(x,y) \leq r \}$.
First, we construct a maximal $\frac13 r^{j-1}$ net $N_{j}$ in $X$ in the following way.
Supposing that $y_1, \ldots, y_k$ have already been chosen, let $y_{k+1}$
be a point satisfying
$$
\f_{j-2}(y_{k+1}) =
\max \left\{ \f_{j-2}(y) : y \in X \setminus \bigcup_{i=1}^{k} B\left(x, \frac13 r^{j-1}\right)\right\}\,,
$$
as long as there exists some point of $X \setminus \bigcup_{i=1}^{k} B(x, \frac13 r^{j-1})$ remaining.
For $x \in X$, set $$g_j(x) = y_{\min \{ k : d(x,y_k) \leq \frac13 r^{j-1} \}}.$$

Now we define $\f_j(x)$ for $x \in X$.
Suppose that $B(x, 2 r^j) \cap N_j = \{y_{\ell_1}, y_{\ell_2}, \ldots, y_{\ell_h}\},$
with $\ell_1 \leq \ell_2 \leq \cdots \leq \ell_h$, and define
\begin{enumerate}
\item[I.]
$\f_j(x) = \f_{j-1}(x)$ if $B(g_j(x), 4 r^{j}) \setminus B(g_j(x), \frac1{16} r^{j-2})$ is empty.
\item[II.]
Otherwise,
\begin{equation}\label{eq:phidef}
\f_j(x) = \max\left\{ \max_{k\leq h}\left( r^j \sqrt{\log k} + \min_{i\leq k}
\f_{j-2}(y_{\ell_i})\right), \max \{\f_{j-1}(z) : z \in B(x, \tfrac13 r^{j-1})\} \right\}.
\end{equation}
\end{enumerate}
Now, we verify that $\{\f_j\}_{j=0}^M$ satisfies the conditions of
Theorem \ref{thm:T5}. The monotonicity condition (1) is satisfied by
construction.
We will now verify condition (2), starting with the following lemma.

\begin{lemma}\label{lem:firstcase}
For any $j \geq 0$,
If $d(s,t) \leq r^j$ and
$B(g_j(s), 4 r^j) \setminus B(g_j(s), \frac{1}{16} r^{j-2})$ is empty,
then $\f_j(s)=\f_j(t)$.
\end{lemma}

\begin{proof}
We prove this by induction on $j$.
Clearly it holds vacuously for $j \leq 2$.
Assume that it holds for $\f_0, \f_1, \ldots, \f_{j-1}$ and $j \geq 2$.
By the condition of the lemma and the fact that $s \in B(g_j(s), \frac13 r^{j-1})$, we have
\begin{equation}\label{eq:distone}
d(s, g_j(s)) \leq \frac1{16} r^{j-2},
\end{equation}
which implies that $B(s, 2 r^j) \setminus B(s, \frac18 r^{j-2})$ is also empty.
Furthermore, we have $g_j(s)=g_{j}(t)$,
since otherwise
$d(g_j(s), g_j(t)) \geq \frac13 r^{j-1}$,
and we would conclude that
$$
2 r^j \geq d(g_j(t), s) \geq d(g_j(s), g_j(t)) - d(s, g_j(s)) \geq \frac13 r^{j-1} - \frac1{16} r^{j-2} \geq \frac{1}{8} r^{j-1},
$$
contradicting the fact that $B(s, 2 r^j) \setminus B(s, \frac18 r^{j-2})$ is empty.
It follows that
\begin{equation}\label{eq:empty}
B(s, 2 r^j) \setminus B(s, \tfrac18 r^{j-2}) = \emptyset \ \textrm{ and }\
B(t, 2 r^j) \setminus B(t, \tfrac18 r^{j-2}) = \emptyset\,.
\end{equation}

Since $g_j(s)=g_j(t)$, we conclude that both $\f_j(s)$ and $\f_j(t)$
are defined by case (I) above, hence
\begin{equation}\label{eq:equality}
\f_j(s)=\f_{j-1}(s) \ \textrm{ and } \ \f_j(t)=\f_{j-1}(t)\,.
\end{equation}

So we are done by induction unless $B(g_j(s), 4 r^{j-1}) \setminus B(g_j(s), \frac{1}{16} r^{j-3})$ is non-empty,
in which case $\f_{j-1}(s)$ and $\f_{j-1}(t)$ are defined by case (II).
But from \eqref{eq:empty} and $d(s,t) \leq r^j$, we see that
 $B(t, 2 r^{j-1}) = B(s, 2 r^{j-1})$ and
$B(s, \frac13 r^{j-2}) = B(t, \frac13 r^{j-2})$ as well.
This implies that $\f_{j-1}(s)$ and $\f_{j-1}(t)$ see the same maximization
in \eqref{eq:phidef}, hence $\f_{j-1}(s)=\f_{j-1}(t)$ and by \eqref{eq:equality} we are done.
\end{proof}

Now, let $s,t_1, \ldots, t_N \in X$ be as in
condition (2), and let $B(s, 2 r^j) \cap N_j = \left\{y_{\ell_1}, y_{\ell_2}, \ldots, y_{\ell_h}\right\}$
be such that
$\ell_1 \leq \ell_2 \leq \cdots \leq \ell_h$.
If $B(g_j(s), 4 r^j) \setminus B(g_j(s), \frac1{16} r^{j-1})$ is empty,
then $N=1$, and Lemma \ref{lem:firstcase}
implies that $\f_{j}(s) = \f_j(t_1) \geq \f_{j-2}(t_1)$, where
the latter inequality follows from monotonicity.
Thus we may assume that $\f_j(s)$ is defined by case (II).

\medskip

To every $t_i$, we can associate a distinct point $g_j(t_i) \in B(s, 2 r^j) \cap N_j$,
and by construction we have $\f_{j-2}(g_j(t_i)) \geq \f_{j-2}(t_i)$,
since $\f_{j-2}(y_k)$ is decreasing as $k$ increases.
Using this property again in conjunction
with the definition \eqref{eq:phidef},
we have
\begin{align*}
\f_j(s) &\geq r^j \sqrt{\log N} + \min \{ \f_{j-2}(y_{\ell_i}) :
i=1,\ldots,N \} \\
&\geq r^j \sqrt{\log N} + \min \{ \f_{j-2}(g_j(t_i)) : i=1,\ldots, N \} \\
&\geq r^j \sqrt{\log N} + \min \{ \f_{j-2}(t_i) : i=1,\ldots, N \},
\end{align*}
completing our verification of condition (2) of Theorem
\ref{thm:T5}. Applying Theorem \ref{thm:T5}, we see that
\begin{equation}\label{eq:upper}
\gamma_2(X,d) \lesssim \sup_{x \in X, i \in \mathbb Z} \f_i(x) =
\f_M(x_0) = A(X,d).
\end{equation}

To prove the matching lower bound, we first build a tree $\mcT$
whose vertex set is a subset of $X \times \mathbb Z$. The root of
$\mcT$ is $(x_0, M)$. In general, if $(x,j)$ is already a vertex of
$\mcT$ with $j \geq 1$, then we add children to $(x,j)$ according to
the maximizer of \eqref{eq:phidef}.  If $\f_j(x)=\f_{j-1}(z)$,
then we make $(z,j-1)$ the only child of
$(x,j)$. Otherwise, we put the nodes $(y_1, j-2), \ldots, (y_h,
j-2)$ as children of $(x,j)$, where $\{y_i\} \subseteq N_j$
are the nodes that achieve the maximum in \eqref{eq:phidef}.

Let the pair $(\mcT',s)$ be a constructed in the following way from $\mcT$.
We replace every maximal path of the form
$(x,j_0), (x,j_0-1), \ldots, (x,j_0-k)$
by the vertex $x$ and put $s(x)=j_0-k$.
It follows immediately by construction that
\begin{equation}\label{eq:Tprime}
\mathsf{val}_r(\mcT', s) \lesssim \f_M(x_0) + r\,\diam(X,d) \lesssim \f_M(x_0),
\end{equation}
where the latter inequality follows from \eqref{eq:upper},
since $\f_M(x_0) \gtrsim \gamma_2(X,d) \gtrsim \diam(X,d)$.
Note that the correction term of $\diam(X,d)$ in \eqref{eq:Tprime}
is simply because of the use of $\Delta(v)=|\Gamma(v)|+1$
in the definition \eqref{eq:treevalue}.

We next build a $(1,r,8,\frac1{16})$-tree $\mc F$, which essentially
captures the structure of the tree $\mc T$.  In general, the sets in
$\mc F$ will be balls in $X$, with the node $(x,j) \in \mcT$ being
associated with the set $B(x, 4 r^j)$ in $\mc F$, which will have
label $n(B(x, 4 r^j))=j$.

We construct the $(1,r,8,\frac1{16})$-tree $\mc F$ recursively.
The root of $\mc F$ is $B(x_0, 4 r^M)$ (which is equal to $X$), and we define $n(B(x,4 r^j))=M$.
In general, if $\mc F$ contains the set $B(x, 4 r^j)$ corresponding to the
node $(x, j) \in \mcT$, and if $(x, j)$ has children
$(y_1, j-2), (y_2, j-2), \ldots, (y_h, j-2) \in \mcT$,
we add the sets $B(y_i, 4 r^{j-2})$ as children of $B(x, 4 r^j)$ in $\mc F$,
with $n(B(y_i, 4 r^{j-2}))=j-2$.
Likewise, if $(z,j-1)$ is the child of $(x,j)$, then we add the set
$B(z, 4 r^{j-1})$ as the unique child of $B(x, 4 r^j)$ in $\mc F$
and put $n(B(z,4 r^{j-1}))=j-1$.
We continue in this manner until $\mc T$ is exhausted.

\medskip

We now verify that $\mc F$ is indeed a $(1,r,8,\frac1{6})$-tree.
First, note that if $(z, j-1)$ is a child of $(x, j)$ in $\mcT$, then
clearly $B(z, 4 r^{j-1}) \subseteq B(x, 4 r^j)$
since this can only happen if $d(x,z) \leq \frac13 r^{j-1}$.
Also, if $(y_1, j-2), \ldots, (y_h, j-2)$ are the
children of $(x, j)$, then
by the construction of the maps in \eqref{eq:phidef},
we have $d(y_i, x) \leq 2 r^j$, hence
$B(y_i, 4 r^{j-2}) \subseteq B(x, 4 r^j)$, recalling
that $r \geq 16$.
Furthermore, for $i \neq k$,
since $y_i, y_k \in N_j$,
we have $d(y_i, y_k) \geq \frac13 r^{j-1}$, so $B(y_i, 4 r^{j-2}) \cap B(y_k, 4 r^{j-2}) = \emptyset$,
verifying that $\mc F$ is indeed a tree of subsets.  In fact, we have the estimate
\begin{equation*}\label{eq:balldistance}
d\left(B(y_i, 4 r^{j-2}), B(y_k, 4 r^{j-2})\right) \geq \frac13 r^{j-1} - 8 r^{j-2} \geq \frac16 r^{j-1}
= \frac{1}{6} r^{n(B(x, 4 r^j))-1},
\end{equation*}
using $r \geq 16$.  This verifies that property (2) of a $(1,r,1,\frac1{6})$-tree is satisfied.
Furthermore,
property (1) of a $(1,r,8,\frac{1}{6})$-tree follows immediately by construction.
Finally, to verify property (3), note that
for any set in our tree of subsets $\mc F$, corresponding to a node of the form $(x,j) \in \mcT$,
we have $\diam(B(x, 4 r^j)) \leq 8 r^j$ and $n(B(x,4 r^j))=j$.

By construction, we have
$$
\mathsf{val}_r(\mcT',s) \lesssim \mathsf{size}_r(\mc F) + r \,\diam(X,d),
$$
and Lemma \ref{lem:mmlower} yields $\gamma_2(X,d) \gtrsim \mathsf{size}_r(\mc F) + \diam(X,d)$
(using $\gamma_2(X,d) \gtrsim \diam(X,d)$).
Combining this with \eqref{eq:Tprime}
shows that
$$
\gamma_2(X,d) \gtrsim \mathsf{val}_r(\mcT',s) \gtrsim \varphi_M(x_0) = A(X,d).
$$
Together with \eqref{eq:upper}, this shows that $\gamma_2(X,d) \asymp A(X,d)$.

\medskip

The only thing left is to remove the dependence of our running time
on $M$. But since there are at most $n^2$ distinct
distances in $(X,d)$, only $O(n^2)$ of the maps $\f_0,
\f_1, \ldots, \f_M$ are distinct.
More precisely, suppose that there is no pair $u,v \in X$
satisfying $d(u,v) \in [r^{j-3}, r^{j+1}]$ for some $j \in \mathbb Z$.
In that case, $\f_j(x)$ is defined by case (I) for all $x \in X$,
and thus $\f_j \equiv \f_{j-1}$.
Obviously, we may skip computation
of the intermediate non-distinct maps (and it is easy to see which
maps to skip by precomputing the values of $j$ such that there are
$u,v \in X$ with $d(u,v) \in [r^{j-3}, r^{j+1}]$.)
Since there are only $O(n^2)$ non-trivial values of $j$,
this completes the proof.
\end{proof}

\subsection{Tree-like properties of the Gaussian free field}
\label{sec:treelike}

Finally, we consider how the resistance metric (and hence the
Gaussian free field) allows us to obtain trees with special
properties. Consider a network $G(V)$, and the associated metric
space $(V,\sqrt{\reff})$. Let $(\mcT,s)$ be an $r$-separated tree in
$G$. We say that $(\mcT,s)$ is {\em strongly $r$-separated} if, for
every non-root node $v \in \mcT$, we have the inequality
\begin{equation}\label{eq-res-strong}
\sqrt{\reff(v, \mathcal{T} \setminus \mcT_v)} \geq \frac{1}{20}
r^{s(p(v)) -1},
\end{equation}
where $p(v)$ denotes the parent of $v$ in $\mcT$.

\begin{lemma}\label{lem-strong-tree-like}
For any network $G(V)$ and any $r \geq 96$, let $(\mcT_0, s)$ be an
arbitrary $r$-separated tree on the space $(V,
\sqrt{R_{\mathrm{eff}}})$. Then there is an induced strongly
$r$-separated tree $(\mcT,s)$ such that $|\Gamma_{\mc T}(v)| \geq
|\Gamma_{\mc T_0}(v)|/2$ for all $v\in \mc T \setminus \mc L_{\mc
T}$. Furthermore
\begin{equation}\label{eq:GFFass}
\val_r(\mcT,s) \asymp \val_r(\mcT_0,s).
\end{equation}
\end{lemma}
\begin{proof}
Consider any non-leaf node $v \in \mcT_0$ with children $c_1,
\ldots, c_k$, where $k \geq 1$. If $k=1$, let $S_v = \{c_1\}$.
Otherwise, we wish to apply Proposition~\ref{prop:resistance} to the
sets $\{\mcT_{c_i}\}_{i=1}^k$. By property (2) of separated trees, we get that for all
$x \in \mcT_{c_i}, y\in \mcT_{c_j}$ with $i\neq j$
\[R_{\mathrm{eff}}(x, y) \geq \left(\frac{1}{2} r^{s(v) -1}\right)^2 = \frac{1}{4} r^{2(s(v)-1)}\,.\]
Combined with property (3) of separated trees,
Proposition~\ref{prop:resistance} yields that there exists a subset
$S_v \subseteq \{c_1,\ldots,c_k\}$ with $|S_v| \geq k/2$ such that
for $c \in S_v$, we have
\begin{equation*}
\reff\left(\mcT_c, \mcT_{v} \setminus (\mcT_c \cup \{v\})\right)
\geq \frac{1}{4} r^{2 (s(v)-1)} \cdot \frac{1}{24} \geq \frac{1}{96}
r^{2 (s(v)-1)}\,.
\end{equation*}
Applying Lemma~\ref{lem-three-sets} with $A = \mcT_c, B_1 = \mcT_{v}
\setminus (\mcT_c \cup \{v\})$ and $B_2 = \{v\}$, we get that
\begin{equation}\label{eq:preprocess}
\reff\left(\mcT_c, \mcT_{v} \setminus \mcT_c\right) \geq
\frac{1}{100}\, r^{2(s(v)-1)}\,.
\end{equation}
Next, consider the induced $r$-separated tree $(\mcT, s)$ that
arises from deleting, for every non-leaf node $v \in \mcT_0$, all
the children not in $S_v$ as well as all their descendants. It is
clear that for all $v\in \mc T \setminus \mc L_{\mc T}$, we have
$|\Gamma_{\mc T}(v)| \geq |\Gamma_{\mc T_0}(v)|/2$. Lemma
\ref{lem:delete} then yields that $$\val_r(\mcT,s) \asymp
\val_r(\mcT_0,s).$$

It remains to verify that $(\mcT,s)$ is strongly $r$-separated.
Define $D_0 = 1$ and for $h \geq 1$,
$$
D_{h} = D_{h-1} \left(1 - D_{h-1}^2 r^{-4h}\right).
$$
It is straightforward to verify that $D_h \geq 1/2$ for all $h \geq
0$, since $r\geq 2$.

We now prove, by induction on the height of $\mcT$, that for every
node $u$ at depth $h \geq 1$ in $\mcT$,
\begin{equation}\label{eq:IH}
\sqrt{\reff\left(u, \mcT \setminus \mcT_u\right)} \geq \frac{1}{10}
r^{s(p(u)) - 1} D_{h-1}.
\end{equation}
By the preceding remarks,  this verifies \eqref{eq-res-strong},
completing the proof of the lemma.

Let $z=z(\mathcal T)$ be the root, and let $v$ be some child of $z$.
Let $u \in \mcT_v$ be a node at depth $h$ in $\mcT_v$ (and hence at
depth $h+1$ in $\mcT$). By \eqref{eq:preprocess}, we have
\begin{equation}\label{eq:con2}
\sqrt{\reff\left(u, \mcT \setminus \mcT_v\right)} \geq
\sqrt{\reff\left(\mcT_v, \mcT \setminus \mcT_v\right)} \geq
\frac{1}{10} r^{s(p(v)) -1}.
\end{equation}
If $u=v$, then the preceding inequality yields \eqref{eq:IH}.
Otherwise, $u \neq v$, and $h \geq 1$.

By the induction hypothesis \eqref{eq:IH} applied to $u$ and
$\mcT_v$, we have
\begin{equation}\label{eq:con1}
\sqrt{\reff\left(u, \mcT_v \setminus \mcT_u\right)} \geq
\frac{1}{10}r^{s(p(u)) -1 } D_{h-1}.
\end{equation}
Since $u \in \mathcal{T}_v$ is a node at depth $h$, we get from
property (1) of a separated tree that $s(p(v)) \geq s(p(u)) + 2h$
and therefore \begin{equation}\label{eq-s-u-v}\frac{1}{10}
r^{s(p(u))-1} D_{h-1} \leq r^{-2h} \cdot \frac{1}{10}
r^{s(p(v))-1}D_{h-1}\,.\end{equation} Now, using \eqref{eq:con2} and
\eqref{eq:con1}, we apply Lemma ~\ref{lem-three-sets} with $A=
\{u\}$, $B_1=\mcT_v \setminus \mcT_u$ and $B_2=\mcT \setminus
\mcT_v$, yielding
\begin{align*}
\sqrt{\reff\left(u, \mcT \setminus \mcT_u\right)} &\geq
\frac{\frac{1}{10} r^{s(p(u)) -1} D_{h-1} \cdot \frac{1}{10}
r^{s(p(v))-1}}{\sqrt{(\frac{1}{10} r^{s(p(u)) -1 } D_{h-1})^2 +
(\frac{1}{10} r^{s(p(v)) -1})^2}}\\
&\geq \frac{1}{10} r^{s(p(u)) -1} D_{h-1} \frac{1}{\sqrt{1+
(D_{h-1}r^{-2h})^2}} \\
& \geq \frac{1}{10} r^{s(p(u)) -1} D_{h-1} (1 -
D_{h-1}^2 r^{-4h}),
\end{align*}
where the second transition follows from \eqref{eq-s-u-v} and the
third transition follows from the fact that $(1 + x^2)^{-1/2} \geq 1
- x^2$. This completes the proof.
\end{proof}

\medskip
\noindent{\bf Good trees inside the GFF.} Consider a
Gaussian free field $\{\eta_x\}_{x \in V}$ corresponding to network
$G(V)$ with the associated metric space $(V,d)$, where $d(x, y) =
(\E (\eta_x -\eta_y)^2)^{1/2}$.
\begin{prop}\label{prop-strong-tree-like}
For some $r_0 \geq 2$ and any $r \geq r_0$ and $C \geq 1$, there
exists a constant $K=K(C,r)$ depending only on $C$ and $r$ such that
the following holds. For an arbitrary  Gaussian free field
$\{\eta_x\}_{x \in V}$ with $\gamma_2(V, d) \geq K\, \diam(V)$,
there exists an $r$-separated tree $(\mathcal{T}, s)$ with set of
leaves $\mathcal{L}$, such that the following properties hold.
\begin{enumerate}[(a)]
\item $\val_r(\mcT,s) \asymp_{r, C} \gamma_2(X,d)$.
\item For every $v \in V$, $\dist_{L^2}\left(\eta_v, \CH(\{\eta_u\}_{u \notin \mcT_v})\right)
\geq \frac{1}{20} r^{s(p(v))-1}$.
\item For every $v \in V$, $\Delta(v) \geq \exp\left(\vphantom{\bigoplus} C^2 r^2
4^{s(z)-s(v)}\right)$ for all $v \in \mathcal{T} \setminus
\mathcal{L}$.
\item For every $v \in \mcT \setminus \mc L$ and  $w \in \mc L \cap \mcT_v$,
\begin{equation*}
\sum_{u\in \mathcal{P}(v, w)} r^{s(u)} \sqrt{\log \Delta(u)} \geq
\frac12 r^{s(p(v))} \sqrt{\log \Delta(p(v))}.
\end{equation*}
\end{enumerate}
We call such a tree $\mathcal{T}$ a $C$-good $r$-separated tree.
\end{prop}
\begin{proof}
By definition of the GFF, we have $d = \sqrt{R_{\mathrm{eff}}}$ for
some network $G(V)$. Applying Theorem~\ref{thm:packingtree}, there
exists an $r$-separated tree $(\mc T_0, s_0)$ such that
$\val_r(\mcT_0,s_0) \asymp_{r} \gamma_2(V,d)$.

Recalling property (3) of Definition~\ref{def:septree} and the
assumption that $\gamma_2(V, d) \geq K\, \diam(V)$, we can then
select $K$ large enough such that the condition of
Lemma~\ref{lem:regular} is satisfied for the separated tree
$(\mcT_0, s_0)$. Then applying Lemma~\ref{lem:regular}, we can get a
$2C$-regular separated tree $(\mathcal{T}_1, s_1)$ with
$\val_r(\mcT_1,s_1) \asymp_{r, C}\val_r(\mcT_0,s_0)$.

At this point, using Lemma~\ref{lem-strong-tree-like}, we obtain a
$C$-regular strongly $r$-separated tree $(\mathcal{T}_2, s_2)$ such
that $\val_r(\mcT_2,s_2) \asymp_{r} \gamma_2(V,d)$. That is to say,
the tree $(\mc T_2, s_2)$ satisfies properties (a) and (c).
Furthermore, by Lemma~\ref{lem-identity}, we see that property (b)
holds for $(\mathcal{T}_2, s_2)$ because it is equivalent to the
strongly $r$-separated property \eqref{eq-res-strong}.

Finally, Lemma~\ref{lem:subtree} implies that there exists a subtree
$\mathcal{T} \subseteq \mathcal{T}_2$ with $\val_r(\mcT,s_2|_{\mc
T}) \asymp_{r, C}\val_r(\mcT_2,s_2)$ such that property (d) holds
for $\mathcal{T}$ and properties (a) and (c) are preserved (note
that by property (2) of Lemma~\ref{lem:subtree}, the degrees of
non-leaf nodes are preserved). Observe that property (b) is
preserved by taking subtrees. Writing $s = s_2|_{\mc T}$, we
conclude that the separated tree $(\mc T, s)$ satisfies all the
required properties, completing the proof.
\end{proof}

\section{The cover time}
\label{sec:cover}

We now turn to our main theorem.

\begin{theorem}\label{thm:covtime}
For any network $G(V)$ with total conductance $\mathcal C = \sum_{x \in V} c_x$, we have
$$
t_{\mathrm{cov}}^{\rc}(G) \asymp \mathcal C
\left[\gamma_2(V,\sqrt{\reff})\right]^2.
$$
\end{theorem}

Combined with Theorem \ref{thm:blanket}, this also yields a positive
answer to the strong conjecture of Winkler and Zuckerman
\cite{WZ96}.

\begin{cor}
For every $\delta \in (0,1)$, for any network $G(V)$ with total
conductance $\mathcal C = \sum_{x \in V} c_x$,
$$
t_{\mathrm{cov}}^{\rc}(G) \asymp \mathcal C
\left[\gamma_2(V,\sqrt{\reff})\right]^2 \asymp_{\delta}
t_{\mathrm{bl}}(G,\delta).
$$
\end{cor}

\medskip

For the remainder of this section, we denote
\begin{equation}\label{eq:thesup}
\mathfrak{S}  = \gamma_2(V,\sqrt{\reff}).
\end{equation}
It is clear that for all $0<\delta<1$, we have
$t_{\mathrm{cov}}^{\rc}(G) \leq t_{\mathrm{bl}}(G,\delta)$, and
$t_{\mathrm{bl}}(G,\delta) \lesssim_\delta \mathcal C
\mathfrak{S}^2$ by Theorem \ref{thm:blanket}. Thus, in order to
prove the preceding corollary and Theorem~\ref{thm:covtime}, we need
only show that
\begin{equation}\label{eq:covLB}
t_{\mathrm{cov}}^{\rc}(G) \gtrsim \mathcal C \mathfrak{S}^2.
\end{equation}

Let $\{W_t\}$ be the continuous-time random walk on $G(V)$,
and let $\{L_t^v\}_{v \in V}$ be the local times,
as defined in Section \ref{sec:iso}.
Applying the isomorphism theorem (Theorem \ref{thm:rayknight}) with some fixed $v_0 \in V$, we have
\begin{equation}\label{eq:isomorph}
\left\{L_{\tau(t)}^x + \frac{1}{2} \eta_x^2: x\in V\right\}
\stackrel{law}{=} \left\{\frac{1}{2} (\eta_x + \sqrt{2t})^2 : x\in V
\right\}\,,
\end{equation}
for some associated Gaussian process $\{\eta_x\}_{x \in V}$.
By Lemma \ref{lem:resistGFF}, this process is a  Gaussian free field,
and we have for every $x,y \in V$,
\begin{equation}\label{eq:gdist}
d(x,y) \deq \sqrt{\E\, |\eta_x - \eta_y|^2} = \sqrt{\reff(x,y)}.
\end{equation}
Let $\mathfrak{D} = \max_{x,y \in V} d(x,y)$ be
the diameter of the Gaussian process.

\bigskip
\noindent
{\bf Proof outline.}
Let $\{\mathfrak L > 0\}$ be the event $\{L_{\tau(t)}^x > 0 : x
\in V \}$.  Consider a set $S \subseteq \mathbb R^V$,
and let $S_L$ and $S_R$ be the events corresponding
to the left and right-hand sides of
\eqref{eq:isomorph} falling into $S$.
Our goal is to find such a set $S$ so that
for some $t \asymp \mathfrak{S}^2,$ we have
\begin{equation}\label{eq:outline}
\P(S_R) - \P(S_L \cap \{\mathfrak L > 0\}) \geq c,
\end{equation}
for some universal constant $c > 0$. In this case, with probability at least $c$,
the set of uncovered vertices
$\{ v : L_{\tau(t)}^v = 0\}$ is non-empty.
Using the fact
that the inverse local time $\tau(t)$ is $\gtrsim \mathcal C t$ with
probability at least $1-c/2$, we will conclude that
$t_{\mathrm{cov}}^{\rc}(G) \gtrsim \mathcal C \mathfrak{S}^2.$

\medskip

Thus we are left to give a lower bound on $\P(S_R)$ and an upper bound
on $\P(S_L \cap \{\mathfrak L > 0\})$.  Since the structure of the local times
process $\{L_t^x\}$ conditioned on $\{\mathfrak L > 0\}$ can be quite unwieldy,
we will only use first moment bounds for the latter task.  Calculating
a lower bound on $\P(S_R)$ will require a significantly more
delicate application of the second-moment method,
but here we will be able to exploit the full power of Gaussian processes
and the majorizing measures theory.

\medskip

Before defining the set $S \subseteq \mathbb R^V$, we describe it in broad terms. By
\eqref{eq:gdist} and Theorem (MM), we know that for
some $t_0 \asymp \mathfrak{S}^2$, we should have $\E \inf_{x \in V}
\eta_x = - \E \sup_{x \in V} \eta_x$ close to $-\sqrt{2t_0}$. By
Lemma \ref{lem-gaussian-concentration}, we know that the standard
deviation of $\inf_{x \in V} \eta_x$ is $O(\mathfrak{D})$. Thus we
can expect that with probability bounded away from 0, for the right
choice of $t_0 \asymp \mathfrak{S}^2$, some value on the right-hand
side of \eqref{eq:isomorph} is $O(\mathfrak{D})$ for $t=t_0$.

Now, when $\E \sup_{x \in V} \eta_x \gg \mathfrak{D}$, it is
intuitively true that for $t = \e t_0$ and $\e > 0$ small, there
should be {\em many} points $x \in V$ with $\eta_x \approx
-\sqrt{2t}$. If these points have some level of independence, then
we should expect that with probability bounded away from 0, there is
some $x \in V$ with $|\eta_x - \sqrt{2t}|$ very small (much smaller
than $O(\mathfrak{D})$). Our set $S$ will represent the
existence of such a point. On the other hand, we will argue that if
all the local times $\{L_{\tau(t)}^x\}$ are positive, then the
probability for the left-hand side to have such a low value is
small.

\subsection{A tree-like sub-process}

First, observe that by the commute time identity,
$t_{\mathrm{cov}}^{\rc}(G) \geq \mc C \max_{x,y \in V} \reff(x,y) = \mc
C \mathfrak{D}^2$. Thus in proving Theorem \ref{thm:covtime}, we may
assume that
\begin{equation}\label{eq:diameter}
\mathfrak{S} \geq K \mathfrak{D}\,,
\end{equation}
for any universal constant $K  \geq 1$. In particular, by an
application of Proposition \ref{prop-strong-tree-like}, we can
assume the existence of an $r$-separated tree $(\mcT,s)$ in $(V,d)$,
for some fixed $r \geq 128$, with root $z = v_0$, and such that for
some constant $C \geq 1$ and $\theta = \theta(C)$, properties
\eqref{eq:MM}, \eqref{eq:MMbelow}, \eqref{eq-deg}, and
\eqref{eq:treelike} below are satisfied. We will choose $C$
sufficiently large later, independent of any other parameters.

\medskip

For each $u \in \mcT$, let $h_u$ denote the height of $u$, where we
order the tree so that $h_z=0$, where $z$ is the root. Recalling
that $\mathcal L$ is the set of leaves of $\mc T$, for each $v \in
\mc L$, let $$\mc P(v) = \{f_v(0), f_v(1), \ldots, f_v(h_v)\}$$ be
the set of nodes on the path from $z = f_v(0)$ to $v = f_v(h_v)$,
where $f_v(i)$ is the parent of $f_v(i+1)$, for $0 \leq i < h$.
First, we can require that for every $v \in \mathcal L$,
\begin{equation}\label{eq:MM}
\sigma_v \geq \frac{1}{\theta} \mathfrak{S},
\end{equation}
where
\begin{eqnarray}
\chi_v(k) &\deq & r^{s(f_v(k))} \sqrt{\log \Delta(f_v(k))}\,, \\
\sigma_v &\deq& \sum_{k=0}^{h_v-1} \chi_v(k). \label{eq:sigmav}
\end{eqnarray}
Furthermore, we can require that the tree $\mathcal{T}$ satisfies,
for every $v \in V$,
\begin{equation}\label{eq:MMbelow}
\sum_{i=j+1}^{h_v-1} \chi_v(i) \geq C \cdot 2^j \cdot
r^{s(f_v(j))},
\end{equation}
as well as
\begin{equation}\label{eq-deg}
\Delta(f_v(k)) \geq \exp(C^2 r^2 4^k)\,.
\end{equation}
Finally, we require that for every $v \in \mcT$,
\begin{equation}\label{eq:treelike}
\dist_{L^2}\left(\eta_v, \CH(\{\eta_u\}_{u \notin \mcT_v})\right)
\geq \frac{1}{20} r^{s(p(v))-1}.
\end{equation}

All these requirements are justified by
Proposition~\ref{prop-strong-tree-like}.

\medskip
\noindent {\bf The distinguishing event.} For $u, v\in \mathcal{L}$,
we let $h_{uv}$ be the height of the least common ancestor of $u$
and $v$.  We will use $\ddeg(v) = |\Gamma(v)|$ to denote the number
of children of $v$. Define
\begin{equation}\label{eq-def-m-u-v}
m_u = \prod_{k=0}^{h_u-1} \ddeg(f_u(k))\,, \mbox{ and } m_{uv} =
\prod_{k=0}^{h_{uv}-1} \ddeg(f_u(k))\,.
\end{equation}
First, we fix
\begin{equation}\label{eq:epsilon}
\e = \frac{1}{2^{10}\, r \theta}\,.
\end{equation}

For every $v \in \mc L$, consider the events
\begin{equation}\label{eq:theevent}
\mc E_v(\e) = \left\{|\eta_v - \e \mathfrak{S}| \leq 50
\,r^{s(p(v))} m_v^{-3/4} \right\}.
\end{equation}

Instead of arguing directly
about the events $\mc E_v(\e)$, we will couple them to leaf events
of a ``percolation'' process on $\mcT$.  In particular, in Section
\ref{sec:coupling}, we will prove the following lemma.

\begin{lemma}\label{lem-coupling}
For all $v\in \mathcal{L}$, there exist events $\mathcal{E}_v$ such
that
the following properties hold.
\begin{enumerate}
\item $
\mathcal{E}_v  \subseteq \mc E_v(\e) = \left\{|\eta_v -  \epsilon
\mathfrak{S}| \leq 50\, r^{s(p(v))}
m_v^{-3/4}\right\}$.\label{item-coupling-1}
\item $\P(\mathcal{E}_v) \geq \frac{1}{2} m_v^{-7/8}$.
\item $\P(\mathcal{E}_u \cap \mathcal{E}_v) \leq m_{uv}^{1/8} (m_u
m_v)^{-7/8}$.
\end{enumerate}
\end{lemma}

In Section \ref{sec:percolation},
we will prove that for any events $\{\mc E_v\}_{v \in \mc L}$ satisfying
properties (2) and (3) of Lemma \ref{lem-coupling},
we have
\begin{equation}\label{eq:perc-con}
\P\left(\bigcup_{u\in \mathcal{L}} \mathcal{E}_u\right) \geq
\frac{1}{8}.
\end{equation}
Thus for $t = \frac12 \e^2 \mathfrak{S}^2$, we have
\begin{equation}\label{eq:RHS}
\P\left(\exists v \in V : \frac12 (\eta_v + \sqrt{2t})^2 \leq 50^2
r^{2 s(p(v))} m_v^{-3/2}\right) \geq \frac18\,.
\end{equation}

In light of the discussion surrounding \eqref{eq:outline},
the reader should think of
$$S = \left\{ s \in \mathbb R^V : s_v \leq 50^2 r^{2 s(p(v))} m_v^{-3/2} \textrm{ for some } v \in V\right\},$$
and then \eqref{eq:RHS} gives the desired lower bound on $\P(S_R)$.
We now turn to an upper bound on $\P(S_L \cap \{ \mathfrak L > 0 \})$.
The next lemma is proved in
Section \ref{sec:localtimes}.

\begin{lemma}\label{lem:EL0}
For $t \geq \frac12 \e^2 \mathfrak{S}^2$,
\begin{equation}\label{eq:LHS}
\P\left(\bigcup_{v \in \mc L} \left\{0 < L_{\tau(t)}^v \leq  50^2
\cdot r^{2 s(p(v))} m_v^{-3/2}\right\} \right) \leq \frac{1}{16}.
\end{equation}
\end{lemma}

From \eqref{eq:LHS} and \eqref{eq:RHS}, we conclude that
with probability at least $1/16$, we must have
$L_{\tau(t)}^v=0$ for some $v \in V$ and $t= \frac12 \e^2 \mathfrak{S}^2$,
else \eqref{eq:isomorph} is violated.  This implies that
\begin{equation}\label{eq:penult}
\P_{v_0}\left(\tau_{\mathrm{cov}}^{\rc} \geq \tau(\tfrac12 \e^2
\mathfrak{S}^2)\right) \geq \frac{1}{16}\,.
\end{equation}

To finish our proof of \eqref{eq:covLB} and complete the proof
of Theorem \ref{thm:covtime}, we will apply
Lemma~\ref{lem-tau-t} with $\beta = \frac{1}{96}$. In particular, we
may choose $K=96/\e$ in \eqref{eq:diameter}, and then applying
Lemma~\ref{lem-tau-t} yields
$$
\P\left(\tau(\tfrac12 \e^2 \mathfrak{S}^2) \leq \mathcal C \frac{\e^2 \mathfrak{S}^2}{192}\right) \leq \frac{1}{32}\,.
$$
Combining this with \eqref{eq:penult} yields
\[
\P_{v_0} \left(\tau_{\mathrm{cov}}^{\rc} \geq \mathcal C \frac{\e^2 \mathfrak{S}^2}{192}\right) \geq \frac{1}{16}.
\]
In particular, $\tau_{\mathrm{cov}}^{\rc} \gtrsim \mathcal C \e^2 \mathfrak{S^2}$.
This completes the proof of \eqref{eq:covLB}, and hence of Theorem \ref{thm:covtime}.

\subsection{The coupling}
\label{sec:coupling}

\remove{
In light of the previous subsection, we study some separated tree
$\mathcal{T} \subset G$ with the set of leaves $\mathcal{L}$. We may
assume that $z$ is the root of $\mathcal{L}$. For $v\in
\mathcal{L}$, denote by $\mathcal{P}(v) = \{f_v(0), f_v(1), \ldots,
f_{v}(h_v)\}$ the set of nodes on the path from the root to $v$,
where $z = f_v(0)$, $f_v(h_v) = v$ and $f_v(i)$ is the parent of
$f_v(i+1)$ for $0\leq i < h_v$.

Write $\mathfrak{S} = \gamma_2(G)$. By Theorem [Preprocessing], 
we can construct a separated tree $\mathcal{T}$ with $r\geq 10^4$
such that for some constant $\theta \geq 1$ and for all  $v \in
\mathcal L$:
\begin{equation}\label{eq:MM}
\frac{1}{\theta} \mathfrak{S} \leq \sigma_v \leq \theta
\mathfrak{S},
\end{equation}
where
\begin{eqnarray*}
\chi_v(k) &\deq & r^{s(f_v(k))} \sqrt{\log \ddeg(f_v(k))}\,, \\
\sigma_v &\deq& \sum_{k=0}^{h_v-1} \chi_v(k).
\end{eqnarray*}
Furthermore, we can require that the tree $\mathcal{T}$ satisfies
that
\begin{equation}\label{eq:MMbelow}
\sum_{i=j+1}^{h_v-1} \chi_v(i) \geq 10^5 \cdot 2^{8} \cdot 2^j \cdot
r^{s(f_v(j))},
\end{equation}
as well as
\begin{equation}\label{eq-deg}
\ddeg(f_v(k)) \geq \exp(2^{14} \cdot 50^2 rk)\,.
\end{equation}
}

The  present section is devoted to the proof of
Lemma~\ref{lem-coupling}.
Toward this end, we will try to find a leaf $v \in \mathcal L$
for which $\eta_v \approx \varepsilon \mathfrak S$.
As in Lemma \ref{lem-coupling}(1), the level of closeness we desire is gauged according
to a proper scale, $r^{s(p(v))}$, as well as to the number of
other leaves we expect to see at this scale, which is represented roughly by $m_v^{-3/4}$
(the value $3/4$ is not essential here, and any other value in $(1/2, 1)$ would suffice).

Our goal is to find a such a leaf by starting at the root of the tree,
and arguing that some of its children should be somewhat close to
the target $\varepsilon \mathfrak{S}$.
This closeness is achieved using the fact that, by definition of an $r$-separated tree,
the children are separated in the Gaussian distance, and thus
exhibit some level of independence.  We will continue in this
manner inductively, arguing that the children which are somewhat close
to the target have their own children which we could expect to
be even closer, and so on.
We aim to shrink these windows around the target more and more so they are
small enough once we reach the leaves.
There are a number of difficulties involved in executing this scheme.
In particular, conditioning on the exact values of the children
of the root could determine the entire process, making future
levels moot.  Thus we must first select a careful filtering
which allows us to reserve some randomness for later levels.
This is done in Section \ref{sec:restructure}.

Furthermore, the intermediate
targets have to be arranged according to the
variances along the root-leaf paths in our tree.
This corresponds to the fact that,
although we have a uniform lower bound on each $\sigma_v$ (from \eqref{eq:MM}),
the summation defining the $\sigma_v$'s could put different
weights on the various levels (recall \eqref{eq:sigmav}).
The targets also have to take into account random ``noise''
from the filter described above, and thus
the targets themselves must be random.
This ``window analysis'' is performed in Section \ref{sec:window}.

\subsubsection{Restructuring the randomness}
\label{sec:restructure}

We know that $\eta_z =
0$, since $z=v_0$ is the root of $\mathcal{T}$ (and the starting
point of the associated random walk). Fix a depth-first ordering of
$\mcT$ (one starts at the root and explores as far as possible
along each branch before backtracking). Write $u \prec v$ if $u$ is
explored before $v$, and $u \preceq v$ if $u \prec v$ or $u=v$. For
$u \neq z$, we write $u^-$ for the vertex preceding $u$ in the DFS
order. Let $\mc F = \spn\left(\{\eta_x: x\in \mathcal{T}\}\right)$.
For a node $v \in \mcT$, let $\mathcal F_v = \spn(\{\eta_u\}_{u
\preceq v})$ and $\mathcal F^-_v = \spn(\{\eta_u\}_{u \prec v})$. We
next associate a centered Gaussian process $\{\xi_x: x\in
\mathcal{T}\}$ to $ \{\eta_x : x\in \mathcal{T}\}$ in the following
inductive way. Define $\xi_{z} = 0$. Now, assuming we have defined
$\xi_u$ for $u \prec v$, we define $\xi_{v}$ by writing
$$
\eta_{v} = \zeta_{v} + \xi_{v},
$$
where $\zeta_{v} \in \mathcal F_{v^-}$ and $\xi_{v} \perp \mathcal
F_{v^-}$. Observe that, by construction, $\{\xi_u\}_{u \preceq v}$
forms an orthogonal basis in $L^2$ for $\mathcal F_v$.

\medskip

Applying \eqref{eq:treelike}, we have for all $u\in \mathcal{T}$,
\begin{equation}\label{eq-var-xi}
\|\xi_u\|_{2} = \dist_{L^2}\left(\eta_u, \spn\left(\{\eta_w\}_{w
\prec u}\right)\right) \geq \dist_{L^2}\left(\eta_u,
\spn\left(\{\eta_w\}_{w \notin \mcT_u}\right)\right) \geq
\frac{1}{20} r^{s(p(u))-1}\,,\end{equation} where we used the fact
that the span and the affine hull are the same since $\xi_z = 0$.
For $v\in \mathcal{L}$, define the subspaces
\begin{eqnarray*}
\mathcal{F}_{v , k}
&=& \spn\left(\left\{\xi_u:  f_{v}(k) \prec u \preceq f_{v}(k+1)\right\}\right)\,, \\
\mathcal F_{v,k}^{-} &=& \spn\left(\left\{\xi_u:  f_{v}(k) \prec u \prec f_{v}(k+1)\right\}\right).
\end{eqnarray*}

For $0\leq k \leq h_v - 1$, define inductively $\tilde \eta_{v,
0}=0$, and
\begin{equation}\label{eq-def-eta-prime}
\tilde\eta_{v,k+1}  = \tilde\eta_{v,k} + \proj_{\mathcal
F_{v,k}}(\eta_v).\end{equation} Note that the subspaces $\{\mc
F_{v,k}\}_{k=0}^{h_v}$ are mutually orthogonal, and together they
span $\mc F_v$. Thus, \begin{equation}\label{eq-eta-tilde-eta}
\tilde \eta_{v, h_v} = \eta_v\,.\end{equation}

Furthermore, by the definition of the subspace $\mathcal{F}_{v,
k}$, we can decompose
\begin{equation}\label{eq-eta-prime}
\tilde\eta_{v, k+1} - \tilde\eta_{v, k} = \tilde\zeta_{v , k} +
\tilde\xi_{v, k}\,,
\end{equation}
where $\tilde\zeta_{v, k} \in \mathcal F_{v,k}^{-}$, and
$\tilde\xi_{v, k} \perp \mathcal F_{v,k}^{-}$. The next lemma states
that $\tilde{\xi}_{v,k}$ has at least comparable variance to
$\tilde{\zeta}_{v,k}$.

\begin{lemma}\label{lem:variance}
For every $v \in \mathcal L$ and $k=0,1,\ldots,h_v-1$, we have the estimates
\begin{equation}\label{eq:varzeta}
\left\|\tilde\zeta_{v,k}\right\|_2 \leq 8\,r^{s(f_v(k))},
\end{equation}
and,
\begin{equation}
 \tfrac{1}{64} r^{s(f_v(k))-1} \leq \left\|\tilde\xi_{v,k}\right\|_2 \leq 8\,r^{s(f_v(k))}.
\end{equation}
\end{lemma}

\begin{proof}
Writing the telescoping sum,
$$
\eta_v = \sum_{j=0}^{h_v-1} \eta_{f_v(j+1)} - \eta_{f_v(j)},
$$
we see that
\begin{equation}\label{eq:projF}
\left\|\proj_{\mc F_{v,k}}(\eta_v)\right\|_2 \leq \sum_{j=k}^{h_v-1}
\|\eta_{f_v(j+1)}-\eta_{f_v(j)}\|_2 \leq \sum_{j=k}^{h_v-1} 4
r^{s(f_v(j))} \leq 8 \, r^{s(f_v(k))},
\end{equation}
where we used properties (1) and (3) of the separated tree, and have
assumed $r \geq 2$.

Thus by orthogonality and \eqref{eq-eta-prime},
we have
\begin{equation*}
\left\|\tilde \zeta_{v,k}\right\|_2 \leq \left\|\tilde \eta_{v, k+1}
- \tilde \eta_{v, k}\right\|_2 = \left\|\proj_{\mc
F_{v,k}}(\eta_v)\right\|_2 \leq 8\, r^{s(f_v(k))},
\end{equation*}
and precisely the same conclusion holds for $\tilde\xi_{v,k}$.

Next, we establish a lower bound on $\|\tilde\xi_{v,k}\|_2$. From
\eqref{eq-def-eta-prime} and \eqref{eq-eta-prime},
\begin{eqnarray}
\tilde\xi_{v,k} &=& \proj_{\mc F_{v,k}} (\eta_v) - \proj_{\mc F_{v,k}^-} (\eta_v) \label{eq:auxxi}\\
&=& \sum_{j=k}^{h_v-1}\left( \proj_{\mc F_{v,k}}\nonumber
(\eta_{f_v(j+1)}-\eta_{f_v(j)}) - \proj_{\mc F_{v,k}^-}
(\eta_{f_v(j+1)}-\eta_{f_v(j)})\right)
\\
&= &\nonumber
\left[\proj_{\mc F_{v,k}} (\eta_{f_v(k+1)}-\eta_{f_v(k)}) - \proj_{\mc F_{v,k}^-} (\eta_{f_v(k+1)}-\eta_{f_v(k)})\right] \\
&& + \sum_{j=k+1}^{h_v-1} \left(\proj_{\mc F_{v,k}}\nonumber
(\eta_{f_v(j+1)}-\eta_{f_v(j)}) - \proj_{\mc F_{v,k}^-}
(\eta_{f_v(j+1)}-\eta_{f_v(j)})\right).
\end{eqnarray}
Observe that the term in brackets is precisely
$$
\proj_{\mc F_{v,k}} (\eta_{f_v(k+1)}) - \proj_{\mc F_{v,k}^-}
(\eta_{f_v(k+1)}) = \xi_{f_v(k+1)},
$$
since $\eta_{f_v(k)} \perp \mc F_{v,k}$.  In particular, we arrive at
\begin{eqnarray*}
\left\|\tilde\xi_{v,k}\right\|_2 &\geq & \left\|\xi_{f_v(k+1)}\right\|_2 - \sum_{j=k+1}^{h_v-1} \left\|\eta_{f_v(j+1)}-\eta_{f_v(j)}\right\|_2 \\
&\geq& \tfrac{1}{32} r^{s(f_v(k))-1} - 2\, r^{s(f_v(k+1))} \\
&\geq& \tfrac{1}{32} r^{s(f_v(k))-1} - 2\, r^{s(f_v(k))-2} \\
&\geq& \tfrac{1}{64} r^{s(f_v(k))-1},
\end{eqnarray*}
where in the second line we have used \eqref{eq-var-xi} and
properties (1) and (2) of the separated tree,
and in the final line we have used $r \geq 128$.
\end{proof}

\subsubsection{Defining the events $\mathcal{E}_v$}
\label{sec:window}

Recall that our
goal now is to find many leaves $v \in \mathcal L$ with $\eta_v
\approx \epsilon \mathfrak{S}$.
Now, writing
$$
\eta_v = \sum_{k=0}^{h_v-1} \proj_{\mc F_{v,k}}(\eta_v) =
\sum_{k=0}^{h_v-1}( \tilde\zeta_{v,k} + \tilde\xi_{v,k}),
$$
our ``ideal'' goal would be to hit a window around the target by getting the $k$th term of this sum
close to $$a_v(k) \deq \e \mathfrak{S} \frac{\chi_v(k)}{\sigma_v},$$ for $k=0,1,\ldots,h_v-1$.
We will use the variance of the $\tilde\xi_{v,k}$ variables (recall Lemma \ref{lem:variance}) to lower bound
the probability that
some points get closer to the desired target.  On the other hand,
we will treat the $\tilde\zeta_{v,k}$ variables as noise which has to be bounded
in absolute value.

This noise cannot always be countered in a single level, but it
can be countered on average along the path to the leaf; this is the
content of \eqref{eq:MMbelow}.
We will amortize this cost over future targets as follows.
Let $b_v(0)=0$ and for $k=0,1,\ldots,h_v-2$, define
\begin{eqnarray*}
\rho_v(k) &=& \tilde\zeta_{v,k} + \tilde\xi_{v,k} - a_v(k) + b_v(k)\,, \\
b_v(k+1) &=& \sum_{i=0}^k \frac{\chi_v(k+1)}{\sum_{\ell=i+1}^{h_v-1} \chi_v(\ell)} \rho_v(i).
\end{eqnarray*}

Clearly $\rho_v(0) = \tilde\zeta_{v,0} + \tilde\xi_{v,0} - a_v(0)$
represents how much we miss our first target.  A similar fact holds
for the final target, as the next lemma argues; in between, the
errors are spread out proportional to the contribution to
$\val_r(\mcT, s)$ for each of the the remaining levels (represented
by the $\chi_v(k)$ values). Here $b_v(k)$ represents the error that
is meant to be absorbed in the $k$-th level.

\begin{lemma}\label{lem:rho}
For every $v \in \mathcal L$,
$$
\rho_v(h_v-1) = \eta_v - \e \mathfrak{S}.
$$
\end{lemma}

\begin{proof}
\remove{
Note that $\rho_v(k)$ is the error which is created $k$-th round. So
at the beginning of $(k+1)$-round, the process has been accumulated
errors from the first $k$-rounds and we would like to correct them
in the later steps of the process. In particular, $b_v(k+1)$ counts
the amount of error that is supposed to correct in the $(k+1)$-th
round. In this case, at the end of the process ($(h_v-1)$-round),
all the error created in the first $h_v-2$ steps has been corrected.}
We have,
\begin{align}\label{eq-eta-b}
\sum_{k=0}^{h_v - 2} b_v(k+1) = \sum_{k=0}^{h_v - 2} \sum_{i=0}^k
\frac{\chi_v(k+1)}{\sum_{\ell = i+1}^{h_v-1} \chi_v (\ell)}
\rho_v(i)   = \sum_{i=0}^{h_v -2} \rho_v(i)
\sum_{k=i}^{h_v-2}\frac{\chi_v(k+1)}{\sum_{\ell = i+1}^{h_v-1}
\chi_v (\ell)} =  \sum_{i=0}^{h_v - 2} \rho_v(i)\,.
\end{align}
Also note that
\[\sum_{k=0}^{h_v-1}\rho_v(k) = \sum_{k=0}^{h_v-1}(\tilde\zeta_{v,k} + \tilde\xi_{v,k} - a_v(k) + b_v(k)) = \eta_v - \e\mathfrak{S} + \sum_{k=0}^{h_v-1}b_v(k)\,.\]
Combined with $b_v(0)=0$ and \eqref{eq-eta-b}, it follows that
$\rho_v(h_v-1) = \eta_v - \e \mathfrak{S}$, completing the proof.
\end{proof}

We now define the events
\begin{align*}
\mc A_v(k) &= \{|\tilde\zeta_{v, k}| \leq \epsilon \theta \chi_v(k)\}\,, \\
\mc B_v(k)&= \{|\rho_v(k)| \leq w_v(k)\}\,,
\end{align*}
where, for $0 \leq k \leq h_v-2$, $w_v(k)$ is selected so that
\begin{equation}\label{eq:select1}
\P\left(\mc B_v(k) \mid \tilde{\zeta}_{v, k} + b_v(k)\right) =
\ddeg(f_v(k))^{-1/8}.
\end{equation}
We emphasize that the windown $w_v(k)$ is not deterministic.
And, for $k = h_v-1$, we
select $w_v(k)$ so that
\begin{equation}\label{eq:select2}
\P(\mc B_{v}(k) \mid  \tilde{\zeta}_{v, k} + b_v(k)) =
\ddeg(f_v(k))^{-1/8} m_v^{-3/4},
\end{equation}

\begin{remark}
Here, $w_v(k)$ can be thought to represent the window size around
the random target.  The value of $w_v(k)$ is chosen to make the probabilities
in \eqref{eq:select1} and \eqref{eq:select2} exact, allowing us to
couple seamlessly to the percolation process in Section \ref{sec:percolation}.
The key fact, proved in Lemma \ref{lem:window}, is that
the window sizes actually satisfy a deterministic upper bound, assuming
that all the ``good'' events on the path from the root to $f_v(k)$
occurred.   Thus one should think of the true window size
as the bounds specified in \eqref{eq:window} and \eqref{eq:window2}, while
the random value is for the purpose of the coupling.
\end{remark}

For $0\leq k\leq \ell \leq
h_v-1$, define
\begin{align}
\mathcal{A}_v(k, \ell) \deq \bigcap_{i=k}^\ell \mc A_v(i)\ \mbox{ and }\
\mathcal{B}_v(k, \ell) \deq \bigcap_{i=k}^\ell \mc B_v(i).
\end{align}
 Since $\tilde{\xi}_{v, k} \in \sigma(\mathcal F_{v,k} \setminus \mathcal F_{v,k}^{-})$ (see, e.g. \eqref{eq:auxxi}),
 we see that the event $\mc B_v(k)$ is
conditionally independent of $\sigma(\mathcal{F}^-_{f_v(k+1)})$ given the
value of $\tilde{\zeta}_{v, k} + b_v(k)$. This implies that for all
events $\mathcal{E}_0 \in \sigma(\mathcal{F}^-_{f_v(k+1)})$ such that
$\mathcal{E}_0 \cap \mathcal{A}_v(0, k) \cap \mathcal{B}_v(0, k-1)
\neq \emptyset$,
\begin{equation}\label{eq-B-k}\P\left(\mc B_v(k) \mid \mathcal{A}_v (0, k), \mathcal{B}_v (0, k-1), \mathcal{E}_0\right) =
\begin{cases} \ddeg(f_v(k))^{-1/8}, & \mbox{ if } 0 \leq k < h_v-1,\\
\ddeg(f_v(k))^{-1/8} m_v^{-3/4}, & \mbox{ if } k = h_v -1.
\end{cases}\end{equation}

\medskip

Finally, for $v \in \mathcal L$, we define the event
\begin{equation}\label{eq-def-E}
\mc E_v = \mathcal{A}_v(0, h_v -1) \cap \mathcal{B}_v(0,
h_v-1)\,.\end{equation}

\medskip
\noindent
{\bf Window analysis.}
We will now show that our final window $w_v(h_v-1)$ is small enough.
Observe that our choice of $w_v(k)$ is not deterministic.
Nevertheless, we will give an absolute upper bound.
The bound is essentially the natural one:  For any node $u$
in the tree, and any child $v$ of $u$,
the standard deviation of $\eta_{u}-\eta_{v}$ is $O(r^{s(u)})$.
This follows from property (3) of the $r$-separated tree (recall
Definition \ref{def:septree}).

\begin{lemma}
\label{lem:window} For every $v \in \mathcal L$ and
$k=0,1,\ldots,h_v-2$, if $\mathcal A_v(0,k)$ and $\mathcal B_v(0,k-1)$
hold then,
\begin{equation}\label{eq:window}
w_v(k) \leq 50\, r^{s(f_v(k))}.
\end{equation}
Furthermore, if $\mathcal A_v(0,h_v-1)$ and $\mathcal B_v(0,h_v-2)$ hold, then
\begin{equation}\label{eq:window2}
w_v(h_v-1) \leq 50\, r^{s(f_v(h_v-1))}\, m_v^{-3/4}.
\end{equation}
\end{lemma}

\begin{proof}
For $k=0$, we have
$\rho_v(0)=\tilde\zeta_{v,0} + \tilde\xi_{v,0} - a_v(0)$.
By \eqref{eq:MM}, we have
\begin{equation}\label{eq:SS}
a_v(0) = \e \mathfrak{S} \chi_v(0)/\sigma_v \leq \theta \e \chi_v(0)
= \theta \e r^{s(f_v(0))} \sqrt{\log \Delta(f_v(0))}.
\end{equation}
Furthermore, from Lemma \ref{lem:variance}, we know that
for all $k \geq 0$,
\begin{equation}\label{eq:rememberxi}
\tfrac{1}{64} r^{s(f_v(k))-1} \leq \left\|\tilde\xi_{v,k}\right\|_2
\leq 8\,r^{s(f_v(k))}.
\end{equation}

Now, consider a value $w > 0$ such that
\begin{equation}\label{eq:wcon}
w\leq a_v(0) + \e \theta \chi_v(0) \le 2 \theta \e r^{s(f_v(0))}
\sqrt{\log \Delta(f_v(0))}\,.
\end{equation}
Using \eqref{eq:rememberxi} and recalling the Gaussian density,
we have
\begin{align}
\P\left(|\rho_v(0)| \leq w \mid \mc A_v(0)\vphantom{\bigoplus}\right)&
\geq \P\left(|\rho_v(0)| \leq w \mid
\tilde{\zeta}_{v,0} = -\e \theta \chi_v(0)\vphantom{\bigoplus}\right)\nonumber\\
&=
\P\left(|\tilde\xi_{v,0} - a_v(0) - \e \theta \chi_v(0)| \leq w\right)\nonumber \\
& \geq \frac12 \frac{w}{\sqrt{2\pi} \,8 r^{s(f_v(0))}}
\exp\left(-\tfrac12 (128 \e r \theta)^2 \log \Delta(f_v(0))\right)\nonumber\\
& = \frac{w}{16\sqrt{2\pi} r^{s(f_v(0))}} \Delta(f_v(0))^{-\tfrac12
(128 \e r \theta)^2}\,. \label{eq:tailb}
\end{align}

Recalling the assumption \eqref{eq-deg}, we have $\sqrt{\log
\Delta(f_v(0))} \geq C r \geq 16 \sqrt{2\pi} 2^{10} r$, by choosing
$C$ large enough. In particular, $$\e \theta \chi_v(0) \geq (16
\sqrt{2\pi} 2^{10} \e \theta r) r^{s(f_v(0))} = 16 \sqrt{2\pi}
r^{s(f_v(0))},$$ recalling \eqref{eq:epsilon}. Thus setting $w = 16
\sqrt{2\pi} r^{s(f_v(0))}$ satisfies \eqref{eq:wcon}, and applying
\eqref{eq:tailb} we have
\[\P\left(\vphantom{\bigoplus}|\rho_v(0)| \leq 16 \sqrt{2\pi} r^{s(f_v(0))} \mid \mc A_v(0)\right) \geq \Delta(f_v(0))^{-\tfrac12 (128 \e r \theta)^2}
\geq \ddeg(f_v(0))^{-1/8}\,,\] where we have used $\tfrac12 (128 \e
r\theta)^2 = \frac{1}{128}$, and $\Delta(f_v(0)) \geq 16$ from
\eqref{eq-deg}. Therefore $$w_v(0) \leq 16 \sqrt{2\pi} r^{s(f_v(0))}
\leq 50\, r^{s(f_v(0))},$$ recalling the definition of $w_v(0)$ from
\eqref{eq:select1}.

\medskip

Now suppose that \eqref{eq:window} holds for all $k \leq \ell <
h_v-2$, and consider the case $k=\ell+1$. If the events $\{\mc B_v(j) :
0 \leq j \leq \ell\}$ hold, then
\[
|\rho_v(j)| \leq w_v(j) \leq 50\, r^{s(f_v(j))},
\]
where the first inequality is from the definition of $\mc B_v(j)$, and
the second is from the induction hypothesis.
Using \eqref{eq:MMbelow},
it follows that
\begin{equation}\label{eq:bvk}
|b_v(k)| \leq \sum_{i=0}^{k-1}
\frac{\chi_v(k)}{\sum_{\ell=i+1}^{h_v-1} \chi_v(\ell)} |\rho_v(i)|
\leq \frac{2}{C} \chi_v(k).
\end{equation}

Recall that $\rho_v(k) = \tilde\zeta_{v,k} + \tilde\xi_{v,k} -
a_v(k) + b_v(k)$. Similar to the $k=0$ case, we obtain that for $$0
<w \leq  2 \theta \e r^{s(f_v(k))} \sqrt{\log \Delta(f_v(k))},$$ we
have,
\begin{eqnarray*}
\!\!\!\!\!\!\!\!\!\!\!\!\P && \!\!\!\!\!\!\!\!\!\!\!\!\!\!\!\left(\vphantom{\bigoplus}|\rho_v(k)|  \leq w \mid \mc A_v(i), \mc B_v(i) \mbox{ for all } 0\leq i<
k, \mc A_v(k)\right)
\\
&\geq& \P\left(\vphantom{\bigoplus} \left|\tilde \xi_{v,k} - a_v(k)
- \e \theta \chi_v(k) - \frac{2}{C} \chi_v(k)\right| \leq w \right)
\\
&\geq& \frac12 \frac{w}{\sqrt{2\pi}8\, r^{s(f_v(k))}}
\Delta(f_v(k))^{-\frac12 (128 r)^2 (\e \theta + C^{-1})^2}.
\end{eqnarray*}
Now, by choosing $C \geq 1024 r$, and recalling \eqref{eq:epsilon},
we see that $$\frac12 (128 r)^2 (\e \theta + C^{-1})^2 \leq
\frac{1}{32}\,.$$ Since $\Delta(f_v(k)) \geq 16$ (again, by
\eqref{eq-deg}), we conclude that
\[
\P\left(\vphantom{\bigoplus}|\rho_v(k)|  \leq 16 \sqrt{2\pi}
r^{s(f_v(k))} \mid \mc A_v(i), \mc B_v(i) \mbox{ for all } 0\leq i<
k, \mc A_v(k)\right) \geq \ddeg(f_v(k))^{-1/8}.\] This implies
$w_v(k) \leq 16 \sqrt{2\pi} r^{s(f_v(k))} \leq 50\, r^{s(f_v(k))}$,
where we recall once again the definition of $w_v(k)$ from
\eqref{eq:select1}.

An almost identical argument yields that $w_v(h_v-1) \leq 50\,
r^{s(f_v(h_v-1))} m_v^{-3/4}$.
\end{proof}

The next lemma states that the events
$\mathcal{E}_v$ as defined in \eqref{eq-def-E} satisfy requirement
\eqref{item-coupling-1} of Lemma~\ref{lem-coupling}.

\begin{lemma}\label{lem-coupling-1}
If $\mc E_v$ occurs, then
$$
\left|\eta_v - \e \mathfrak{S}\right| \leq w_v(h_v-1) \leq 50\,
r^{s(f_v(h_v-1))} m_v^{-3/4}.
$$
\end{lemma}

\begin{proof}
This follows directly from Lemma \ref{lem:rho}, the identity
\eqref{eq-eta-tilde-eta} and the definition of $\mc B_v(k)$.
\end{proof}

\medskip
\noindent {\bf The first moment.} We now give
lower bounds
on the probability of the event $\mathcal{E}_v$.

 \begin{lemma}\label{lem-coupling-2}
For every $v \in \mathcal L$, $$\P(\mathcal E_v) \geq \frac{1}{2} \,
m_v^{-7/8}.$$
\end{lemma}

\begin{proof}
We
have,
\begin{eqnarray}\label{eq-prob-E-v}
\P(\mathcal E_v) &=& \prod_{k=0}^{h_v-1} \P\left(\mc A_v(k) \mid
\mathcal{A}_v(0, k-1),  \mathcal{B}_v(0, k-1)\right)
\P\left(\mc B_v(k) \mid \mathcal{A}_v(0, k),  \mathcal{B}_v(0, k-1)\right) \nonumber \\
&=& m_v^{-3/4} \prod_{k=0}^{h_v-1} \ddeg(f_v(k))^{-1/8}
\prod_{k=0}^{h_v-1} \P\left(\mc A_v(k) \mid \mathcal{A}_v(0, k-1), \mathcal{B}_v(0, k-1)\right)\nonumber\\
 &=& m_v^{-7/8} \prod_{k=0}^{h_v-1} \P\left(\mc A_v(k)\right)
\label{eq:Ev-two},
\end{eqnarray}
where the second line follows from \eqref{eq-B-k}, and the third
line from the fact that $\mc A_v(k)$ is independent of $\{ \mc A_v(i),
\mc B_v(i) : 0 \leq i < k \}$.

\medskip

Using \eqref{eq:varzeta}, we have
\begin{align*}
\P(\mc A_v(k)) \geq 1 - \frac{2}{\sqrt{2\pi}} \int_{\e \theta
\chi_v(k)}^{\infty} \exp\left(-\frac{x^2}{128 r^{2
s(f_v(k))}}\right)dx &\geq
1 - 2 \Delta(f_v(k))^{-\frac{1}{128} \e^2 \theta^2} \\
& \geq 1 - 2 \exp\left(-\frac{1}{128} 2^{-20} C^2 4^k\right).
\end{align*}
where we have used \eqref{eq-deg}, the definition of $\e$ \eqref{eq:epsilon}, and $\chi_v(k) = r^{s(f_v(k))} \sqrt{\log \Delta(f_v(k))}$.

Clearly by choosing $C$ a large enough constant, we have
$$
\prod_{k=0}^{h_v-1} \P\left(\mc A_v(k)\right) \geq \frac{1}{2},
$$
completing the proof.
\end{proof}
\medskip

\noindent {\bf The second moment.} Finally, we bound the
probability of $\mc E_u \cap \mc E_v$ for $u \neq v$.

\begin{lemma}\label{lem-coupling-3}
For every $u,v \in \mathcal L$,
$$
\P(\mc E_u \cap \mc E_v) \leq m_{uv}^{1/8} (m_u m_v)^{-7/8}.
$$
\end{lemma}

\begin{proof}
Assume, without loss of generality, that $u \prec v\in \mathcal{L}$.
It is clear from \eqref{eq-prob-E-v} that $\P(\mathcal{E}_u) \leq
m_u^{-7/8}$. Also, we have
\begin{align*}
\P(\mathcal{E}_v \mid \mathcal{E}_u) \leq \P(\mathcal{A}_v(0,
h_v-1), \mathcal{B}_v(0, h_u-1)\mid \mathcal{E}_u) \leq \prod_{k =
h_{uv}}^{h_v - 1} \P (\mathcal B_v(k) \mid \mathcal{E}_u, \mathcal{A}_v(0, k),
\mathcal{B}_v(0, k-1))\,.
\end{align*}
Now recall that $\mathcal{E}_u \in \sigma(\mathcal F^{-}_{f_v(h_{uv}+1)}) \subset
\sigma(\mathcal F^{-}_{f_v(k+1)})$ for all $k\geq h_{uv}$. By \eqref{eq-B-k}, we
obtain,
\[\prod_{k =
h_{uv}}^{h_v - 1} \P (\mathcal B_v(k) \mid \mathcal{E}_u, \mathcal{A}_v(0, k),
\mathcal{B}_v(0, k-1)) = \ddeg (f_v(h_v-1))^{-3/4}\prod_{k=
h_{uv}}^{h_v - 1} \ddeg(f_v(k))^{-1/8} = m_{uv}^{1/8} m_v^{-7/8}\,.\]
Altogether, we conclude that
\[\P(\mathcal{E}_u \cap \mathcal{E}_v)  = \P(\mathcal{E}_u) \P(\mathcal{E}_v \mid \mathcal{E}_u) \leq m_{uv}^{1/8} (m_u m_v)^{-7/8}\,,\]
as required.
\end{proof}

The main coupling lemma, Lemma~\ref{lem-coupling}, is an immediately
corollary of Lemmas~\ref{lem-coupling-1}, \ref{lem-coupling-2} and
\ref{lem-coupling-3}.

\subsection{Tree-like percolation}
\label{sec:percolation}

Lemma \ref{lem-percolation} below yields \eqref{eq:perc-con}. Its
proof is a variant on the well-known second moment method for
percolation in trees (see \cite{Lyons92}). First, we define a
measure $\nu$ on $\mc L$ via $\nu(u) = m_u^{-1}$. Observe that $\nu$
is a probability measure on $\mc L$, i.e.
\begin{equation}\label{eq:flow}
\sum_{u \in \mc L} \nu(u)=1.
\end{equation}
To see this,
construct a unit flow from the root to the leaves, where
each non-leaf node splits its incoming flow equally
among its children.  Clearly the amount that
reaches a leaf $u$ is precisely $\nu(u)$.

\begin{lemma} \label{lem-percolation}
Suppose that to
 each $v\in \mathcal{L}$, we associate an event $\mathcal{E}_v$
 such that the following bounds old.
\begin{enumerate}
\item $\P(\mathcal{E}_v) \geq \frac{1}{2} m_v^{-7/8}$ for all $v\in
\mathcal{L}$.\label{item-percolation-1}
\item $\P(\mathcal{E}_u \cap \mathcal{E}_v) \leq  m_{uv}^{1/8}
(m_u m_v)^{-7/8}$ for all $u, v\in \mathcal{L}$.
\label{item-percolation-2}
\end{enumerate}
Define $Z = \sum_{u\in \mathcal{L}}  m_u^{-1/8}
\one_{\mathcal{E}_u}$. Then,
\[\P\left(Z > 0 \right) \geq \frac{1}{8}\,.\]
\end{lemma}

\begin{proof}
By assumption \eqref{item-percolation-1},
\[\E Z \geq \sum_{u\in \mathcal{L}} \frac{1}{2} m_u^{-1/8} m_u^{-7/8} = \frac{1}{2}\sum_{u\in \mathcal{L}} m_u^{-1} = \frac{1}{2}\,.\]
where the last equality follows from \eqref{eq:flow}.

By assumption \eqref{item-percolation-2}, we have
\begin{align*}
\E Z^2 = \sum_{u, v\in \mathcal{L}} (m_u m_v)^{-1/8}
\P(\mathcal{E}_u \cap \mathcal{E}_v) \leq  \sum_{u, v \in
\mathcal{L}}  m_{u v}^{1/8} (m_u m_v)^{-1}\,.
\end{align*}
In order to estimate the second moment, we first fix $u$ and sum
over $v$. To be more precise, let $$\mathcal{L}_{h}(u) = \{v \in
\mathcal{L}: h_{uv} = h\},$$
where we recall that $h_u$ is the height of a node $u$,
and $h_{uv}$ is the height of the least-common ancestor
of $u$ and $v$.

 We can then partition $\mathcal{L} = \bigcup_{h \geq 0} \mc L_h(u)$
and obtain for every $u \in \mc L$,
\begin{align*}
\sum_{v\in \mathcal{L}} m_{uv}^{1/8} m_v^{-1} =
\sum_{h=0}^{h_u} \sum_{v \in \mathcal{L}_h(u)} m_{uv}^{1/8}
m_v^{-1} &=\sum_{h=0}^{h_u}
 \prod_{i=0}^{h -
1}\ddeg(f_u(i))^{1/8} \sum_{v \in \mathcal{L}_h(u)}m_v^{-1}  \\
& =
\sum_{h=0}^{h_u}
 \prod_{i=0}^{h -
1}\ddeg(f_u(i))^{1/8} \nu(\mathcal{L}_h(u))\,.
\end{align*}
Recalling the flow representation of the measure $\nu$, we see that
\[\nu(\mathcal{L}_h(u)) = \frac{\ddeg(f_u(h)) -
1}{\ddeg(f_u(h))}\prod_{i=0}^{h - 1}\ddeg(f_u(i))\,.\]
Therefore,
\begin{align*} \sum_{v\in \mathcal{L}} m_{uv}^{1/8} m_v^{-1}= \sum_{\ell=0}^{h_u} \frac{\ddeg(f_u(h)) -
1}{\ddeg(f_u(h))}\prod_{i=0}^{h - 1}\ddeg(f_u(i))^{-7/8}
\leq\sum_{\ell=0}^{h_u}\prod_{i=0}^{h - 1}\ddeg(f_u(i))^{-7/8}
\leq 2 \,,
\end{align*}
where the last transition follows from \eqref{eq-deg},
for $C$ chosen sufficiently large.  Applying the
second moment method, we deduce that
\begin{equation*}\P\left(Z >0 \right) \geq \frac{(\E Z)^2}{\E Z^2} \geq \frac{1}{8}\,,
\end{equation*}
completing the proof.
\end{proof}

\subsection{The local times}
\label{sec:localtimes}

We now prove Lemma \ref{lem:EL0}, in order to
the complete the analysis of the left-hand side of
\eqref{eq:isomorph}.

\begin{lemma}\label{lem-left-side}
 Consider the local times $L^v_{\tau(t)}$ as defined in Theorem~\ref{thm:rayknight}. For $v \in
\mathcal{L}$, define
\begin{equation*} \tilde{\mathcal{E}}_v = \left\{ 0 < L^v_\tau(t)  \leq  50^2 \cdot r^{2 s(f_v(h_v -1))}
m_v^{-3/2}\right\}\,.\end{equation*} Then, for any $t>0$
\[\P\left(\bigcup_{v\in \mathcal{L}} \tilde{\mathcal{E}}_v\right) \leq \frac{1}{16}\,.\]
\end{lemma}
\begin{proof}
Note that the random walk is at vertex $v_0$ at time $\tau(t)$.
Hence, given that $L^v_{\tau(t)}>0$, the random walk contains at
least one excursion which starts at $v$ and ends at $v_0$.
Therefore, given that $L^v_{\tau(t)}>0$, we see $c_v L^v_{\tau(t)}$
stochastically dominates the random variable \[L = \int_0^{T_{v_0}}
\one_{\{X_t = v\}}dt\,,\] where $X_t$ is a random walk on the
network started at $v$ and $T_{v_0}$ is the hitting time to $v_0$.

By definition, every time the random walk hits $v$, it takes an
exponential time for the walk to leave.  Also, the probability that
the random walk would hit $v_0$ before returning to $v$ can be
related to the effective resistance (see, for example, \cite{LP}).
Formally, when the random walk $W_t$ is at vertex $v$, it will wait
until the Poisson clock $\sigma$ with rate $1$ rings and then move
to a neighbor (possibly $v$ itself) selected proportional to the edge
conductance. Define
\[T_v^+ = \min\{t > \sigma: X_t = v\}\,.\] Then we have the
continuous-time version of \eqref{eq-walk-resistance},
\[\P_v(T_v^+ > T_{v_0}) = \frac{1}{c_v R_{\mathrm{eff}} (v, v_0)}\,.\]

By the strong Markov property, $L$ follows the law of the sum of a
geometric number of i.i.d. exponential variables. Thus $L$ follows
the law of an exponential variable with $\E L = c_v
R_{\mathrm{eff}}(v, v_0)$.

Recalling property \eqref{eq:treelike} of our separated tree $\mathcal{T}$, we see that
\[R_{\mathrm{eff}}(v, v_0) = \E(\eta_v - \eta_{v_0})^2 \geq 2^{-10} r^{2 s(f_v(h_v-1))-2}\,.\]
Thus,
\begin{align*}
\P(0<L^v_{\tau(t)} \leq  50^2 \cdot r^{2 s(f_v(h_v -1))} m_v^{-3/2})
& \leq \P(L \leq c_v \cdot  50^2 \cdot r^{2 s(f_v(h_v -1))}
m_v^{-3/2})\\
& \leq  \frac{ 50^2 \cdot r^{2 s(f_v(h_v -1))}m_v^{-3/2}}{
R_{\mathrm{eff}}(v,
v_0) } \\
& \leq 2^{11} \cdot 50^2 \cdot r^2\, m_v^{-3/2} \\
& \leq \frac{1}{16} m_v^{-1}\,,\end{align*}
where the last transition using \eqref{eq-deg} for $C$ chosen large enough,
and $m_v \geq \exp(C^2 r^2).$

Therefore, we
conclude that
\[\P\left(\bigcup_{v\in \mathcal{L}} \tilde{\mathcal{E}}_v\right) \leq \frac{1}{16} \sum_{v \in \mathcal{L}} m_v^{-1} = \frac{1}{16}\,,\]
where we used, from \eqref{eq:flow}, the fact that $\sum_{v \in \mathcal{L}} m_v^{-1} = 1$,
completing the proof.
\end{proof}

\remove{
We are now ready to establish the lower bound on the cover time.

\noindent{\bf Proof of the lower bound.} We denote by $\C$ the total
conductance of the network and by $R$ the diameter in resistance
metric. Let $t = \theta \mathcal{S}^2$ for $\theta = \epsilon^2 /2$,
where $\epsilon>0$ is as selected in Lemma~\ref{lem-coupling}. We
can assume that $\theta \mathcal{S}^2 \geq 96^2 R$, since otherwise
the lower bound on the cover time can be easily deduced from the
maximal hitting time.

We consider the associated Gaussian processes $\{\eta_u^L: u\in
\mathcal{T}\}$ and $\{\eta_u^R: u\in \mathcal{T}\}$ as defined in
Theorem~\ref{thm:rayknight}, where the superscript $L$ ($R$) denote
the process on the left (right) hand side of \eqref{eq:law}. For
$v\in \mathcal{L}$, write $w_v = 500^2 r^{2s(f_v(h_v-1))}
m_v^{-3/2}$. Combining Lemmas~\ref{lem-coupling} and
\ref{lem-percolation}, we obtain that
\[\P\left(\exists u\in \mathcal{L} \mbox{ such that } \frac{1}{2} (\eta_u^R + \sqrt{2t})^2 \leq w_v\right) \geq \frac{1}{8}\,.\]
It is an immediate corollary from Lemma~\ref{lem-left-side} that
\[\P\left( \exists u\in \mathcal{L} \mbox{ such that } L^u_{\tau(t)}>0\,,\, L^u_{\tau(t)} +
\frac{1}{2}(\eta_u^L)^2 \leq w_v \right) \leq \frac{1}{16}\,.\]
Altogether with Theorem~\ref{thm:rayknight}, we obtain that
\[\P(\exists u\in \mathcal{L} \mbox{ such that } L^u_{\tau(t)} = 0) \geq \frac{1}{16}\,,\]
and therefore \[\P(\tau_{\mathrm{cov}} \geq \tau(t)) \geq
\frac{1}{16}\,.\]}

\subsection{Additional applications}
\label{sec:coverapps}

We now prove a generalization of Theorem \ref{thm:pseudo}.
Suppose that $V = \{1,2,\ldots,n\}$,
and let $G(V)$ be a network with conductances $\{c_{ij}\}$.
We define real, symmetric $n \times n$ matrices $D$ and $A$ by
\begin{eqnarray*}
D_{ij} &=& \begin{cases}
c_i & i=j \\
0 & \textrm{otherwise.}
\end{cases} \\
A_{ij} &=& c_{ij}.
\end{eqnarray*}
We write
\begin{equation}\label{eq:laplace}
L_G = \frac{D-A}{\tr(D)}\,,
\end{equation}
and $L_G^+$ for the pseudoinverse of $L_G$.

\begin{theorem}\label{thm:remarkfull}
For any connected network $G(V)$,
$$
t_{\mathrm{cov}}(G) \asymp \E \,\left\|\sqrt{L_G^+}\, g\right\|_{\infty}^2,
$$
where $g=(g_1, \ldots, g_n)$ is a standard $n$-dimensional Gaussian.
\end{theorem}

\begin{proof}
If $\kappa$ denotes the commute time in $G$, then the following
formula is well-known (see, e.g. \cite{KR93}),
$$
\kappa(i,j) = \langle e_i - e_j, L_G^+ (e_i-e_j) \rangle,
$$
where $\{e_1, \ldots, e_n\}$ are the standard basis vectors in $\mathbb R^n$.
Using the fact that $L_G^+$ is self-adjoint and positive semi-definite, this yields
$$
\kappa(i,j) = \left\|\sqrt{L_G^+}\, e_i - \sqrt{L_G^+}\, e_j\right\|^2.
$$

Let $g=(g_1, \ldots, g_n) \in \mathbb R^n$ be a standard $n$-dimensional Gaussian,
and consider the Gaussian processes $\{\eta_i : i=1,\ldots,n\}$
where $\eta_i = \left\langle g, \sqrt{L_G^+}\, e_i \right\rangle$.
One verifies that for all $i,j \in V$,
$$
\E\, |\eta_i-\eta_j|^2 = \left\|\sqrt{L_G^+} (e_i-e_j)\right\|^2 = \kappa(i,j),
$$
thus by Theorem (MM),
\begin{equation}\label{eq:infinity}
\gamma_2(V,\sqrt{\kappa}) \asymp \E \max_{i \in V} \eta_i = \E \max_{i \in V} \left\langle g, \sqrt{L_G^+}\, e_i\right\rangle
=
\E \max_{i \in V} \left\langle \sqrt{L_G^+}\, g, e_i\right\rangle \asymp \E\,\left\|\sqrt{L_G^+}\, g\right\|_{\infty}.
\end{equation}
By Theorem \ref{thm:mainthm}, $\left[\gamma_2(V,\sqrt{\kappa})\right]^2 \asymp t_{\mathrm{cov}}(G)$.
Finally, one can use Lemma \ref{lem-gaussian-concentration} to conclude that $$\left(\E\,\left\|\sqrt{L_G^+}\, g\right\|_{\infty}\right)^2 \asymp
\E\,\left\|\sqrt{L_G^+}\, g\right\|^2_{\infty},$$ completing the proof.
\end{proof}

\begin{theorem}
There a randomized algorithm which, given
any connected network $G(V)$, with $m = |\{(x,y) : c_{xy} \neq 0\}|$,
runs in time $O(m (\log m)^{O(1)})$ and outputs a number $A(G)$ such that
$t_{\mathrm{cov}}(G) \asymp \mathbb E\left[A(G)\right] \asymp (\E\left[A(G)^2\right])^{1/2}$.
\end{theorem}

\begin{proof}
In \cite[\S 4]{SS08}, it is shown how to
compute a $k \times n$ matrix $Z$,
in expected time $O(m (\log m)^{O(1)})$,
with $k = O(\log n)$, and such that
for every $i,j \in V$,
\begin{equation}\label{eq:Zprocess}
\kappa(i,j) \leq \|Z(e_i-e_j)\|^2 \leq 2 \kappa(i,j).
\end{equation}
We can associate the Gaussian processes $\{\eta_i\}_{i \in V}$,
where $\eta_i = \langle g, Z e_i \rangle$, and $g$ is
a standard $k$-dimensional Gaussian.
Letting $d(i,j)=\sqrt{\E\,|\eta_i-\eta_j|^2}$,
we see from \eqref{eq:Zprocess}
that $\sqrt{\kappa} \leq d \leq \sqrt{2\kappa}$,
therefore $\gamma_2(V,\sqrt{\kappa}) \asymp \gamma_2(V, d)$.
It follows (see \eqref{eq:infinity}) that
$$
\E \left\|Z g\right\|_{\infty}^2 \asymp \E\,\left\|\sqrt{L_G^+}\, g\right\|^2_{\infty} \asymp t_{\mathrm{cov}}(G),
$$
where the last equivalence is the content of Theorem \ref{thm:remarkfull}.

The output of our algorithm is thus $A(G) = \|Z g\|^2_{\infty}$,
where $g$ is a standard $k$-dimensional Gaussian vector.
The fact that $\E[A(G)] \asymp (\E[A(G)^2])^{1/2}$
follows from Lemma \ref{lem-gaussian-concentration}.
\end{proof}

\section{Open problems and further discussion}

We now present two open questions that arise naturally
from the present work.
The first question concerns obtaining a better deterministic approximation
to the cover time.
\begin{question}
Is there, for any
$\epsilon > 0$, a deterministic, polynomial-time algorithm that approximates $t_{\mathrm{cov}}(G)$ up to a $(1+\epsilon)$ factor?
\end{question}
Note that the preceding question has been solved by Feige and Zeitouni \cite{FZ09} in the case of trees.

\medskip

The second question involves concentration of $\tau_{\mathrm{cov}}$ around its expected value.
Under the assumption that
$\lim_{n \to \infty} \frac{t_{\mathrm{cov}}(G_n)}{t_{\mathrm{hit}}(G_n)} = \infty$,
where $t_{\mathrm{hit}}$ denotes the maximal hitting time, Aldous
\cite{Aldous91b} proves that
$\frac{\tau_{\mathrm{cov}}(G_n)}{t_{\mathrm{cov}}(G_n)}$ converges
to 1 in probability.   We ask whether it is possible to
obtain sharper concentration.
\begin{question}
Is the standard deviation of $\tau_{\mathrm{cov}}$ bounded by the
maximal hitting time $t_{\mathrm{hit}}$? Furthermore, does
$\frac{\tau_{\mathrm{cov}} - t_{\mathrm{cov}}}{t_{\mathrm{hit}}}$
exhibit an exponential decay with constant rate?
\end{question}

It is interesting to consider the extent to which Theorem \ref{thm:tight}
is sharp.
Consider a family of graphs $\{G_n\}$.
We point out that the asymptotic formula,
\begin{equation}\label{eq-asympt}
 t_{\mathrm{cov}} (G_n) \sim |E(G_n)| \cdot \left(\E \sup_{v\in V} \eta_v\right)^2\,,
 \end{equation}
holds for both the family of complete graphs and
the family of regular trees, where
we write $a_n \sim b_n$ for $\lim a_n/b_n = 1$,
and $E(G_n)$ denotes the set of edges in $G_n$.
Here, $\{\eta_v\}$ is the GFF associated to $G_n$
with $\eta_{v_0}=0$ for some fixed vertex $v_0$.

To see this,
note that the GFF on the $n$-vertex complete graph satisfies $\var \eta_v
= \frac{2}{n}$ and $\E (\eta_v \eta_u) = \frac{1}{n}$ for $v_0 \notin \{u,v\}$.
Therefore, we can write $\eta_v = \xi + \xi_v$ for every $v\neq
v_0$, where $\xi$ and all $\{\xi_v\}_{v \in V}$ are i.i.d.\ Gaussian variables with
variance $\frac{1}{n}$. It is now clear that $\E\sup_v \eta_v \sim
\sqrt{2\log n/n} $.
  Combined with the facts that $t_{\mathrm{cov}}(G_n) \sim n \log n$ and $|E(G_n)| = \frac{n(n-1)}{2}$,
  this confirms \eqref{eq-asympt} for complete graphs.

Fix $b\geq 2$  and consider a regular $b$-ary tree $T_m$ of height
$m$ with $n= \frac{b^{m+1} - 1}{b-1}$ vertices. It is shown in \cite{Aldous91} that $t_{\mathrm{cov}}(T_m)
\sim 2m n \log n  $. On the other hand, Biggins
\cite{Biggins77} proved that the corresponding GFF satisfies $\E
\sup_v \eta_v \sim \sqrt{2 m\log n}\,  $. Since the number of
edges in $T_m$ is $n -1 $, we
infer that \eqref{eq-asympt} holds for regular trees.
It is clearly very interesting to understand
the generality under which \eqref{eq-asympt} holds.


\subsection*{Acknowledgements}

We are grateful to Martin Barlow and Asaf Nachmias for helpful
discussions in the early stages of this work. We thank Jay Rosen and
an anonymous referee for a very thorough reading of the manuscript,
along with numerous insightful comments. We also thank Nike Sun,
Russ Lyons, Saran Ahuja, and Yoshihiro Abe for useful comments.

\bibliography{covertime}

\begin{thebibliography}{10}

\bibitem{Aldous89}
D.~Aldous.
\newblock {\em Probability approximations via the {P}oisson clumping
  heuristic}, volume~77 of {\em Applied Mathematical Sciences}.
\newblock Springer-Verlag, New York, 1989.

\bibitem{AF}
D.~Aldous and J.~Fill.
\newblock {\em Reversible Markov Chains and Random Walks on Graphs}.
\newblock In preparation, available at {\tt
  http://www.stat.berkeley.edu/~aldous/RWG/book.html}.

\bibitem{Aldous82}
D.~J. Aldous.
\newblock Markov chains with almost exponential hitting times.
\newblock {\em Stochastic Process. Appl.}, 13(3):305--310, 1982.

\bibitem{Aldous91}
D.~J. Aldous.
\newblock Random walk covering of some special trees.
\newblock {\em J. Math. Anal. Appl.}, 157(1):271--283, 1991.

\bibitem{Aldous91b}
D.~J. Aldous.
\newblock Threshold limits for cover times.
\newblock {\em J. Theoret. Probab.}, 4(1):197--211, 1991.

\bibitem{AKLLR79}
R.~Aleliunas, R.~M. Karp, R.~J. Lipton, L.~Lov{\'a}sz, and C.~Rackoff.
\newblock Random walks, universal traversal sequences, and the complexity of
  maze problems.
\newblock In {\em 20th {A}nnual {S}ymposium on {F}oundations of {C}omputer
  {S}cience ({S}an {J}uan, {P}uerto {R}ico, 1979)}, pages 218--223. IEEE, New
  York, 1979.

\bibitem{BDNP09}
M.~T. Barlow, J.~Ding, A.~Nachmias, and Y.~Peres.
\newblock The evolution of the cover time.
\newblock Preprint, available at {\tt http://arxiv.org/abs/1001.0609}.

\bibitem{Biggins77}
J.~D. Biggins.
\newblock Chernoff's theorem in the branching random walk.
\newblock {\em J. Appl. Probability}, 14(3):630--636, 1977.

\bibitem{BK89}
A.~Z. Broder and A.~R. Karlin.
\newblock Bounds on the cover time.
\newblock {\em J. Theoret. Probab.}, 2(1):101--120, 1989.

\bibitem{Campbell}
G.~A. Campbell.
\newblock Cisoidal oscillations.
\newblock {\em Trans. Amer. Inst. Elec. Engrs.}, (30), 1911.

\bibitem{CRRST96}
A.~K. Chandra, P.~Raghavan, W.~L. Ruzzo, R.~Smolensky, and P.~Tiwari.
\newblock The electrical resistance of a graph captures its commute and cover
  times.
\newblock {\em Comput. Complexity}, 6(4):312--340, 1996/97.

\bibitem{CF08}
C.~Cooper and A.~Frieze.
\newblock The cover time of the giant component of a random graph.
\newblock {\em Random Structures Algorithms}, 32(4):401--439, 2008.

\bibitem{CW90}
D.~Coppersmith and S.~Winograd.
\newblock Matrix multiplication via arithmetic progressions.
\newblock {\em J. Symbolic Comput.}, 9(3):251--280, 1990.

\bibitem{DPRZ04}
A.~Dembo, Y.~Peres, J.~Rosen, and O.~Zeitouni.
\newblock Cover times for {B}rownian motion and random walks in two dimensions.
\newblock {\em Ann. of Math. (2)}, 160(2):433--464, 2004.

\bibitem{DS84}
P.~G. Doyle and J.~L. Snell.
\newblock {\em Random walks and electric networks}, volume~22 of {\em Carus
  Mathematical Monographs}.
\newblock Mathematical Association of America, Washington, DC, 1984.

\bibitem{Dudley67}
R.~M. Dudley.
\newblock The sizes of compact subsets of {H}ilbert space and continuity of
  {G}aussian processes.
\newblock {\em J. Functional Analysis}, 1:290--330, 1967.

\bibitem{Dynkin84}
E.~B. Dynkin.
\newblock Gaussian and non-{G}aussian random fields associated with {M}arkov
  processes.
\newblock {\em J. Funct. Anal.}, 55(3):344--376, 1984.

\bibitem{Dynkin83}
E.~B. Dynkin.
\newblock Local times and quantum fields.
\newblock In {\em Seminar on stochastic processes, 1983 ({G}ainesville, {F}la.,
  1983)}, volume~7 of {\em Progr. Probab. Statist.}, pages 69--83. Birkh\"auser
  Boston, Boston, MA, 1984.

\bibitem{Eisenbaum95}
N.~Eisenbaum.
\newblock Une version sans conditionnement du th\'eor\`eme d'isomorphisms de
  {D}ynkin.
\newblock In {\em S\'eminaire de {P}robabilit\'es, {XXIX}}, volume 1613 of {\em
  Lecture Notes in Math.}, pages 266--289. Springer, Berlin, 1995.

\bibitem{EKMRS00}
N.~Eisenbaum, H.~Kaspi, M.~B. Marcus, J.~Rosen, and Z.~Shi.
\newblock A {R}ay-{K}night theorem for symmetric {M}arkov processes.
\newblock {\em Ann. Probab.}, 28(4):1781--1796, 2000.

\bibitem{Feige95b}
U.~Feige.
\newblock A tight lower bound on the cover time for random walks on graphs.
\newblock {\em Random Structures Algorithms}, 6(4):433--438, 1995.

\bibitem{Feige95a}
U.~Feige.
\newblock A tight upper bound on the cover time for random walks on graphs.
\newblock {\em Random Structures Algorithms}, 6(1):51--54, 1995.

\bibitem{FZ09}
U.~Feige and O.~Zeitouni.
\newblock Deterministic approximation for the cover time of trees.
\newblock Preprint, available at {\tt
  http://arxiv1.library.cornell.edu/abs/0909.2005},.

\bibitem{F1}
X.~Fernique.
\newblock R\'egularit\'e de processus gaussiens.
\newblock {\em Invent. Math.}, 12:304--320, 1971.

\bibitem{F2}
X.~Fernique.
\newblock Regularit\'e des trajectoires des fonctions al\'eatoires gaussiennes.
\newblock In {\em \'{E}cole d'\'{E}t\'e de {P}robabilit\'es de {S}aint-{F}lour,
  {IV}-1974}, pages 1--96. Lecture Notes in Math., Vol. 480. Springer, Berlin,
  1975.

\bibitem{Foster48}
R.~M. Foster.
\newblock The average impedance of an electrical network.
\newblock In {\em Reissner {A}nniversary {V}olume, {C}ontributions to {A}pplied
  {M}echanics}, pages 333--340. J. W. Edwards, Ann Arbor, Michigan, 1948.

\bibitem{GZ03}
O.~Gu{\'e}don and A.~Zvavitch.
\newblock Supremum of a process in terms of trees.
\newblock In {\em Geometric aspects of functional analysis}, volume 1807 of
  {\em Lecture Notes in Math.}, pages 136--147. Springer, Berlin, 2003.

\bibitem{Janson97}
S.~Janson.
\newblock {\em Gaussian {H}ilbert spaces}, volume 129 of {\em Cambridge Tracts
  in Mathematics}.
\newblock Cambridge University Press, Cambridge, 1997.

\bibitem{JS00}
J.~Jonasson and O.~Schramm.
\newblock On the cover time of planar graphs.
\newblock {\em Electron. Comm. Probab.}, 5:85--90 (electronic), 2000.

\bibitem{KKLV00}
J.~Kahn, J.~H. Kim, L.~Lov{\'a}sz, and V.~H. Vu.
\newblock The cover time, the blanket time, and the {M}atthews bound.
\newblock In {\em 41st {A}nnual {S}ymposium on {F}oundations of {C}omputer
  {S}cience ({R}edondo {B}each, {CA}, 2000)}, pages 467--475. IEEE Comput. Soc.
  Press, Los Alamitos, CA, 2000.

\bibitem{KLNS89}
J.~D. Kahn, N.~Linial, N.~Nisan, and M.~E. Saks.
\newblock On the cover time of random walks on graphs.
\newblock {\em J. Theoret. Probab.}, 2(1):121--128, 1989.

\bibitem{KR93}
D.~J. Klein and M.~Randi{\'c}.
\newblock Resistance distance.
\newblock {\em J. Math. Chem.}, 12(1-4):81--95, 1993.
\newblock Applied graph theory and discrete mathematics in chemistry
  (Saskatoon, SK, 1991).

\bibitem{Knight63}
F.~B. Knight.
\newblock Random walks and a sojourn density process of {B}rownian motion.
\newblock {\em Trans. Amer. Math. Soc.}, 109:56--86, 1963.

\bibitem{Ledoux89}
M.~Ledoux.
\newblock {\em The concentration of measure phenomenon}, volume~89 of {\em
  Mathematical Surveys and Monographs}.
\newblock American Mathematical Society, Providence, RI, 2001.

\bibitem{LT91}
M.~Ledoux and M.~Talagrand.
\newblock {\em Probability in {B}anach spaces}, volume~23 of {\em Ergebnisse
  der Mathematik und ihrer Grenzgebiete (3) [Results in Mathematics and Related
  Areas (3)]}.
\newblock Springer-Verlag, Berlin, 1991.
\newblock Isoperimetry and processes.

\bibitem{LPW09}
D.~A. Levin, Y.~Peres, and E.~L. Wilmer.
\newblock {\em Markov chains and mixing times}.
\newblock American Mathematical Society, Providence, RI, 2009.
\newblock With a chapter by James G. Propp and David B. Wilson.

\bibitem{Lov96}
L.~Lov{\'a}sz.
\newblock Random walks on graphs: a survey.
\newblock In {\em Combinatorics, {P}aul {E}rd{\H o}s is eighty, {V}ol.\ 2
  ({K}eszthely, 1993)}, volume~2 of {\em Bolyai Soc. Math. Stud.}, pages
  353--397. J\'anos Bolyai Math. Soc., Budapest, 1996.

\bibitem{Lyons92}
R.~Lyons.
\newblock Random walks, capacity and percolation on trees.
\newblock {\em Ann. Probab.}, 20(4):2043--2088, 1992.

\bibitem{LP}
R.~{Lyons, with Y. Peres}.
\newblock {\it Probability on Trees and Networks}.
\newblock In preparation. Current version available at \texttt{
  http://mypage.iu.edu/\~{}rdlyons/prbtree/book.pdf}, 2009.

\bibitem{MR92}
M.~B. Marcus and J.~Rosen.
\newblock Sample path properties of the local times of strongly symmetric
  {M}arkov processes via {G}aussian processes.
\newblock {\em Ann. Probab.}, 20(4):1603--1684, 1992.

\bibitem{MR01}
M.~B. Marcus and J.~Rosen.
\newblock Gaussian processes and local times of symmetric {L}\'evy processes.
\newblock In {\em L\'evy processes}, pages 67--88. Birkh\"auser Boston, Boston,
  MA, 2001.

\bibitem{MR06}
M.~B. Marcus and J.~Rosen.
\newblock {\em Markov processes, {G}aussian processes, and local times}, volume
  100 of {\em Cambridge Studies in Advanced Mathematics}.
\newblock Cambridge University Press, Cambridge, 2006.

\bibitem{Matthews88}
P.~Matthews.
\newblock Covering problems for {M}arkov chains.
\newblock {\em Ann. Probab.}, 16(3):1215--1228, 1988.

\bibitem{Ray63}
D.~Ray.
\newblock Sojourn times of diffusion processes.
\newblock {\em Illinois J. Math.}, 7:615--630, 1963.

\bibitem{SpielmanICM}
D.~Spielman.
\newblock Algorithms, graph theory, and linear equations in {L}aplacian
  matrices.
\newblock To appear, {\em Proceedings of the International Congrees of
  Mathematicians,} Hyderabad, India, 2010.

\bibitem{SS08}
D.~Spielman and N.~Srivastava.
\newblock Graph sparsification by effective resistances.
\newblock Available at \verb|http://arxiv.org/abs/0803.0929|, 2008.

\bibitem{ST06}
D.~Spielman and S.-H. Teng.
\newblock Nearly-linear time algorithms for preconditioning and solving
  symmetric, diagonally dominant linear systems.
\newblock Available at \verb|http://arxiv.org/abs/cs.NA/0607105|, 2006.

\bibitem{Talagrand87}
M.~Talagrand.
\newblock Regularity of {G}aussian processes.
\newblock {\em Acta Math.}, 159(1-2):99--149, 1987.

\bibitem{Talagrand95a}
M.~Talagrand.
\newblock Embedding subspaces of {$L\sb p$} in {$l\sp N\sb p$}.
\newblock In {\em Geometric aspects of functional analysis ({I}srael,
  1992--1994)}, volume~77 of {\em Oper. Theory Adv. Appl.}, pages 311--325.
  Birkh\"auser, Basel, 1995.

\bibitem{Talagrand96}
M.~Talagrand.
\newblock Majorizing measures: the generic chaining.
\newblock {\em Ann. Probab.}, 24(3):1049--1103, 1996.

\bibitem{Tala01}
M.~Talagrand.
\newblock Majorizing measures without measures.
\newblock {\em Ann. Probab.}, 29(1):411--417, 2001.

\bibitem{Talagrand05}
M.~Talagrand.
\newblock {\em The generic chaining}.
\newblock Springer Monographs in Mathematics. Springer-Verlag, Berlin, 2005.
\newblock Upper and lower bounds of stochastic processes.

\bibitem{Tetali91}
P.~Tetali.
\newblock Random walks and the effective resistance of networks.
\newblock {\em J. Theoret. Probab.}, 4(1):101--109, 1991.

\bibitem{WZ96}
P.~Winkler and D.~Zuckerman.
\newblock Multiple cover time.
\newblock {\em Random Structures Algorithms}, 9(4):403--411, 1996.

\bibitem{Zuckerman92}
D.~Zuckerman.
\newblock A technique for lower bounding the cover time.
\newblock {\em SIAM J. Discrete Math.}, 5(1):81--87, 1992.

\end{thebibliography}
\bibliographystyle{abbrv}

\end{document}